\documentclass[11pt,bull-l]{amsart}
\usepackage[top = 0.9in, bottom = 0.9in, left =1in, right = 1in]{geometry}
\usepackage{packages}
\usepackage{macros}
\usepackage{setspace}

\usepackage{cleveref}

\usepackage{amsmath}
\usepackage{algorithm}
\usepackage{algpseudocode}
\usepackage{mathabx}

\usepackage{todonotes}

\newtheorem*{theorem*}{Theorem}
\newtheorem{theorem}{Theorem}[section]
\newtheorem{definition}[theorem]{Definition}
\newtheorem{corollary}[theorem]{Corollary}
\newtheorem{lemma}[theorem]{Lemma}
\newtheorem{assumption}{Assumption}
\newtheorem{remark}[theorem]{Remark}

\newtheorem*{customplm}{Spike Detection}

\renewcommand{\vec}{\mathbf}
\newcommand{\ri}{}\renewcommand{\ri}{\mathrm{i}}
\newcommand{\dMP}{\mathrm{dMP}}
\newcommand{\OO}{\mathrm{O}}
\newcommand{\bunderline}[1]{\mkern2mu\underline{\mkern-2mu#1\mkern-4mu}\mkern4mu }

\theoremstyle{definition}
\newtheorem{simulation}{Simulation}
\newtheorem{example}{Example}

\title{A Lanczos-Based Algorithmic Approach for Spike Detection in Large Sample Covariance Matrices}

\author{Charbel Abi Younes}
\address{University of Washington, Seattle, WA}
\email{cyounes@uw.edu}

\author{Xiucai Ding}
\address{University of California, Davis, Davis, CA}
\email{xcading@ucdavis.edu}

\author{Thomas Trogdon}
\address{University of Washington, Seattle, WA}
\email{trogdon@uw.edu}

\date{}
\usepackage{hyperref}
\hypersetup{
  colorlinks   = true, 
  urlcolor     = red, 
  linkcolor    = blue, 
  citecolor   = blue 
}
\usepackage{booktabs}
\usepackage{siunitx}

\begin{document}

\pagestyle{plain}

\maketitle

\begin{abstract}
We introduce a new approach for estimating the number of spikes in a general class of spiked covariance models without directly computing the eigenvalues of the sample covariance matrix. This approach is based on the Lanczos algorithm and the asymptotic properties of the associated Jacobi matrix and its Cholesky factorization. A key aspect of the analysis is interpreting the eigenvector spectral distribution as a perturbation of its asymptotic counterpart.  The specific exponential-type asymptotics of the Jacobi matrix enables an efficient approximation of the Stieltjes transform of the asymptotic spectral distribution via a finite continued fraction. As a consequence, we also obtain estimates for the density of the asymptotic distribution and the location of outliers. We provide consistency guarantees for our proposed estimators, proving their convergence in the high-dimensional regime. We demonstrate that, when applied to standard spiked covariance models, our approach outperforms existing methods in computational efficiency and runtime, while still maintaining robustness to exotic population covariances.

\end{abstract}

\section{Introduction}\label{sec:Overview}
Large sample covariance matrices play a fundamental role in high-dimensional data analysis. A widely studied and sophisticated model in the literature is the spiked covariance matrix model \cite{Ding2021c,Johnstone2001}, in which a finite number of spikes—eigenvalues that separate from the bulk of the spectrum—are introduced into the population covariance matrix $\Sigma$. In this model, a key task is to estimate the number of spikes using the associated sample covariance matrices.
 Such a problem has been studied in various settings; see Section \ref{sec_existingresults} for a detailed review. However, all existing works leverage the spectral properties of $W$, relying on specific statistics constructed from its eigenvalues—either the prominent outlier eigenvalues or the relatively small bulk eigenvalues. Despite the potential statistical consistency of these approaches under certain conditions, they can be computationally intensive and prone to numerical errors due to the direct calculation of eigenvalues\footnote{Modern eigensolvers are very accurate for symmetric matrices. The errors here stem from using individual eigevalues that may fluctuate wildly, while averages of eigenvalues have much smaller variance.} and associated Monte Carlo estimations.

In this paper, instead of relying on eigenvalue-based statistics, we introduce a novel algorithmic approach built upon the well-known Lanczos iteration\footnote{Due to instabilities in the iteration we use reorthogonalization \cite{Paige1980}.} \cite{Lanczos1950}, which operates directly on the sample covariance matrix $W$ using random vectors, uniformly distributed on the hypersphere. Our proposed methods are both statistically consistent and numerically cheap while still being robust. Furthermore, as a byproduct of determining the number of spikes, we also obtain estimates for the asymptotic density of the empirical spectral distribution of $W$ and the asymptotic locations of the outliers in $W$. 


Our approach is based on the observation that, under mild assumptions, the number of spikes is, with high probability, identical to the number of poles of a complex function—specifically, the Stieltjes transform of a measure that connects the empirical spectral distribution (ESD) and its asymptotic counterpart, the asymptotic spectral distribution (ASD). This stochastic measure, which we refer to as the spiked ASD (see \eqref{eq:hatmu0} and \eqref{eq:hatmu} below), preserves a scaled version of the continuous density of the ASD while accounting for the random outliers in the ESD. Consequently, a key challenge is to develop an efficient and robust estimator for the spiked ASD. Rather than directly relying on the eigenvalues of $W$, the central innovation of our method is the use of Cholesky factorization for the Jacobi matrix obtained by applying the Lanczos iteration with $W$ to several test vectors. This approach is motivated by the fact that the Stieltjes transform of the spiked ASD of $W$ can be approximately characterized by a new fixed-point equation, based on continued fraction theory, whose coefficients are derived from the entries of the Cholesky factorization of the Jacobi matrix associated with the spiked ASD (see Section \ref{sec_newnewnew}). This fixed-point equation differs from the commonly-used self-consistent equation in random matrix theory \cite{Bai2010,Knowles2017}. Crucially, this equation, along with the Lanczos algorithm, enables a numerical estimation of the Stieltjes transform of the ASD through an iterative process without directly computing the eigenvalues of $W$ (see Algorithm \ref{Ea:realVESDB} and Section \ref{sec_newnewnew}). 

On the techinical side, our proofs rely on three main ideas. The first one is the fact that the orthogonal polynomials for a generic limiting random matrix spectral distribution exhibit exponentially small error terms in their asymptotics \cite{Kuijlaars2003,Yattselev2015,Ding2021b}. These results are based on the analysis of the Fokas--Its--Kitaev Riemann--Hilbert characterization of orthogonal polynomials \cite{FokasItsKitaev}. The second is the so-called anisotropic local laws for covariance matrices \cite{Knowles2017,Pillai2014} and their extension to spiked models \cite{Ding2019}. The last technique is a recently developed perturbation theory \cite{Ding2021b} for orthogonal polynomials, which builds on the Fokas--Its--Kitaev Riemann--Hilbert theory and is particularly compatible with the local laws from random matrix theory.



\begin{figure}[tbp]
    \centering
    \begin{subfigure}[b]{0.48\textwidth}
        \centering
        \includegraphics[width=\textwidth]{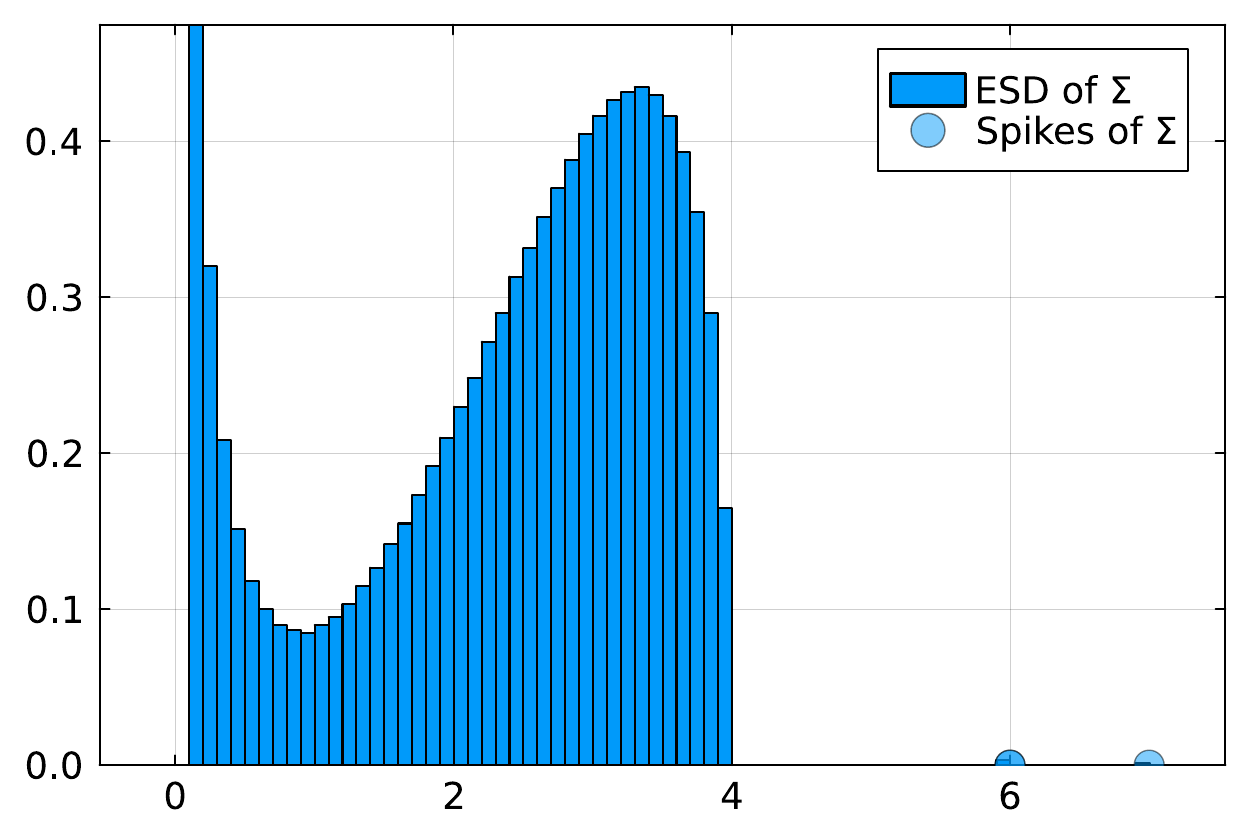}
    \end{subfigure}
    \hfill
    \begin{subfigure}[b]{0.48\textwidth}
        \centering
        \includegraphics[width=\textwidth]{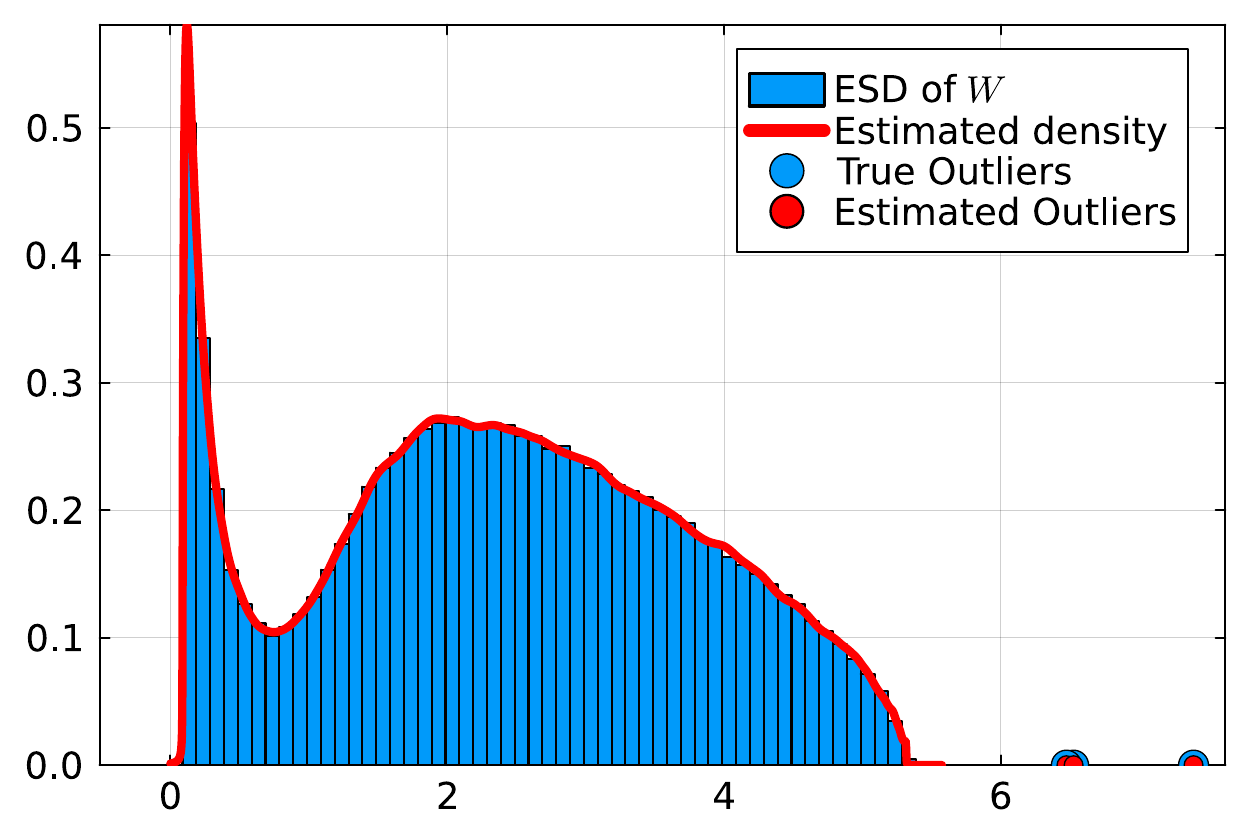}
    \end{subfigure}
    \caption{Left: The ESD of $\Sigma$, where $\Sigma$ is an $N \times N$ diagonal matrix with $N = 6000$ and entries given by the quantiles of the density $\frac{1}{K}\frac{x^4+1}{x^2} \sqrt{x - 0.1} \sqrt{4 - x}$ with $K$ being a normalizing constant. The first three diagonal entries are modified to 7, 6, and 6, forming the spikes of $\Sigma$. Right: The ESD of $W$ defined in \eqref{eq:SCM_Model}, with $c_N = 0.1$, $M = 60000$, and $X$ having iid standard normal entries. The ESD is compared against the estimated outliers, their locations, and the approximate ASD obtained using Algorithms~\ref{Ea:ESD} and~\ref{finaldetectionalgorithm} with parameters $k = 150$ and $C = 1$.}
    \label{fig:Demo}
\end{figure}


\subsection{The model and problem formulation}

Consider the sample covariance matrix
\begin{align}\label{eq:SCM_Model}
    W = Y Y^*, \quad Y = \Sigma^{1/2}X,
\end{align}
where $\Sigma$ is a positive-definite deterministic matrix, referred to as the population covariance matrix, and $X$ is an $N \times M$ random matrix with centered, independent, and identically distributed (iid) entries. In this paper, we are interested in the high-dimensional setting, where $M$ is comparably large to $N$. Specifically, we assume the existence of a small constant $0 < \tau < 1$ such that the aspect ratio $c_N:= N/M$ satisfies
\begin{align}\label{eq:AsymptRatio}
    \tau \leq c_N \leq \tau^{-1} \quad \text{for all }N.
\end{align}
Moreover, we assume that the entries of $X$, denoted by $x_{ij}$ for $1\leq i\leq N, 1\leq j\leq M$, satisfy
\begin{align}\label{eq:MomentCondition}
    \E x_{ij} = 0, ~\E |x^2_{ij}| = \frac{1}{M}.
\end{align}
For definiteness, we focus on the case where the random variables $x_{ij}$ are real. However, we note that our discussions and numerical methods can be extended to the complex case after minor adjustments if we further require that $\re x_{ij}$ and $\im x_{ij}$ are independent centered random variables with variance $(2M)^{-1}$. We also assume that the random variables $x_{ij}$ possess arbitrarily high moments, meaning that for any fixed $k\in \N$, there exists a constant $C_k$ such that
\begin{align}\label{eq:BddMoments}
   \paren{\E|x_{ij}^k|}^{1/k} \leq C_k M^{-1/2}.
\end{align}
The assumption that \eqref{eq:BddMoments} holds for all $k \in \N$ may be easily relaxed. For instance, it is easy to check that our results remain valid, after minor modifications using some truncation and comparison techniques, if we only require that \eqref{eq:BddMoments} holds for all $k\leq C$ for some finite $C$.

We adopt the spiked covariance matrix model following the framework established in \cite{Ding2021}. Let the spiked covariance matrix $\Sigma$ admit the following spectral decomposition
\begin{align}\label{eq:SigmaEigen}
    \Sigma = \sum_{i=1}^N \tilde{\sigma}_i\mathbf{v}_i\mathbf{v}_i^*, ~ \tilde{\sigma}_i = (1+d_i)\sigma_i,
\end{align}
where $\sigma_1\geq\sigma_2\geq \dots \geq \sigma_N >0$, $d_1\geq d_2\geq \dots \geq d_N$.   Further, we assume for $r > 0$, fixed,
\begin{align*}
    d_i>0,~ i\leq r; ~ d_i = 0, i>r.
\end{align*}
We assume that $\sigma_1, \ldots, \sigma_N$ are sufficiently regular and $d_r$ exceeds a certain threshold (see Assumption~\ref{as:TechA}), ensuring that the top $r$ eigenvalues of $\Sigma$, referred to as \emph{spikes in the population covariance matrix}, lead to \emph{outliers among the eigenvalues of the sample covariance matrix} $W$. 
To fix more notation, we express the spectral decomposition of $W$ as
\begin{align}\label{eq_spectralW}
    W = \sum_{i=1}^N \lambda_i\vec u_i \vec u_i^*, \quad \lambda_i = \lambda_i(W), \quad \lambda_i \geq \lambda_{i+1}.
\end{align}
The \emph{ESD} for $W$ is given by
\begin{align*}
    \mu_{\mathrm{ESD}} = \mu_W :=  \frac{1}{N} \sum_{j=1}^N \delta_{\lambda_j(W)}.
\end{align*}
With our assumptions, the ESD has a deterministic approximation, the \emph{ASD}, $\mu_{\mathrm{ASD}}$, such that $\mu_{\mathrm{ESD}} - \mu_{\mathrm{ASD}}$ tends weakly to zero.  
For a unit vector $\vec b$, the spiked ASD is given by
\begin{align}\label{eq:sASD}
    \mu_{\mathrm{sASD}} = \mu_{\mathrm{ASD}} + \sum_{j=1}^r| \vec u_j^* \vec b |^2 \delta_{\lambda_j(W)}.
\end{align}
It is for this measure that we construct an approximation.

The spiked sample covariance matrix $W$ is strongly connected to its non-spiked counterpart, given by
\begin{align}\label{eq:NonSpikedModel}
    W_0 = \Sigma_0^{1/2}XX^*\Sigma_0^{1/2},
\end{align}
where $\Sigma_0$ is represented by the spectral decomposition
\begin{align}\label{eq:Sigma0Eigen}
    \Sigma_0 = \sum_{i=1}^N \sigma_i \mathbf{v}_i \mathbf{v}_i^*.
\end{align}
We differentiate $\Sigma_0$ from $\Sigma$ because if a limit is desired for certain spectral statistics of \eqref{eq:SCM_Model}, then additional assumptions will need to be imposed on $\Sigma_0$; see Assumption \ref{as:TechA} for more details. As a historical note, we recall that if $\sigma_i\equiv 1$ for $1\leq i\leq N$, it is well known that the eigenvalues of $W_0$ obey the Marchenko-Pastur law  \cite{Marcenko1967}. For general $\Sigma_0$, they are governed by the so-called \textit{deformed Marchenko-Pastur law} \cite{Bai2010,Knowles2017}, which, in turn, determines $\mu_{\mathrm{ASD}}$.

A common challenge in statistics is the estimation of the number of spikes $r$ and the asymptotic locations of the outliers $\lambda_{1}, \dots, \lambda_r$ from a single sample of $W$ in the high-dimensional regime. Due to the large matrix dimensions, eigenvalue computations can become computationally expensive, and the matrix $Y$ is typically accessible only through matrix-vector products. Clearly, the matrix $Y$ can be fully accessed and constructed using $M$ matrix-vector products. Therefore, the objective is to provide a solution to the following problem with a minimal number of products, specifically with $n \ll M$.


\begin{customplm}\label{Prob:SpikeNbr}
    Given the sample covariance matrix $W$ defined in \eqref{eq:SCM_Model}, find a computationally efficient and robust estimator, $\widehat{r}$, for the number of spikes, $r,$ that is consistent in the high-dimensional regime as $N\to \infty$, i.e.,
    \begin{align*}
        \Prb (\widehat{r} = r) \to 1, \quad \text{as }N\to \infty,
    \end{align*}
    and can be computed using only $n \ll N$ matrix-vector products.
\end{customplm}
\newcommand{\MainProblem}{\hyperref[Prob:SpikeNbr]{Spike Detection}}

\subsection{An overview of our approach}\label{sec:overviewofourapparoach}
As previously discussed, our method relies on understanding and approximating the Stieltjes transform of the spiked ASD. Rather than approximating the ESD, or the spiked ASD, using the eigenvalues of $W$, our approach employs Lanczos iterations for $(W, \mathbf{b})$ with multiple vectors $\mathbf{b}$ sampled independently and uniformly from the hypersphere. Technically,  a central component of our algorithmic approach and its analysis is the eigenvector empirical spectral distribution (VESD) and its asymptotic properties. Given a sample covariance matrix $W$ and a vector $\vec b\in \R^N$, the VESD is defined as 
\begin{align}\label{eq:VESDdef}
    \mu_{W,\mathbf{b}} := \sum_{i=1}^N \abs{ \mathbf{u}_i^*\mathbf{b}}^2 \delta_{\lambda_{i}(W)}.
\end{align}
We also refer to this as the VESD at $\vec b$. 

The use of the VESD is motivated by several factors. On one hand, the VESD serves as an unbiased, and asymptotically consistent, estimator of the ESD when $\mathbf{b}$ is uniformly distributed on the hypersphere. This allows for an efficient approximation of the ESD by averaging the VESD over a sufficiently large number of sampled vectors $\mathbf{b}$ --- or just one if $N$ is sufficiently large. On the other hand, the Stieltjes transform of the VESD approximately satisfies a novel fixed-point equation and can be computed iteratively using the Lanczos-produced Jacobi matrix\footnote{A Jacobi matrix is a finite or semi-infinite symmetric tridiagonal matrix with positive off-diagonal entries.} and its Cholesky decomposition (see Algorithm \ref{Ea:realVESDB} and Section \ref{sec_newnewnew}).

Leveraging the VESD alongside local laws from random matrix theory provides insight into the Jacobi matrix corresponding to its deterministic counterpart. First, the eigenvector asymptotic spectral distribution (VASD) for $\mu_{W,\vec b}$, denoted by $\mu_{\vec b} = \mu_{\vec b}(\Sigma)$, is described by the anisotropic local laws (c.f. Section \ref{subsec:LocLaw}) and is closely related to the ASD of $W$. Specifically, both measures share the same support for their continuous densities, and the density $\varrho_{\mathbf{b}}$ of $\mu_{\vec b}$ concentrates around that of the ASD when the vector $\mathbf{b}$ is chosen uniformly on the hypersphere. Second, for a certain class of population covariance matrices $\Sigma$, as specified in Assumption \ref{as:TechA}, the limiting distribution ${\mu}_{\mathbf{b}}$ has one bulk component and takes the form given by \eqref{eq:MeasureForm}, with a finite number of delta masses outside the support of its density. For important, but technical reasons, we are led to consider a modified, spiked version of ${\mu}_{\mathbf{b}}$ (see \eqref{eq:hatmu})
\begin{align}\label{eq:hatmu0}
    \widehat {\mu}_{\mathbf{b}} = {\mu}_{\mathbf{b}} - \mu_{\mathrm{Disc}},
\end{align}
which we refer to as the spiked VASD.  We then show that
\begin{align*}
    \widehat {\mu}_{\mathbf{b}} \approx \mu_{\rm sASD} \text{ when }  \vec b \text{ is uniform on } \mathbb S^{N-1},
\end{align*}
where $\mathbb S^{N-1}$ denotes the sphere in $\mathbb R^N$, see Theorem \ref{thm:VESDHaar}.

The entries of the semi-infinite Jacobi matrix $\mathcal{J}(\widehat \mu_{\vec b})$ associated with $\widehat \mu_{\vec b}$ (see \eqref{eq:JacobOPDef}) and its Cholesky factor $\mathcal{L}(\widehat \mu_{\vec b})$ (c.f. \eqref{eq:CholOpDef}), will converge exponentially to constants that depend only on the edges $\gamma_\pm$ of the support of $\varrho_\mathbf{b}$, as shown by Theorem~\ref{thm:JacobCholAsympt}.  This, combined with the perturbation analysis of \cite{Ding2021b}, shows that the $n\times n$ principal submatrices of $\mathcal{L}(\widehat \mu_{\mathbf{b}})$ and $\mathcal{L}(\mu_{W,\mathbf{b}})$ are relatively close for $n\ll N^{1/6}$, as detailed in Corollary \ref{cor:checkmuPert}. As a result, $\mu_{W,\vec b}$ exhibits simple asymptotic properties that allow for efficient recovery of information about the ASD and the outliers of $W$ through the spiked ASD.

On the computational and algorithmic side, the Lanczos iteration (Algorithm~\ref{a:lanczos}) serves as a fundamental tool for recovering the Jacobi and Cholesky entries associated with $\mu_{W,\vec{b}}$. Through iterative matrix-vector multiplications between $W$ and a vector $\vec b$, the algorithm generates a $n \times n$ symmetric tridiagonal matrix $J_n$ in the $n$th step, which corresponds exactly to the $n \times n$ principal submatrix of $\mathcal{J}(\mu_{W,\vec b})$.

Given the strong asymptotic relationship between $\mu_{W,\vec b}$ and $\widehat \mu_{\vec b}$, the Lanczos algorithm is applied until a steady state is reached, at which point the matrix $J_n$ is extended by constants to be semi-infinite (we refer the reader to $\mathcal{J}_0 = \mathcal{L}_0\mathcal{L}_0^*$ in the proof of Theorem~\ref{thm:ConvRVESDEstim}). From this semi-infinite Jacobi matrix $\mathcal{J}_0$, we define the spectral measure $\widehat \mu_0$ as \cite{DeiftOrthogonalPolynomials},
\begin{align}\label{eq:mu0def}
    \widehat{m}_0(z) = \vec e^*_1\paren{\mathcal{J}_0-z}^{-1}\vec e_1, \quad \vec e_1 = [1,0,\dots]^*, \quad z\in \C \backslash \R,
\end{align}
where 
\begin{align}\label{eq:stieltjes}
    \widehat{m}_0(z) = \int_{\R} \frac{\widehat\mu_0(\dd x)}{x-z},
\end{align} 
is the Stieltjes transform\footnote{As a notational detail, for a measure ${\mu}_y$ we use ${m}_y$ to denote its Stieljtes transform and ${\varrho}_y$ to denote its density, if it exists.} of $\widehat \mu_0$. Given the structure of $\mathcal{J}_0$ and its Cholesky factor $\mathcal{L}_0$, the support endpoints (denoted as $\widehat{\gamma}_{\pm}$) can be exactly computed from the asymptotic Cholesky entries. Moreover, 
the Stieltjes transform $\widehat m_0(z)$ is recovered iteratively through a finite continued fraction whose coefficients are derived from these entries of $\mathcal{L}_0$. For the actual implementation, the details are provided in Algorithm \ref{Ea:realVESD}.  Theoretically, we  show that $\widehat{m}_0(z)$ provides an accurate estimate of the Stieltjes transform of $\widehat \mu_{\mathbf{b}}$; see Theorem \ref{thm:ConvRVESDEstim} for further details.  Schematically, our methodology is
\begin{align}\label{alg}
   (W, \vec b \sim \mathrm{Unif}(\mathbb S^{N-1})) \overset{\substack{\text{Lanczos with} \\n~\text{steps}}}{\longrightarrow} J_n \overset{\text{extend}}{\longrightarrow} \mathcal J_0 = \mathcal L_0 \mathcal L_0^* \overset{\substack{\text{continued} \\ \text{fractions}}}{\longrightarrow} \widehat m_0 \longrightarrow \widehat \mu_0 \approx \widehat \mu_{\vec b} \approx \mu_{\mathrm{sASD}}.
\end{align}

Finally, as mentioned earlier, the number of spikes can be estimated by counting the poles of the Stieltjes transform of the spiked ASD that lie to the right of support of its density.  We show that to count these poles, with high probability it suffices to count the poles of $\widehat m_0(z)$ that lie to the right of its density.  Additionally, to reduce variance, one can sample a sequence of iid vectors from $\mathbb{S}^{N-1}$ and apply an averaging procedure. We emphasize that in a practical implementation, directly averaging the asymptotic VESDs--or their approximations--is suboptimal, as their estimates may have varying supports. Instead, we first average the Cholesky entries sufficiently far down the matrix to reduce variance, then construct an estimator for the Stieltjes transform $\widehat{m}_0(z)$ (Algorithm \ref{Ea:ESD}).



Our main results, Theorems~\ref{thm:SolveP1} and \ref{thm:SolveP23}, are summarized as follows:
\begin{theorem*}[Informal]
Let $\widehat{m}_0(z)$ be the estimate after running our algorithm \eqref{alg} for $n = \mathrm{O}(\log N)$ iterations. Moreover, let $\widehat \gamma_{\pm}$ be estimates of the support endpoints  and $\widehat \gamma_j$ ($j = 1, 2, \dots, \widehat r$) be the poles of $\widehat m_0(z)$ for $z > \widehat \gamma_{+} + N^{-1/\delta}$ with $0<\delta<1/2$, then the true number of spikes $r$ is equal to $\widehat r$, 
    with overwhelming probability.
\end{theorem*}

To our knowledge, our spike estimator is the first to use the asymptotics of the VESD, treating it as a perturbation of the asymptotic measure, to estimate both the support of the asymptotic density and its Stieltjes transform. Unlike existing methods, it bypasses the need to compute the entire spectrum of the data matrix $W$, significantly enhancing computational efficiency. Additionally, by the Stieltjes inversion formula, the density $\widehat \varrho_{0}$ of $\widehat \mu_0$, if it exists, can be recovered at all points of continuity using
\begin{align}\label{eq:StieltjesInv}
    \widehat \varrho_{0}(\lambda) = \lim_{\epsilon\downarrow 0^+} \frac{\mathrm{Im}\,\widehat m_{0}(\lambda+\ri\epsilon)}{\pi}.
\end{align}
This provides an approximation for the ASD of $W$; see Figure \ref{fig:Demo} for an illustration.

\subsection{Review and discussion of other approaches on addressing spike detection.}
\label{sec_existingresults}  In this section, we summarize existing results on estimating the number of spikes using $W$. To the best of our knowledge, all existing literature relies on computing the entirety of the eigenvalues\footnote{The reader may note that some techniques below only require the top eigenvalues and therefore methods like subspace iteration and block Lanczos are potentially applicable.  Yet, these methods have not, to our knowledge, been used in these settings, likely due, in part, to the fact that one would need an a priori upper bound for the number of spikes.} of $W$, which can be computationally expensive for large matrices. Depending on how the eigenvalues are utilized, the literature can be broadly classified into the following two categories. 

\begin{itemize}
    \item \textbf{Detection based on the first few outlier eigenvalues.} It is well known that if spikes exist, the corresponding outlier eigenvalues may undergo the so-called BBP transition \cite{Baik2005,Ding2021c}. In particular, if the spikes (i.e., $\widetilde{\sigma}_i$ in (\ref{eq:SigmaEigen})) exceed a certain threshold, the corresponding eigenvalues $\lambda_i$ in (\ref{eq_spectralW}) will separate from the support of the ASD. This suggests that under appropriate assumptions on the spikes (see Remark \ref{rmk_detectable} below), one can estimate the number of spikes by counting the number of outlier eigenvalues departing from the ASD. Several estimators have been developed based on this idea, depending on how the outliers are counted. For instance, \cite{Bao2015Universality,Ding2019, Ding2022TracyWidom,onatski2009testing} used the ratio of differences between consecutive eigenvalues, while \cite{10.3150/19-BEJ1129,passemier2012determining,passemier2014estimation,Vinogradova2013StatisticalInference} employed eigen-ratios or eigen-differences. Additionally, \cite{kritchman2009nonparametric} applied a nonparametric approach to identify the separation gap between outliers and bulk eigenvalues, and \cite{dingboostrap} utilized a bootstrap procedure. A common challenge in these methods is the need to carefully select a tuning threshold to distinguish larger outlier eigenvalues from the smaller, more rigid bulk eigenvalues. This threshold is often determined through Monte Carlo simulations and sometimes requires the knowledge of $\Sigma_0$ in (\ref{eq:Sigma0Eigen}), which can be computationally expensive and lack robustness.  
    \item \textbf{Procedures based on all or nearly all the eigenvalues.} In this category of approaches, instead of directly counting the larger outlier eigenvalues, which requires a precise threshold, methods typically rely on utilizing nearly all eigenvalues. For example, \cite{Bai2018ConsistencyOA, Nadakuditi2008Sample} apply information criteria in model selection (such as AIC and BIC), making use of all eigenvalues; \cite{PA} employs a parallel analysis approach that also utilizes all eigenvalues; and \cite{Ke2023} reconstructs the limiting spectrum based on eigenvalue rigidity, using nearly all eigenvalues. Additionally, there is a research direction focused on estimating the number of spikes via hypothesis testing, which depends on all eigenvalues in various forms of linear spectral statistics (LSS); see, for example, \cite{Bianchi2011Performance,Dobriban2017SharpDI,10551866,johnstone2018testing,onatski2014signal}. Many of these methods require prior knowledge of the entry distribution or the structure of $\Sigma_0.$ Moreover, since these approaches rely on computing nearly all eigenvalues, their computational complexity increases significantly as the sample size and dimension grows.
\end{itemize}

As discussed and compared in \cite{Ke2023}, the second category of methods, which utilize all or nearly all eigenvalues, is generally the most numerically robust and relies on weaker eigen-gap assumptions compared to the first category, which only uses a few outlier eigenvalues. However, from a computational perspective, methods in the first category—those relying on the largest few eigenvalues—requires fewer calculations, typically incurring a cost of order $\mathrm{O}(b c N^2 \log N)$. Here, $b$ is a potentially diverging constant that depends on the eigen-gap of the extreme eigenvalues, while $c$ is usually divergent and represents the complexity of selecting the threshold in Monte Carlo simulations. In contrast, methods in the second category, which use all or nearly all eigenvalues, require at least a cost of order $\mathrm{O}(N^3)$ to compute the spectrum of $W$.

Our proposed method not only guarantees statistical efficiency and robustness under weaker separation assumptions but is also computationally cheap, typically requiring a cost of order $\mathrm{O}(N^2 \log N)$, as it uses the well-known Lanczos algorithm and avoids the need to compute eigenvalues and certain critical thresholds. In this way, we combine the relative efficiency of existing first-category methods, which rely only on computing the top eigenvalues, with the robustness gained from utilizing the entire spectrum of $W$, as seen in the second-category methods.







\subsection{Organization of the paper and some conventions.} The structure of the paper is as follows. In Section~\ref{sec_2}, we review the well-known Lanczos algorithm and introduce our proposed algorithms. Section~\ref{sec:JacChol} discusses the details of orthogonal polynomials associated with a given measure and their connection to the corresponding Jacobi matrices and their Cholesky factorizations. Additionally, we provide the fixed-point equation for characterizing the Stieltjes transform via the Cholesky factorization. In Section~\ref{sec:Model&LocLaw}, we formalize our model and present the relevant random matrix theory needed to analyze the VESD, which is central to our theoretical developments. Sections~\ref{sec:VESDLan} and \ref{sec:Estimation} focus on the main analysis of our proposed algorithms. Finally, in Section~\ref{sec:Comp}, we demonstrate the method and present further numerical investigations. The code used to generate the plots in this paper is available at \cite{AbiYounes2025SpikeDetection}.

{Throughout the paper, we will use the following conventions.  The $j$th standard basis vector is denoted by $\vec e_j$, with the dimension inferred from the context. For two non-negative sequences $(A_N)_{N \geq 1}$ and $(B_N)_{N \geq 1}$ depending on $N$, we use the notation $A_N \asymp B_N$ to mean $C^{-1}A_N\leq B_N \leq C A_N$ for some positive constant $C$. The symbol $\sim$ will be used to denote equality in law.  We also use the notation $A_N \gg B_N$ if there exists $\epsilon > 0$ such that $A_N \geq B_N N^{\epsilon}$ for sufficiently large $N$.}

For clarity, we summarize the abbreviations used for different spectral measures.
The empirical spectral distribution (ESD) and its asymptotic counterpart (ASD) are connected through the spiked asymptotic spectral distribution (spiked ASD or sASD). Similarly, the eigenvector empirical spectral distribution (VESD) and its asymptotic equivalent (VASD) are related via the spiked eigenvector spectral distribution (spiked VASD or sVASD).

\section{The Lanczos iterations and our proposed algorithms}\label{sec_2}
In this section, we present our new algorithms. Section \ref{sec:Lanczos&Jacob} introduces a high-level view of the well-known Lanczos algorithm, which serves as the foundation for our proposed algorithms for spike detection, detailed in Sections \ref{sec:pilotestimation} and \ref{sec:finalalgorithm}.  As we discuss the algorithms, including important subroutines, we point to the theoretical results in the forthcoming sections where properties of the algorithms are established.

\subsection{The Lanczos algorithm}\label{sec:Lanczos&Jacob}
The Lanczos algorithm (c.f., \cite{Lanczos1950,TrefethenBau}) is an iterative method used to obtain a tridiagonal approximation of a symmetric or Hermitian matrix through matrix-vector multiplications and inner products.  In its simplest form, it is given by Algorithm~\ref{a:lanczos} in Appendix~\ref{sec:lanczos}.  In our computations we use Lanczos with reorthogonalization \cite{Paige1980}.

Given the matrix $W$ in (\ref{eq:SCM_Model}) and a vector $\mathbf{b}$, the Lanczos iteration at step $n\leq N$ produces a Jacobi matrix $J_n$ and orthogonal vectors $\vec q_1,\dots,\vec q_{n+1}$ such that
\begin{align*}
    W Q_n = Q_n J_n + b_{n-1}\vec q_{n+1}\mathbf{e}_n^*,
\end{align*}
where $Q_n = [\vec q_1,\vec q_2,\dots,\vec q_n],$ $\vec q_1 = \mathbf{b}/\|\mathbf{b}\|$ and
\begin{align}\label{eq:JacMatrix}
    J_n = J_n(W,\mathbf{b}) = \begin{bmatrix}
	a_0 &b_0\\
	b_0 &a_1 & \ddots\\
	& \ddots &\ddots &b_{n-2}\\
	& &b_{n-2} &a_{n-1}
\end{bmatrix},\quad a_j \in \R, \quad b_j>0.
\end{align}
The columns of $Q_n$ form an orthonormal basis for the Krylov subspace $\text{span} \{\vec q_1, W \vec q_1, \dots, W^{n-1} \vec q_1\}$.


Every $n\times n$ Jacobi matrix $J_n$ defined in (\ref{eq:JacMatrix}) produces a probability measure
\begin{align*}
    \mu_{J_n} = \sum_{j=1}^n \omega_j \delta_{\lambda_j},
\end{align*}
where $(\lambda_j)_{j=1}^n$ are the eigenvalues of $J_n$ and $\omega_j$ is the squared modulus of the first component of the normalized eigenvector associated to $\lambda_j$. The spectral measure $\mu_J$, with $n=N$ and $J = J_N(W,\mathbf{b})$, coincides with the VESD associated with $W$ and $\mathbf{b}$ whenever $\mathbf{b}$ is a unit vector. There is a bijection between such measures and Jacobi matrices \cite{DeiftEigenvalue}, making the Lanczos algorithm an effective computational method for iteratively determining the Jacobi matrix entries associated with $\mu_{W,\mathbf{b}}$ without requiring the full spectrum of $W$.

Since $W$, with high-probability, has positive eigenvalues, one can further compute the Cholesky factorization of $J_n$ which is given by $J_n = L_nL_n^*$, where $L_n$ is a lower-bidiagonal matrix with positive entries as in Algorithm \ref{a:Chol} in Appendix \ref{sec:cholesky}. To fix notation, we set 
\begin{align}\label{eq:CholOpDef11}
    L_n: = \begin{bmatrix}
	\alpha_0\\
	\beta_0 &\alpha_1\\
	&\ddots &\ddots\\
	&& \beta_{n-2} &\alpha_{n-1}
\end{bmatrix}.
\end{align}


\subsection{Pilot estimation: algorithmic approach for the spiked VASD Stieltjes transform estimation} \label{sec:pilotestimation}
As discussed in Section \ref{sec:overviewofourapparoach}, our estimator for $r$ relies on the approximation of Stieltjes transforms of spiked VASDs at random directions, which can be computed iteratively using the Cholesky algorithm. In practice, the procedure can be divided into two subroutines.
Algorithm \ref{Ea:realVESDA} first runs the Lanczos iteration on the pair $(W,\vec b)$ and the computes the Cholesky factorization of the associated $n \times n$ Jacobi matrix. The outputted Cholesky factors are subsequently used as input into Algorithm \ref{Ea:realVESDB} to construct Stieltjes transforms of associated extended Jacobi matrices. See Lemma~\ref{l:finitepert} below for a precise discussion of the output of Algorithm~\ref{Ea:realVESDB}.

\newcounter{subrout-algorithm}
\renewcommand{\thealgorithm}{SR.\arabic{subrout-algorithm}}
\setcounter{subrout-algorithm}{0}
\begin{algorithm}[tbp]
    \refstepcounter{subrout-algorithm}
    \caption{Spiked VASD Estimation Subroutine 1} \label{Ea:realVESDA}
    \begin{algorithmic}[1]
        \Statex \textbf{Input:} A positive definite matrix $W$, a vector $\mathbf{b}$, and a convergence criteria.
        \Statex \textbf{Output:} A lower-bidiagonal semi-infinite matrix $\widehat {{L}}$.
        \State Run the Lanczos iteration (Algorithm~\ref{a:lanczos}) using the pair $(W,\mathbf{b})$ with $1 \ll n \ll N^{1/6}$ until the convergence criteria is satisfied. 
        \State Compute the Cholesky factors $\{\widehat{\alpha}_i\}_{i=0}^{n-1}$, $\{\widehat{\beta}_i\}_{i=0}^{n-2}$ using Algorithm~\ref{a:Chol}.
        \State Construct the semi-infinite matrix:
        \begin{align}\label{eq:L00}
            \mathcal{L} =
            \begin{bmatrix}
                \widehat{\alpha}_0 & & & & \\
                \widehat{\beta}_0 & \ddots & & &  \\
                & \ddots & \widehat{\alpha}_{n-3} & &   \\
                & & \widehat{\beta}_{n-3} & \widehat{\alpha}_{n-2} &  \\
                & & & \widehat{\beta}_{n-2} & \widehat{\alpha}_{n-2} \\
                & & & & \widehat{\beta}_{n-2} & \ddots \\
                & & & & & \ddots & \ddots 
            \end{bmatrix}
        \end{align}
        \State \Return $\mathcal {{L}}$.
    \end{algorithmic}
\end{algorithm}

The convergence criteria for Algorithm~\ref{Ea:realVESDA} should guarantee that the Jacobi entries are eventually close to constant. We propose three simple approaches:
\begin{enumerate}  
    \item For $C> 0$ set $n = \lceil C \log N \rceil$.
    \item Monitor the standard deviations of the last $q$ diagonal and off-diagonal entries and stop once they fall below a tolerance $\delta$.  
    \item Track two sequences of length $q$ for the diagonal and off-diagonal entries, separated by a fixed gap The process stops when the difference between their averages is within a tolerance $\delta_1$ and the standard deviations are below $\delta_2$.  
\end{enumerate}
In our experiments, we typically use the third approach. Theorem \ref{thm:ConvRVESDEstim} shows that $\mathrm{O}(\log N)$ Lanczos iterations are sufficient for accurate approximations and give the sufficiency of the first approach, if $C$ is known. As such the maximum number of iteration in Algorithm~\ref{Ea:realVESDA} can be set to $\lceil C \log N \rceil$ for $C$ large. In this case, the sequences $\{\widehat{\alpha}_j\}$ and $\{\widehat{\beta}_j\}$ fluctuate around their deterministic limits with deviations of order $N^{-1/2}$ (see Corollary \ref{cor:checkmuPert}), and a natural choice for the convergence tolerances is $\delta_1 = \delta_2 = C/\sqrt{N}$. In our numerical examples, setting $q = \lfloor \frac{1}{2} \log N \rfloor$, $\delta = 3/\sqrt{N}$, and $C = \lceil \max({6\log N + 24, \sqrt{N})} \rceil$ yields reliable results. It is important to note that if a stricter tolerance $\delta$ is initially set, it can be relaxed later without the need to restart the process or incur additional computational costs, as the Cholesky factors have already been computed. We leave an adaptive procedure that incorporates all of these strategies for future work.

\begin{algorithm}
    \refstepcounter{subrout-algorithm}
    \caption{Spiked VASD Estimation Subroutine 2} \label{Ea:realVESDB}
    \begin{algorithmic}[1]
        \Statex \textbf{Input:} A lower-bidiagonal semi-infinite matrix \eqref{eq:L00} that is constant after column $n-2$.
        \Statex \textbf{Output:} $\widehat \gamma_{\pm}$ and $\widehat m_{0}(z)$ 
        \State {Estimate the support endpoints} by computing
        \[
            \widehat{\gamma}_{-} = (\widehat{\alpha}_{n-2}-\widehat{\beta}_{n-2})^2, \quad \text{and} \quad \widehat{\gamma}_{+} = (\widehat{\alpha}_{n-2}+\widehat{\beta}_{n-2})^2.
        \]
        \State {Initialize the Stieltjes transform:}
        \[
            \widehat{m}_{n-2}(z) = \dfrac{\widehat{\alpha}_{n-2}^2-z-\widehat{\beta}_{n-2}^2+\sqrt{z-\widehat{\gamma}_{+}}\sqrt{z-\widehat{\gamma}_{-}}}{2z\widehat{\beta}_{n-2}^2}.
        \]
        \For{$i=n-3,n-4,\dots,0$} 
            \[
                \widehat{m}_i(z) = \dfrac{1}{\widehat{\alpha}^2_i-z-\widehat{\alpha}_i^2\widehat{\beta}_i^2 \paren{\frac{\widehat{m}_{i+1}(z)}{1+\widehat{\beta}_i^2 \widehat{m}_{i+1}(z)}}}.
            \]
        \EndFor
        \State \Return $\widehat \gamma_\pm$, $\widehat m_0(z)$.
        
    \end{algorithmic}
\end{algorithm}

We combine Algorithms~\ref{Ea:realVESDA} and \ref{Ea:realVESDB} to construct Algorithm~\ref{Ea:realVESD} which is applied to sample covariance matrices.  Theorem~\ref{thm:ConvRVESDEstim} below, in particular, establishes that, under certain general hypotheses on the sample covariance matrix $W$, the estimators $\widehat{\gamma}_{\pm}$, from Algorithm \ref{Ea:realVESD} are asymptotically consistent estimators for the edges of the support of the density for the spiked VASD. It further establishes that $\widehat m_0(z)$ from Algorithm~\ref{Ea:realVESD} is an asymptotically consistent estimator for the Stieltjes transform of the spiked VASD which is critical to spike detection (see $\widecheck\mu_{\vec b}$ in (\ref{eq:checkmu}) below).
\newcounter{pilot-algorithm}
\renewcommand{\thealgorithm}{P.\arabic{pilot-algorithm}}
\setcounter{pilot-algorithm}{0}
\begin{algorithm}
    \refstepcounter{pilot-algorithm}
    \caption{Spiked VASD Estimation} \label{Ea:realVESD}
    \begin{algorithmic}[1]
        \Statex \textbf{Input:} A sample covariance matrix $W$, a unit vector $\mathbf{b}$, and a convergence criteria.
        \Statex \textbf{Output:} Estimators $\widehat \gamma_{\pm}$ and $\widehat m_{0}(z)$ 
        \State Run Algorithm~\ref{Ea:realVESDA} with the stated input.
        \State Run Algorithm~\ref{Ea:realVESDB} on the output of Step 1.
        \State \Return $\widehat m_0(z)$.
    \end{algorithmic}
\end{algorithm}

\subsection{Averaging procedure and proposed method for spike detection}\label{sec:finalalgorithm} 

As outlined in Theorem~\ref{thm:SolveP1}, when $\mathbf{b}$ is uniformly sampled from the unit hypersphere, the spiked VASD serves as an unbiased, and asymptotically consistent, estimator of the ASD of $W$. However, a single sample can exhibit high variance. To mitigate this, one can draw iid vectors $\{\mathbf{b}_i\}_{i=1}^{k}$ and apply an averaging procedure to effectively reduce statistical instability. We first average the Cholesky entries sufficiently far down the matrix, and then construct an estimator for the Stieltjes transform using the previously employed method, after modifying the semi-infinite Cholesky factors to enforce a support constraint. The averaging process is detailed in Algorithm \ref{Ea:avg}.



\renewcommand{\thealgorithm}{SR.\arabic{subrout-algorithm}}
\setcounter{subrout-algorithm}{2}
\begin{algorithm}[H]
\refstepcounter{subrout-algorithm}
\caption{Cholesky Averaging Subroutine}
\label{Ea:avg}
\begin{algorithmic}[1]
    \Statex \textbf{Input:} Lower-bidiagonal semi-infinite matrices $\mathcal L(j)$, $j = 1, 2, \ldots, k$, that are constant after column $n(j) - 2$, and an integer $q \ll n(j)$.
    \Statex \textbf{Output:} Averaged lower-bidiagonal matrices $\mathcal L(j)$, $j = 1, 2, \ldots, k$.
    
    \State Initialize: $\widehat{\alpha} = 0$, $\widehat{\beta} = 0$
    
    \For{$j = 1$ to $k$}
        \For{$\ell = n(j) - q - 1$ to $n(j) - 2$}
            \State $\widehat{\alpha} \gets \widehat{\alpha} + \widehat{\alpha}_\ell(j)$,~$\widehat{\beta} \gets \widehat{\beta} + \widehat{\beta}_\ell(j)$
        \EndFor
    \EndFor
    
    \State $\widehat{\alpha} \gets \frac{1}{kq} \widehat{\alpha}$,~$\widehat{\beta} \gets \frac{1}{kq} \widehat{\beta}$
    
    \For{$j = 1$ to $k$}
        \For{$\ell = n(j) - q - 1$ to $n(j) - 2$}
            \State Replace $\widehat{\alpha}_\ell(j)$ and $\widehat{\beta}_\ell(j)$ with $\widehat{\alpha}$ and $\widehat{\beta}$, respectively.
        \EndFor
    \EndFor
    
    \State \Return $\mathcal L(j)$ for each $j = 1, 2, \ldots, k$
\end{algorithmic}
\end{algorithm}


\renewcommand{\thealgorithm}{P.\arabic{pilot-algorithm}}
\setcounter{pilot-algorithm}{1}
\begin{algorithm}[H]
    \refstepcounter{pilot-algorithm}
    \caption{ASD Estimation} \label{Ea:ESD}
    \begin{algorithmic}[1]
        \Statex \textbf{Input:} Sample covariance matrix $W$ and an integer $k$.
        \Statex \textbf{Output:} An estimator for $\gamma_{\pm}$ and $m_{\mathrm{ASD}}(z)$.
        
        \State Sample $k$ independent vectors $\brac{\vec{b}_j}_{j=1}^k$ that are uniform on the hypersphere.
        \For{$j = 1$ to $k$}
            \State Run Estimation Subroutine~\ref{Ea:realVESDA} to compute $\mathcal{L}(\vec{b}_j)$.
        \EndFor
        \State Run Algorithm \ref{Ea:avg} on $\mathcal L(\vec{b}_j)$ for $j = 1,\dots,k$, obtaining $\widehat {\mathcal L}(\vec{b}_j)$ for $j = 1,\dots,k$.
        \For{$j = 1$ to $k$}
            \State Run Algorithm \ref{Ea:realVESDB} on $\widehat{{\mathcal L}}(\vec b_j)$ to compute\textsuperscript{\ref{foot:gammapm}} $\widehat{\gamma}_{\pm}$ and ${\widehat{m}}^j_0(z)$.
        \EndFor
        \State Compute
        \[
            \widehat{m}_0(z)= \frac{1}{k} \sum_{j=1}^k {\widehat{m}^j}_0(z),  
        \]
        \State \Return $\widehat{\gamma}_{\pm}$, $\widehat{m}_0(z)$ and $(\widehat m^j_0(z))_{j=1}^k$.
    \end{algorithmic}
\end{algorithm}
\footnotetext[1]{\label{foot:gammapm}Note that $\widehat \gamma_\pm$ will not depend on $j$.}

Using Algorithm \ref{Ea:avg}, one can draw sequences of $\mathbf{b}_i, 1 \leq i \leq k$ and obtain a robust estimator for the ASD of $W$ in Algorithm \ref{Ea:ESD}. The consistency of this estimator put forth in Algorithm~\ref{Ea:ESD} is guaranteed by Theorem~\ref{thm:SolveP1}.

With Algorithm~\ref{Ea:ESD} in hand, we are able to develop our final spike detection algorithm.  The algorithm works by first running Algorithm~\ref{Ea:ESD} and counting the poles in the computed Stieltjes transforms that lie sufficiently far to the right of the estimate support $[\widehat \gamma_-, \widehat \gamma_+]$.  We refer the reader to Appendix~\ref{a:poles} for an in-depth discussion of two methods (c.f. Algorithms~\ref{PE:ConnCoef} and \ref{PE:Trunc}) to compute these poles. 

\renewcommand{\thealgorithm}{P.\arabic{pilot-algorithm}}
\setcounter{pilot-algorithm}{2}
\begin{algorithm}[H]
    \refstepcounter{pilot-algorithm}
    \caption{Spike detection procedure}\label{finaldetectionalgorithm}
    \begin{algorithmic}[1]
    \Statex \textbf{Input:} Sample covariance matrix $W$, an integer $k$ and a threshold parameters $C$, $0<\delta<1/2$.
    \Statex \textbf{Output:} An estimator $\widehat r$ for the number of spikes in the population covariance.
        \State Run Algorithm~\ref{Ea:ESD}, obtaining $\widehat{\gamma}_{\pm}$, $\widehat{m}_0(z)$ and $(\widehat m_0^j(z))_{j=1}^k$.
        \For{$j = 1$ to $k$}
            \State Compute $\widehat r_j$, the number of poles of $\widehat m_0^j(z)$ that are larger than $\widehat \gamma_+ + C N^{-\delta}$.
        \EndFor
    \State \Return $\widehat r$ which could be the mode or the rounded average of $\{\widehat r_j\}_{j=1}^k$.
    \end{algorithmic}
\end{algorithm}



\section{Jacobi matrices, Cholesky factorizations and random matrices}\label{sec:JacChol}
\subsection{Orthogonal polynomials, Jacobi matrices and Cholesky factorizations} Assume that $\mu$ is Borel measure on $\mathbb R$ with finite mass and compact support. The orthonormal polynomials $(p_n)_{n\geq 0}$, $p_n(\lambda)=p_n(\lambda;\mu)$ for $\mu$ are constructed by applying the Gram-Schmidt process to the sequence $(\lambda\mapsto 1,\lambda\mapsto \lambda, \lambda\mapsto \lambda^2, \ldots)$ with the $L^2(\mu)$ inner product \cite{Szego1939},
\begin{align}
    \int_{\R} p_i(\lambda)p_j(\lambda)\mu(\mathrm{d}\lambda) = \delta_{ij},
\end{align}
where $\delta_{ij}$ is the Kronecker delta.  We also impose that the leading coefficient of $p_n(\lambda)$ is positive.  If the support of $\mu$ contains at least $N$ points, then $p_n(\lambda)$ exists for $0\leq n\leq N-1$.  The orthonormal polynomials satisfy a symmetric three-term recurrence 
\begin{align}\label{eq:ScalarOrthonRec}
    \lambda p_n(\lambda) = b_n p_{n+1}(\lambda ) + a_n p_n(\lambda) + b_{n-1} p_{n-1}(\lambda), \quad n\geq0, \quad b_n>0,
\end{align}
where 
\begin{align}
    a_n = \int_\R \lambda p_n^2(\lambda)\mu(\mathrm{d}\lambda) \quad \text{and} \quad b_n = \int_\R \lambda p_n(\lambda)p_{n+1}(\lambda)\mu(\mathrm{d}\lambda),
\end{align}
with the convention $p_{-1} \equiv 0$ and $b_{-1} = 1$. Here $a_n = a_n(\mu)$, $b_n = b_n(\mu)$ are called the recurrence coefficients. With these coefficients, one can define a corresponding semi-infinite, symmetric, tridiagonal matrix,
\begin{align}\label{eq:JacobOPDef}
    \mathcal{J}(\mu) = \begin{bmatrix}
	a_0 &b_0\\
	b_0 &a_1 &b_1\\
	&b_1 &a_2 &\ddots\\
	&&\ddots &\ddots
\end{bmatrix},
\end{align}
commonly referred to as a Jacobi matrix (or operator) with spectral measure $\mu$. 

One key result is that the Jacobi matrix is connected with the Stieltjes transform of $\mu$ in the sense that $\mathcal{J}(\mu)$ is the unique Jacobi matrix that satisfies \cite{DeiftOrthogonalPolynomials},
\begin{align}\label{eq:JacResId}
    \vec e_1^*(\mathcal{J}(\mu)-z)^{-1}\vec e_1 = \nu\int_{\R} \frac{\mu(\mathrm{d}\lambda)}{\lambda-z}, \quad \text{for }\mathrm{Im}~z>0, \quad \nu^{-1} = \int_{\R}{\mu(\mathrm{d}\lambda)}.
\end{align}
Moreover, suppose that $\mathrm{supp}(\mu)\subset (0,\infty)$, then $\mathcal{J}(\mu)$ is invertible and it has a Cholesky decomposition $\mathcal{J}(\mu)=\mathcal{L}(\mu)\mathcal{L}(\mu)^*$, where the Cholesky factor $\mathcal{L}(\mu)$ takes the form
\begin{align}\label{eq:CholOpDef}
    \mathcal{L}(\mu) = \begin{bmatrix}
	\alpha_0\\
	\beta_0 &\alpha_1\\
	&\beta_1 &\alpha_2\\
	&&\ddots &\ddots
\end{bmatrix}, ~ \alpha_n = \alpha_n(\mu)>0, ~ \beta_n = \beta_n(\mu)>0.
\end{align}

The asymptotic properties of $\mathcal{J}(\mu)$ and its Cholesky factorization $\mathcal{L}(\mu)$ are  heavily tied to the regularity conditions outlined in Assumption \ref{as:measurecond} below, see also \cite{DeiftOrthogonalPolynomials,KuijlaarsInterval,Ding2021b}. More specifically, if $\mu$ is supported on a single interval with square root behavior at the endpoints, and possibly a few outliers, the corresponding Jacobi operator and its Cholesky factor exhibit simple, exponential asymptotic behavior that can be characterized by the edges of the support.

\begin{assumption}\label{as:measurecond}
    Consider a measure $\mu = \mu(N)$ that satisfies the following assumptions with absolute constants $D\geq 1$ and $\tau,\sigma>0$:
    \begin{enumerate}
        \item The measure $\mu$ is of the form
        \begin{align*}
            \mu(\mathrm{d}\lambda) = h(\lambda) \mathds{1}_{[a,b]}(\lambda)(b-\lambda)^{\alpha}(\lambda-a)^{\beta}\mathrm{d}\lambda + \sum_{j=1}^p w_j \delta_{c_j}(\mathrm{d}\lambda),
        \end{align*}
        where $\alpha=\pm \frac{1}{2}$, $\beta = \pm \frac{1}{2}$, $b > a \geq \tau $, $w_j>0$ and $c_j>b$ for all $1\leq j\leq p$.

        \item We allow $\mu$ to depend implicitly on a parameter $N$ but require that $p$ be non-negative and constant (for sufficiently large $N$).
        Additionally, we assume that $\min_{i \neq j} |c_i-c_j| \geq C_\gamma e^{-\gamma N}$ for all $\gamma > 0$, and that
        \begin{align*}
            \min\{|a-b|,|a-c_j|,|b-c_j|\}\geq \tau \quad \text{for all }j=1,2,\dots,p.
        \end{align*}

        \item We associate a bounded open set $\Omega$ (independent of $N$) containing $[a,b]$ for all $N$ such that $h$ has an analytic continuation to $\Omega$.

        \item We suppose that $h$ is bounded uniformly from above and below on $\Omega$, i.e.
        \begin{align*}
            \sup_{z\in \Omega} \max \brac{\abs{h(z)},\abs{h(z)}^{-1}}\leq D.
        \end{align*}

        \item For every $j$, we assume that either $N^{-\sigma}/D\leq \abs{w_j}\leq D$, $0\leq \sigma <\infty$ or $w_j = 0$. 
    \end{enumerate}
\end{assumption}

These assumptions impose precise square-root behavior of the density of $\mu$ at the endpoints of the support of $\mu$, something that is absolutely critical to obtaining the exponential asymptotics stated in the next result.  Note that while we allow $\mu$ to depend on $N$, the conditions imposed are sufficient to obtain uniform error terms in the asymptotics.

\begin{theorem}\label{thm:JacobCholAsympt}
    Let $\mu$ be a measure satisfying Assumption \ref{as:measurecond}, then there exists $\kappa>0$ that depends only on $D,\sigma, \Omega, \tau$, such that
    \begin{align*}
        a_n(\mu) = \frac{b+a}{2} + \OO(e^{-\kappa n}), \quad b_n(\mu)  = \frac{b-a}{4} + \OO(e^{-\kappa n}),
    \end{align*}
    and
    \begin{align*}
        \alpha_n(\mu)  = \frac{\sqrt{a}+\sqrt{b}}{2} +\OO(e^{-\kappa n}), \quad \beta_n(\mu)  = \frac{\sqrt b- \sqrt a}{2} + \OO(e^{-\kappa n}). 
    \end{align*}
\end{theorem}
\begin{proof}   
    The result for Jacobi operators in the case $a=-1,b=1$ without discrete contributions was established in \cite{Kuijlaars2003, KuijlaarsInterval}, and the approach was later extended in \cite{Ding2021b}. The only piece that requires an extra argument is that \cite{Ding2021b} required $\{c_j\}_{j=1}^p$ to be well separated.  Yet, the estimates can be seen to hold even if points in this set are allowed to contract at a sufficiently slow exponential rate, see Assumption~\ref{as:measurecond}(2). The corresponding argument for the Cholesky factor is provided in \cite[Theorem 5.4]{Ding2021}.
\end{proof}

Theorem~\ref{thm:JacobCholAsympt} has profound implications. It implies that after analyzing only logarithmically many recurrence coefficients, the support of the density of $\mu$ can be found within any polynomially small error.  This will allow us to run Lanczos on a random matrix for logarithmically many steps to estimate the spiked VASD.


\subsection{Characterization of Stieltjes transforms using the Cholesky decomposition}\label{sec_newnewnew}

In this section we discuss both Jacobi matrices for which the Stieltjes transform can be computed exactly, and those for which we can approximate it reliably. We begin with some general observations regarding Cholesky factorizations of Jacobi matrices.  Recall that the definition of the Stieltjes transform \eqref{eq:stieltjes} for the measure $\mu$ and the properties of its associated Jacobi operators \eqref{eq:JacResId}, we have
\begin{align*}
    m_{\mu}(z) = \vec e_1^*\paren{\mathcal{J}(\mu)-z}^{-1}\vec e_1, \quad \text{for } \mathrm{Im}~z>0.
\end{align*}
The tridiagonal structure of $\mathcal{J}({\mu}) = \mathcal{L}({\mu})\mathcal{L}^*({\mu})$, enables a more detailed analysis of the resolvent. 
Observe that 
\begin{align}\label{eq:CholConst}
    \mathcal{J}(\mu) = \mathcal{L}(\mu)\mathcal{L}(\mu)^* = \begin{bmatrix}
        \alpha_0^2 & \alpha_0\beta_0 & & \\
        \alpha_0\beta_0 & \alpha_1^2 + \beta_0^2 & \alpha_1\beta_1 \\
        & \alpha_1\beta_1 & \ddots & \ddots \\
        & & \ddots & \ddots 
      \end{bmatrix}, \quad \alpha_i = \alpha_i(\mu), \quad \beta_i = \beta_i(\mu).
\end{align}
We aim to express the resolvent of $\mathcal{J}(\mu)$ in terms of a subblock that maintains a similar structure. To achieve this, we first use the Schur complement (see Lemma \ref{lem:SchurCompl}), leading to 
\begin{align}
    m_{\mu}(z) = \dfrac{1}{\alpha_0^2-z-\alpha^2_0\beta^2_0 \vec e_1^*\paren{\mathcal{J}^{(1)}(\mu)-z}^{-1}\vec e_1},
\end{align}
where $\mathcal{J}^{(1)}(\mu)$ is obtained from $\mathcal{J}(\mu)$ by removing the first row and column. Note that $\mathcal{J}(\mu)$ and $\mathcal{J}^{(1)}(\mu)$ do not have the same structure, as $\mathcal{J}^{(1)}(\mu)$ has an additional term in its first entry. To address this, we define $\mathcal{J}_1(\mu) = \mathcal{J}^{(1)}(\mu) - \beta_0^2 \vec e_1\vec e_1^*$. Using the Woodbury matrix identity (c.f. Lemma \ref{lem:Woodbury}), we find that
\begin{align*}
    \vec e_1^*\paren{\mathcal{J}^{(1)}(\mu)-z}^{-1}\vec e_1 = \frac{m_1(z)}{1+\beta_0^2 m_1(z)}, \quad \text{where} \quad m_1(z) := \vec e_1^*(\mathcal{J}_1(\mu)-z)^{-1}\vec e_1.
\end{align*}
This process can be repeated indefinitely, resulting in the following continued fraction representation for the $(1,1)$-entry of the resolvent
\begin{align}\label{eq:ContFracLimMeas}
    m_{\mu}(z) = \dfrac{1}{\alpha_0^2-z-\alpha_0^2\beta_0^2\paren{\frac{m_1(z)}{1+\beta_0^2 m_1(z)}}}, \quad m_{i}(z) = \dfrac{1}{\alpha_i^2-z-\alpha_i^2\beta_i^2\paren{\frac{m_{i+1}(z)}{1+\beta_i^2 m_{i+1}(z)}}}, \quad \text{for }i=1,2,\dots.
\end{align}

At first glance, this expansion may not appear to provide any new insights. But in the special case where $\alpha_i = \alpha, \beta_i = \beta$ for $i \geq n -2$, we have that $m_{n-1}(z) = m_i(z)$ for $i \geq n-1$, and therefore we have
\begin{align*}
    m_{n-1}(z) = m(z) = \dfrac{1}{\alpha^2-z-\alpha^2\beta^2\paren{\frac{m(z)}{1+\beta^2 m(z)}}} \quad \text{with } \mathrm{Im}~m(z)>0.
\end{align*}
This yields 
\begin{align}\label{eq:explicit_solve}
    m_{n-1}(z) = \dfrac{\alpha^2-z-\beta^2+\sqrt{z-(\alpha+\beta)^2}\sqrt{z-(\alpha-\beta)^2}}{2z\beta^2}.
\end{align}
Note that \eqref{eq:explicit_solve} can also be applied to recover the support of $\mu$,
\begin{align}\label{eq:SuppFormulae}
    \gamma_{-} = (\alpha-\beta)^2, \quad \text{and} \quad \gamma_{+} = (\alpha+\beta)^2.
\end{align}
We have established the following in this setting.


\begin{lemma}\label{l:finitepert}
Consider the extended Cholesky factor
    \begin{align}\label{eq:L00_ext}
            \mathcal{L} =
            \begin{bmatrix}
                \widehat{\alpha}_0 & & & & \\
                \widehat{\beta}_0 & \ddots & & &  \\
                & \ddots & \widehat{\alpha}_{n-3} & &   \\
                & & \widehat{\beta}_{n-3} & \widehat{\alpha}_{n-2} &  \\
                & & & \widehat{\beta}_{n-2} & \widehat{\alpha} \\
                & & & & \widehat{\beta} & \widehat{\alpha} \\
                & & & & & \widehat{\beta} & \ddots \\
                & & & & & & \ddots 
            \end{bmatrix}
        \end{align}
        and the associated Jacobi matrix $\mathcal J = \mathcal L \mathcal L^*$.  Then $\widehat m_0(z)$, the output Algorithm~\ref{Ea:realVESDB} applied to the upper $n\times n$ principal subblock of $\mathcal L$, satisfies
        \begin{align*}
            \widehat m_0(z) = \vec e_1^* (\mathcal J - z)^{-1} \vec e_1.
        \end{align*}
        Furthermore, $\widehat m_0(z)$ is the Stieltjes transform of a measure that has its density supported on $[\widehat \gamma_-,\widehat \gamma_+]$ where
        \begin{align*}
            \widehat \gamma_{-} = (\widehat \alpha-\widehat \beta)^2, \quad \text{and} \quad \widehat \gamma_{+} = (\widehat \alpha+\widehat \beta)^2.
        \end{align*}
\end{lemma}

\begin{remark}
    For Cholesky factors that are not exactly constant outside of a finite-size block, resolvent estimates are applicable to estimate the difference between $\vec e_1^* (\mathcal L \mathcal L^* - z)^{-1} \vec e_1$ and its approximation found replacing $\mathcal L$ by an appproximation of the form \eqref{eq:L00_ext}.  This estimation is performed in Theorem~\ref{thm:SolveP23} below.
\end{remark}

\subsection{Random matrices and fixed point equations for the VASD}\label{sec_newnewnewnew1} 

The results from the previous calculation allow one to exactly compute the measure associated to a Jacobi matrix coming from a Cholesky factorization that is eventually constant.  The class of sample covariance matrices that we consider have two important features that interact well with this fact:
\begin{enumerate}
    \item The Cholesky factorization associated to $\mathcal J(\mu)$, where $\mu$ is the (spiked) VASD, is approximately constant if a small number of rows and columns are removed (again, see Theorem~\ref{thm:JacobCholAsympt}).
    \item The Cholesky entries computed via the Cholesky factorization of the Lanczos output concentrate.
\end{enumerate}
This gives a new fixed point equation to solve to approximate the (spiked) VASD.  This is captured in Theorem~\ref{thm:ConvRVESDEstim}, but we demonstrate it here with two examples.

\begin{example}
    Consider $\Sigma = I$.  We apply both stages of Algorithm~\ref{Ea:realVESD} explicitly.  After applying Algorithm~\ref{Ea:realVESDA} to $(W, \vec b)$ for any deterministic vector $\vec b$, it follows that
    \begin{align*}
        L \sim \frac{1}{\sqrt{M}} \begin{bmatrix}
                \chi_M & & & \\
                \chi_{N-1} & \chi_{M-1}\\
                & \ddots & \ddots &   \\
                & & \chi_1 & \chi_{M-N +1}
            \end{bmatrix},
    \end{align*}
    where the entries are independent chi random variables with subscripts denoting the degrees of freedom.
    As $N/M \to c \in (0,1)$, $N \to \infty$, we have, entrywise
    \begin{align*}
        L \to  \begin{bmatrix}
                1 & & & \\
                \sqrt{c} & 1\\
                & \sqrt{c} & 1 &   \\
                & & \ddots & \ddots
            \end{bmatrix}.
    \end{align*}
    Thus, we can take $n = 2$ and find the estimator using Algorithm~\ref{Ea:realVESDB}
    \begin{align*}
    \widehat m_{0}(z) = \dfrac{\alpha_0^2-z-\beta_0^2+\sqrt{z-(\alpha_0+\beta_0)^2}\sqrt{z-(\alpha_0-\beta_0)^2}}{2z\beta_0^2}, \quad \alpha_0 \sim \frac{\chi_M}{\sqrt{M}},\quad \beta_0 \sim \frac{\chi_{N-1}}{\sqrt{M}}.
    \end{align*}
    This gives the limit
    \begin{align}\label{eq:ConvMP}
        \widehat m_{0}(z)  \to \dfrac{1-z-c+\sqrt{z-c_{+}}\sqrt{z-c_{-}}}{2c z}, \quad c_{\pm} = (1\pm\sqrt{c})^2,
    \end{align}
    which is the Stieltjes transform of the Marchenko--Pastur law.
    
\end{example}

\begin{example}\label{Ex:StdSpikedCov}
    We repeat the previous calculation for $\Sigma = \mathrm{diag}(\ell,1,\ldots,1)$ with $\ell>1$.  We again apply both stages of Algorithm~\ref{Ea:realVESD} explicitly but this time with $\vec b = \vec e_1$.  After applying Algorithm~\ref{Ea:realVESDA}, it follows that \cite{Bloemendal2013},
    \begin{align*}
        L \sim \frac{1}{\sqrt{M}} \begin{bmatrix}
                \sqrt{\ell} \chi_M & & & \\
                \chi_{N-1} & \chi_{M-1}\\
                & \ddots & \ddots &   \\
                & & \chi_1 & \chi_{M-N +1}
            \end{bmatrix},
    \end{align*}
    where the entries are independent chi random variables with subscripts denoting the degrees of freedom.
    As $N/M \to c \in (0,1)$, $N \to \infty$, we have, entrywise
    \begin{align*}
        {L} \to  \begin{bmatrix}
                \sqrt{\ell} & & & \\
                \sqrt{c} & 1\\
                & \sqrt{c} & 1 &   \\
                & & \ddots & \ddots
            \end{bmatrix}.
    \end{align*}
    We take $n = 3$ and apply Algorithm~\ref{Ea:realVESDB} to find the estimator
    \begin{align*}
        \widehat{m}_0(z) = \dfrac{1}{\widehat{\alpha}^2_0-z-\widehat{\alpha}_0^2\widehat{\beta}_0^2 \paren{\frac{\widehat{m}_{1}(z)}{1+\widehat{\beta}_0^2 \widehat{m}_{1}(z)}}}
    \end{align*}
    where $\widehat m_1(z)$ satisfies \eqref{eq:ConvMP}. Thus, the limit of $\widehat{m}_0(z)$ exists and can be explicitly determined, yielding after simplification
    \begin{align*}
        \widehat m_0(z) = \dfrac{-2z+\ell\paren{1-c+z+\sqrt{z-c_{+}}\sqrt{z-c_{-}}}}{2z\paren{(\ell+1)z+\ell(\ell-1+c)}}.
    \end{align*}
    This expression shows that $\widehat{m}_0(z)$ potentially has poles at $x_0 = \ell + \frac{\ell c}{\ell - 1}$ and $x_1 = 0$. However, by computing residues at these potential poles, we find
    \begin{align*}
        \mathrm{Res}_{z=x_0}(\widehat m_0(z)) = \dfrac{c+(\ell-1)\paren{1-\abs{\frac{(\ell-1)^2-c}{\ell-1}}-\ell}}{2(\ell-1)(\ell-1+c)}  \quad \text{and} \quad \mathrm{Res}_{z=x_1}(\widehat m_1(z)) = \frac{1-c-|1-c|}{2(\ell+(c-1))}=0.
    \end{align*}
    In this case, the VASD takes the form,
    \begin{align}\label{eq:StdSpikedCovVASD}
        \widehat \mu_0(\mathrm{d}x) = \dfrac{\ell \sqrt{c_{+}-x}\sqrt{x-c_{-}}}{2\pi x(\ell^2+\ell(c-1-x)+x)}\mathds{1}_{[c_{-},c_{+}]}(x)\mathrm{d}x + w_0\mathds{1}_{\ell>1+\sqrt{c}}\delta_{x_0}(\mathrm{d}x), \quad  w_0 = \dfrac{(\ell-1)^2-c}{(\ell-1)(\ell-1+c)}.
    \end{align}
    See Figure~\ref{fig:StdSpikedCov} for a visualization of the density at different values of $\ell$.

    \begin{figure}[tbp]
        \centering
        \begin{subfigure}{0.45\textwidth}
            \centering
            \includegraphics[width=\textwidth]{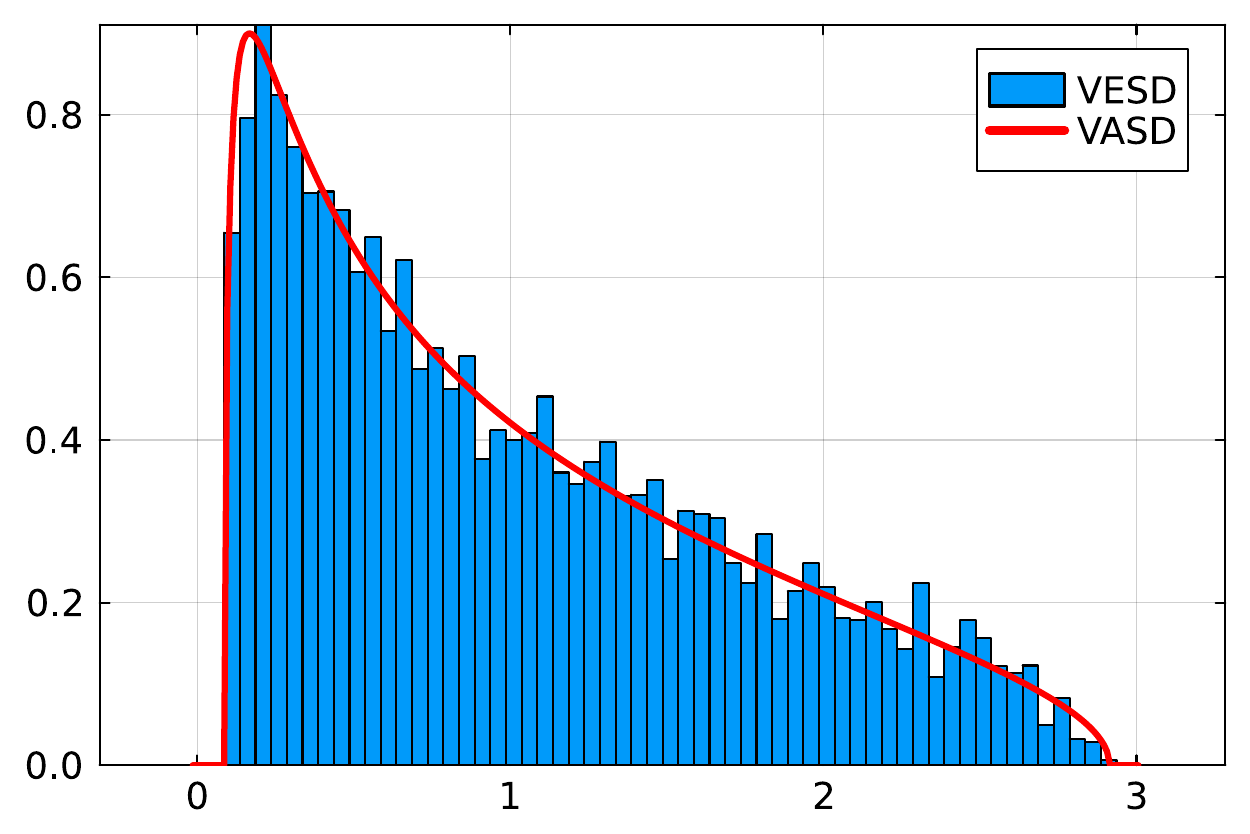}
            \caption{$\ell = 1$}
        \end{subfigure}
        \begin{subfigure}{0.45\textwidth}
            \centering
            \includegraphics[width=\textwidth]{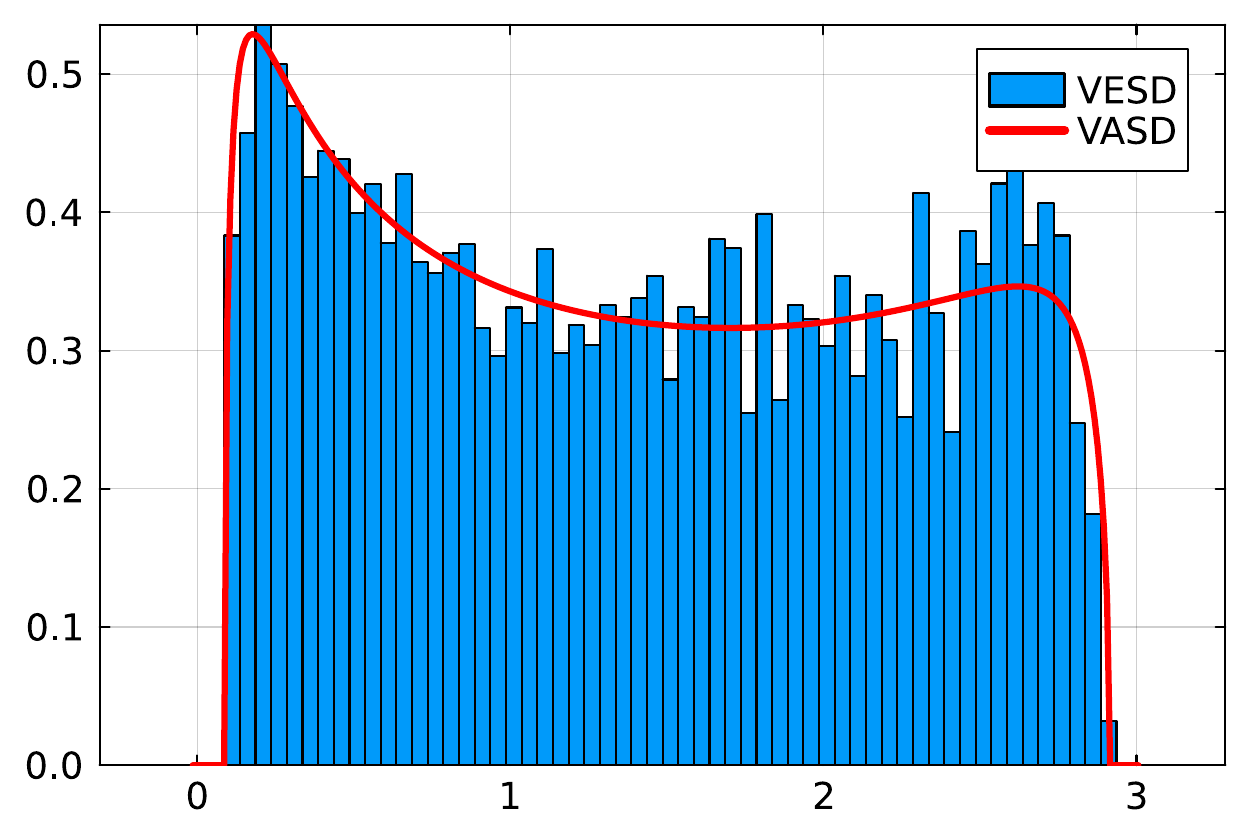}
            \caption{$\ell=1.4$}
        \end{subfigure}
        
        \vspace{0.5cm} 
        
        \begin{subfigure}{0.45\textwidth}
            \centering
            \includegraphics[width=\textwidth]{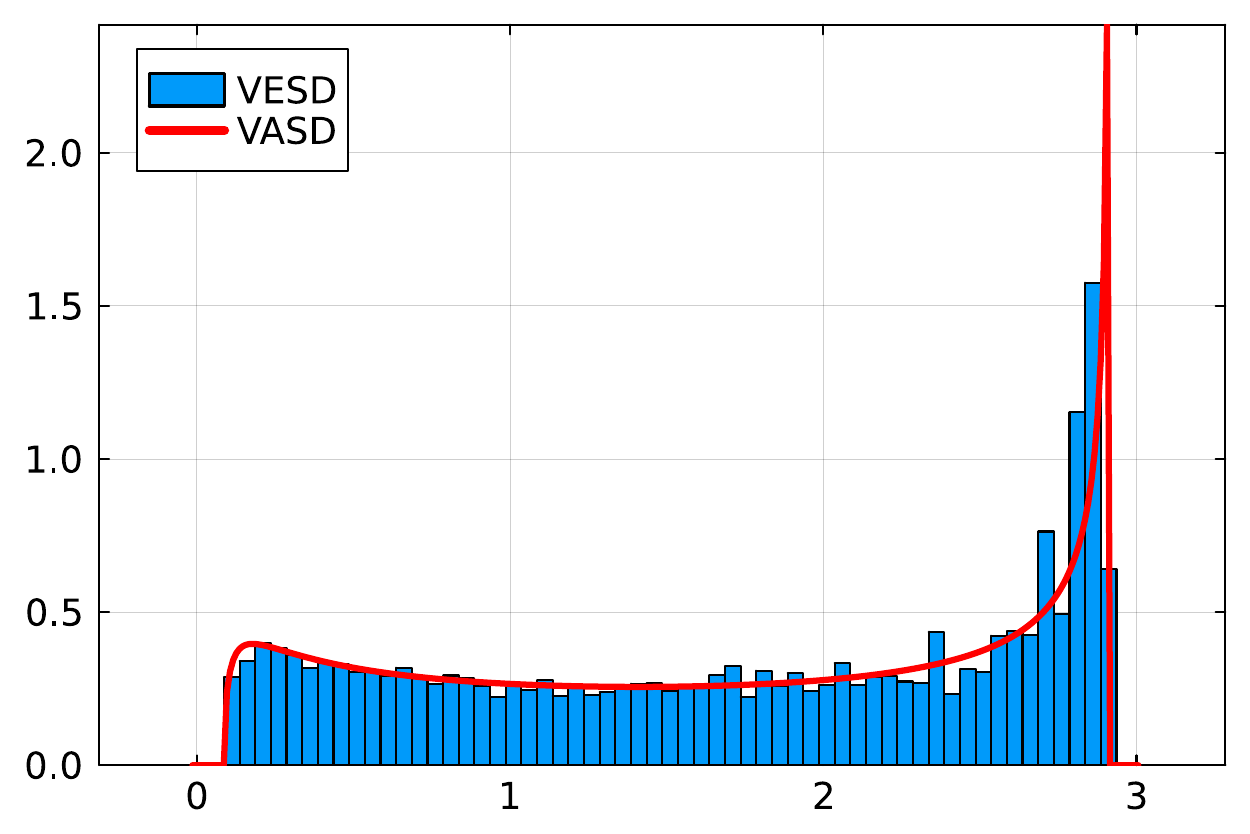}
            \caption{$\ell=1+\sqrt{0.5}$}
        \end{subfigure}
        \begin{subfigure}{0.45\textwidth}
            \centering
            \includegraphics[width=\textwidth]{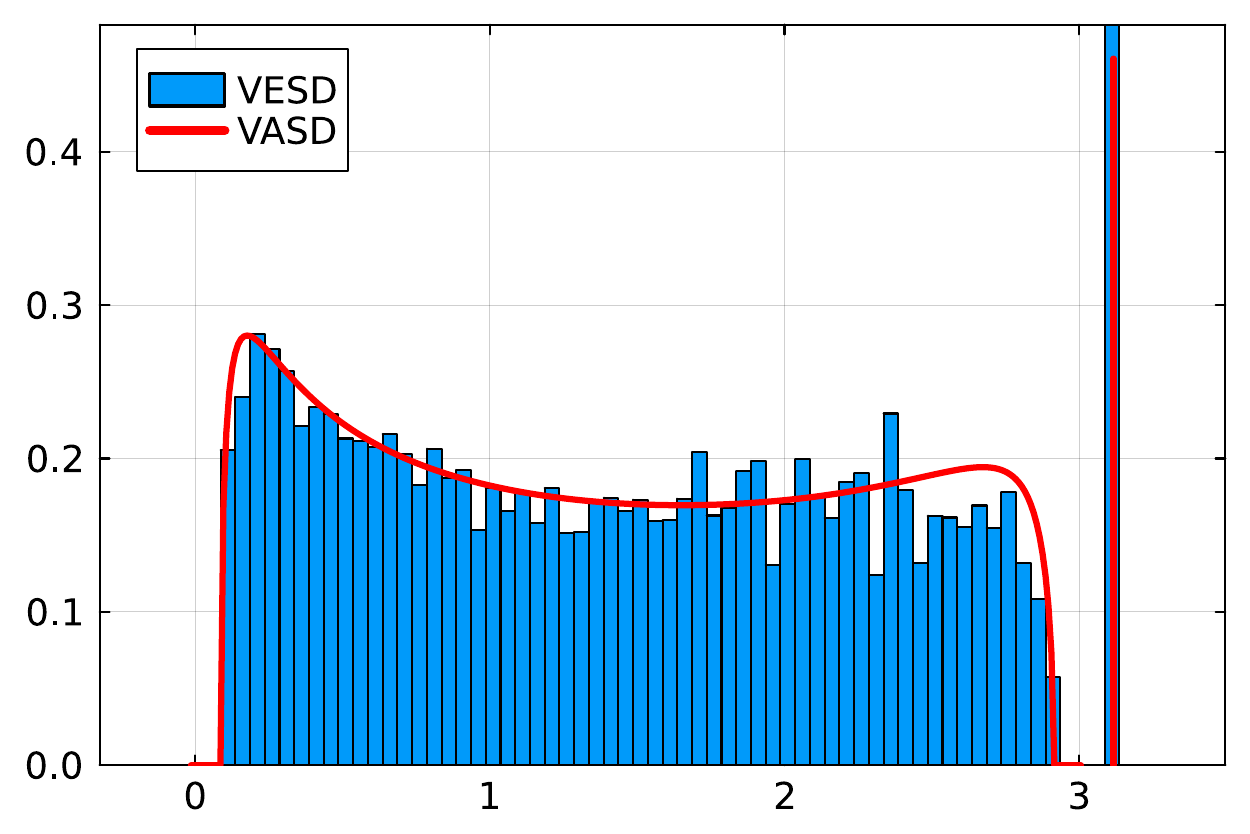}
            \caption{$\ell=2.2$}
        \end{subfigure}
        
        \caption{VESD of the sample covariance matrix in Example~\ref{Ex:StdSpikedCov} for $\vec{b} = \vec{e}_1$ with $N = 10000$, $M = 20000$, and $c = 0.5$, compared to the VASD from \eqref{eq:StdSpikedCovVASD} for different values of $\ell$.}
        \label{fig:StdSpikedCov}
    \end{figure}
    
\end{example}




\section{Large sample covariance matrices and local laws}\label{sec:Model&LocLaw}
Our analysis relies on the investigation of the spiked sample covariance matrix $W$ in (\ref{eq:SCM_Model}) and the associated non-spiked sample covariance matrix $W_0$ in (\ref{eq:NonSpikedModel}). We assume that the ASD associated with $W_0$ satisfies Assumption \ref{as:measurecond}. Specifically, we have that
\begin{align}\label{eq:NonSpikedESD}
    \mu_{W_0}\overset{N \gg 1}{\approx} \mu_{\mathrm{ASD}} , \quad \mu_{W_0} = \sum_{i=1}^N \frac{1}{N}\delta_{\lambda_i(W_0)},
\end{align}
where $\lambda_{i}(W_0)$, $i =1,2,\ldots,N,$ are the eigenvalues of $W_0$ and we assume that $\mu_{\mathrm{ASD}}$ has a density (see Theorem~\ref{thm:VESDHaar} below for a more precise statement), on a single interval, of the form
\begin{align}\label{eq:MeasureForm}
   \mu_{\mathrm{ASD}}(\mathrm{d} \lambda) = \varrho_{{\mathrm{ASD}}}(\lambda) \mathrm{d}\lambda= h_{\mathrm{ASD}}(\lambda) \mathds{1}_{[\gamma_{-},\gamma_{+}]}(\lambda)(\gamma_{+}-\lambda)^{1/2}(\lambda-\gamma_{-})^{1/2}\mathrm{d}\lambda.
\end{align}


To simplify our statements, we adopt the concept of \textit{stochastic domination} \cite{Knowles2017}.
\begin{definition}[Stochastic domination]
    (i) Let 
    \begin{align*}
        \xi = \paren{\xi^{(N)}(u): N\in \N, u\in U^{(N)}}, \quad \zeta = \paren{\zeta^{(N)}(u):N\in \N, u\in U^{(N)}},
    \end{align*}
    be two families of nonnegative random variables defined on the same probability space, where $U^{(N)}$ is a possibly $N$-dependent parameter set. We say $\xi$ is stochastically dominated by $\zeta$, uniformly in $u$, if for any fixed (small) $\epsilon>0$ and (large) $D>0$,
    \begin{align*}
        \sup_{u\in U^{(N)}} \Prb\paren{\xi^{(N)}(u)>N^{\epsilon}\zeta^{(N)}(u)}\leq N^{-D},
    \end{align*}
    for large enough $N\geq N_0(\epsilon,D)$, and we shall use the notation $\xi\prec \zeta$. If a family $\xi$ is not non-negative, then we write $\xi\prec \zeta$ or $\xi = \OO_{\prec}(\zeta)$ if $|\xi| \prec \zeta$.


    (ii) An event $\Xi$ is said to hold with overwhelming probability if for any constant $D>0$, $\Prb(\Xi)\geq 1-N^{-D}$ for sufficiently large $N$.
\end{definition}

\begin{remark}
    Stochastic domination will always be taken to be uniform in all parameters that are not explicitly fixed (such as the matrix indices, and $z$ that takes values in some compact set). Further, $N_0(\epsilon,D)$ may depend on quantities that are explicitly constant, such as $\tau_1$ in Assumption \ref{as:TechA} below. 
\end{remark}

\subsection{The deformed Marchenko-Pastur law and the asymptotics of the outliers}
We first present the deformed Marchenko– Pastur (MP) law.  The companion of the covariance matrix $W_0$ in \eqref{eq:NonSpikedModel} is denoted by
\begin{align}\label{eq:Companion}
    \bunderline{W}_0 = X^*\Sigma_0X.
\end{align}
It is well known that \cite{Knowles2017}, in general, the asymptotic density function of the ESD of $\bunderline{W}_0$ follows what is often referred to as the deformed MP law, denoted as $\mu_{\dMP}$, which is best described by its Stieltjes transform. Let $z\in \C_{+}$, the Stieltjes transform $m_{\dMP}(z)$ can be characterized as the unique solution of the equation \cite{Knowles2017},
\begin{align}
    z = f(m_{\dMP}), \quad \mathrm{Im}\, m_{\dMP}(z) > 0, \quad \mathrm{Im}\, z > 0,
\end{align}
where $f(z)$ is defined as
\begin{align}\label{eq:f_dMP}
    f(z) = -\frac{1}{z} +\frac{1}{M}\sum_{k=1}^{N}\frac{1}{z+\sigma_k^{-1}}.
\end{align}
We use $\varrho_{\dMP}$ to denote the density associated with $\mu_\dMP.$
We work within the framework imposed by the following assumptions, most of which were used in \cite{Ding2021}.
\begin{assumption}\label{as:TechA}
    \phantom{.}
    \begin{enumerate}[(a)]
    
        \item  \textbf{ On $X$ in \eqref{eq:SCM_Model}}: We assume \eqref{eq:AsymptRatio}, \eqref{eq:MomentCondition} and \eqref{eq:BddMoments} hold and that for $X=(x_{ij})$, $x_{ij}$, $1\leq i\leq N$, $1\leq j\leq M$, are iid real random variables.
        \item \textbf{On $\Sigma_0$ in \eqref{eq:Sigma0Eigen}:} We assume that for some small constant $0<\tau_1<1$,
        \begin{align}
            \tau_1 \leq \sigma_N \leq \sigma_{N-1}\leq \dots \leq \sigma_1 \leq \tau_1^{-1}.
        \end{align}
        We further assume $\Sigma_0$ is such that $\varrho_{\dMP}$ is supported on a single bulk component $\mathrm{supp}\,\mu_{\dMP} = [\gamma_{-},\gamma_{+}]$ and that there exists $\delta>0$ such that $w(x):= \varrho_{\dMP}(x)(x -\gamma_{+})^{-1/2}(x-\gamma_{-})^{-1/2}$ and $1/w(x)$ have analytic extensions to $\{z\in \C : \min_{x\in [\gamma_{-},\gamma_{+}]} |x-z|<\delta\}$ that are bounded above by a constant $D$. Moreover, we assume that
        \begin{align}
            \gamma_{\pm}\geq \tau_1, \quad \gamma_{+}-\gamma_{-}\geq \tau_1, \quad \min_i (\sigma_{i}^{-1}+m_{\dMP}(\gamma_{\pm}))\geq \tau_1.
        \end{align}
        \item \textbf{On the spikes in \eqref{eq:SigmaEigen}:} For some real fixed integer $r$ and $i\leq r$, we assume that there exists some constant $\varpi$ such that
        \begin{align}
            \tilde{\sigma}_i > -\frac{1}{m_{\dMP}(\gamma_{+})} + \varpi, \quad i\leq r.
        \end{align}
        We also assume that $\tilde{\sigma}_i, ~ 1\leq i\leq r$, are bounded.
        \item \textbf{On spike spacing:} We assume there exists $\gamma > 0$ such that for $N$ sufficiently large and $t$ sufficiently small
        \begin{align*}
            \mathbb P\left( \min_{\substack{i \neq j \\ i,j \leq r}} N^{1/2}|\lambda_j(W) - \lambda_i(W)| < t \right) \leq  t^\gamma.
        \end{align*}
    \end{enumerate}
\end{assumption} 

Recall that our model of interest \eqref{eq:SCM_Model} is formed by adding $r$ spikes to $\Sigma_0$, where $r\geq 0$ is some fixed integer.
The conditions in Assumption \ref{as:TechA}(b) rule out the existence of spikes in $\Sigma_0$ so that all possible spikes are exclusive to $\Sigma$, and also guarantee that $\varrho_{\dMP}$ displays a regular square-root behavior near the edges $\gamma_{\pm}$. For instance, this condition will be satisfied when the asymptotic spectral distribution of $\Sigma_0$ is supported on some interval $[a,b]\subset (0,\infty)$ and its density function is bounded from both above and below; see \cite[Example 2.9]{Knowles2017} or \cite[Corollary 3]{ElKaroui2007} for more details. Assumption \ref{as:TechA}(c), for reasons we will now describe, imposes the condition that each $\tilde{\sigma}_i$, $1\leq i\leq r$, generate a spike in $W$ that are at an $\OO(1)$ distance from $\gamma_+$ as $N \to \infty$ \cite{Ding2021c,Ding2019}.  Additionally, Assumption \ref{as:TechA}(d) imposes a constraint on how closely the outliers of $W$ can approach one another, ensuring that the gaps between them remain at least polynomially small with overwhelming probability.

\begin{lemma}\label{lem:EigValVecConv}
    Suppose Assumptions \ref{as:TechA}(a,b,c) hold and recall the function $f$ defined in \eqref{eq:f_dMP}. For all $1\leq i\leq r$, we have
    \begin{align}\label{eq:EigValVecConv}
        \abs{\lambda_i(W)- \gamma_i} = \OO_{\prec} (N^{-1/2}), \quad \gamma_i = f(-\tilde{\sigma}_i^{-1}), \quad \text{and} \quad \abs{ |\mathbf{u}_i ^* \mathbf{v}_i|^2-\frac{1}{\tilde{\sigma}_i}\dfrac{f^{\prime}(-\tilde{\sigma}_i^{-1})}{f(-\tilde{\sigma}_i^{-1})}} = \OO_{\prec} (N^{-1/2}),
    \end{align}
    where $\{\vec v_i\}$ and $\{\vec u_i\}$ are defined in \eqref{eq:SigmaEigen} and \eqref{eq_spectralW}, respectively.
\end{lemma}
\begin{proof}
See \cite[Theorem 3.6]{Ding2021c}.  
\end{proof}

\begin{remark}
    
    The previous lemma shows that the condition in Assumption~\ref{as:TechA}(d) holds when the spikes in $\Sigma$ are distinct. Specifically, by a simple application of the mean value theorem, we have $\gamma_i - \gamma_j = f'(\xi)(\tilde{\sigma}_i - \tilde{\sigma}_j)$, with condition (c) in Assumption \ref{as:TechA} ensuring $f'(\xi) \asymp 1$. Thus, Assumption~\ref{as:TechA}(d) holds as a consequence of the convergence of $\lambda_i(W)$, assuming $\{\gamma_i\}$ are distinct.
\end{remark}

\begin{remark}\label{rmk_detectable}
To be detectable, these spikes in $\Sigma$ should be strong, beyond the so-called BBP transition threshold \cite{Baik2005,perry2018optimality}, so that they produce the corresponding outlier eigenvalues in $W$. By Assumption~\ref{as:TechA}, the so-called spiked eigenvalues of $W$ will satisfy
\begin{align*}
    \lambda_j(W) > \gamma_+, \quad j = 1,2,\ldots,r.
\end{align*}
Our assumptions impose that these spiked eigenvalues have asymptotic locations
\begin{align*}
    |\lambda_j(W) - \gamma_j| \overset{N \gg 1}{\longrightarrow} 0,  \quad j = 1,2,\ldots,r.
\end{align*}
\end{remark}

\subsection{Local laws}\label{subsec:LocLaw}
Consider the $(N+M)\times (N+M)$ linearized matrix $\widetilde{H}$ defined as
\begin{align}\label{eq:LinMatrix}
    \widetilde{H} = \widetilde{H}(z,X):= \sqrt{z} \begin{bmatrix}
        0 & \Sigma_0^{1/2} X\\
        X^* \Sigma_0^{1/2} & 0
    \end{bmatrix}.
\end{align}
Let $\widetilde{G}_1(z) = (W_0-z)^{-1}$, $\widetilde{G}_2(z) = (\bunderline{W}_0-z)^{-1}$ and consider
\begin{align}\label{eq:NonSpikedGenRes}
    \widetilde{G}(z) = \widetilde{G}(z,X):= (\widetilde{H}-z)^{-1} = \begin{bmatrix}
        \widetilde{G}_1(z) & z^{-1/2} \widetilde{G}_1(z)\Sigma_0^{1/2}X\\
        z^{-1/2} X^* \Sigma_0^{1/2}\widetilde{G}_1(z) & \widetilde{G}_2(z)
    \end{bmatrix},
\end{align}
where the last equality follows from Schur complements. Define the deterministic approximation of $\widetilde{G}$ as
\begin{align}\label{eq:DetEquiv}
    \widetilde{\Pi}(z) := \begin{bmatrix}
        \widetilde{\Pi}_1(z) & 0 \\
        0 & \widetilde{\Pi}_2(z)
    \end{bmatrix}:= \begin{bmatrix}
        -z^{-1}(I+m_{\dMP}(z)\Sigma_0)^{-1} & 0\\
        0 & m_{\dMP}(z)
    \end{bmatrix}.
\end{align}


We state results relating to the so-called anisotropic local laws \cite{Knowles2017} for the non-spiked model. Throughout the following, we use the notation $z= \lambda+i\eta$ with $\eta>0$ for the spectral parameter $z$.
Fix some small constant $\tau>0$ and consider the set of spectral parameters 
\begin{align}
    \mathcal{D} = \mathcal{D}(\tau,M) := \{z \in \C_{+} : |z|\geq \tau, |\lambda|\leq \tau^{-1}, M^{-1+\tau}\leq \eta \leq \tau^{-1}\}.
\end{align}
Moreover, define the set $\mathcal{D}_0$ as 
\begin{align}
    \mathcal{D}_0 = \mathcal{D}_0(\tau,M) := \{z\in \C_{+} : \tau\leq \lambda\leq \tau^{-1}, 0< \eta\leq \tau^{-1}, \mathrm{dist}(\lambda,\mathrm{supp}\,\varrho_\dMP)\geq M^{-2/3+\tau}\},
\end{align}
and the control parameter
\begin{align}\label{eq:ControlParam}
    \Psi(z) := \sqrt{\frac{\mathrm{Im}~m_{\dMP}(z)}{M\eta}} + \frac{1}{M\eta} \mathds{1}(z\notin \mathcal{D}_0).
\end{align}
Importantly, we have for all $z \in \mathcal D(\tau,M)$, $\Psi(z) = \OO(M^{-\tau/2}).$
We also note that $\mathrm{Im}~m_{\dMP}(z)$ can be bounded as follows:
\begin{align}\label{eq:mImProp}
    \mathrm{Im}~m_{\dMP}(z) \asymp \begin{cases} 
    \sqrt{\kappa + \eta}, & \text{if } \lambda \in \operatorname{supp} \varrho_{\dMP}, \\ 
    \frac{\eta}{\sqrt{\kappa + \eta}}, & \text{otherwise},
    \end{cases}
\end{align}
where $\kappa:= \mathrm{dist}(\lambda,\partial (\mathrm{supp}\,{\varrho_\dMP}))$.
\begin{lemma}[Anisotropic local law]\label{lem:NonSpikedLocLaw}
    Suppose Assumption~\ref{as:TechA}(a,b) holds. For any unit deterministic vectors $\mathbf{u},\mathbf{v}\in \R^{M+N}$, we have that for all $z\in \mathcal{D}\cup \mathcal{D}_0$
    \begin{align*}              
    \abs{\mathbf{u}^*{\paren{\widetilde{G}(z)-\widetilde{\Pi}(z)}}\mathbf{v}} \prec \Psi(z),
    \end{align*}
    and therefore, for any unit deterministic vectors $\mathbf{u},\mathbf{v}\in \R^{N}$
    \begin{align*}
    \abs{\mathbf{u}^*\paren{\widetilde{G}_1(z)+\frac{1}{z}(I+m_{\dMP}(z)\Sigma_0)^{-1}}\mathbf{v}} \prec \Psi(z).
    \end{align*}
\end{lemma}
\begin{proof} 
      In \cite{Knowles2017}, for example, the authors use the alternate definition
    \begin{align*}
        \widetilde{H} =  \begin{bmatrix}
        - \Sigma_0^{-1}  + z & X\\
        X^* & 0
    \end{bmatrix}.
    \end{align*}
    With this, we note that
    \begin{align*}
    \begin{bmatrix} \sqrt{z} \Sigma_0^{1/2} & 0 \\ 0 & I \end{bmatrix} (\widetilde{H} - z) \begin{bmatrix} \sqrt{z} \Sigma_0^{1/2} & 0 \\ 0 & I \end{bmatrix} = \begin{bmatrix}
        - z I & \sqrt{z}\Sigma_0^{1/2} X\\
        \sqrt{z} X^*\Sigma_0^{1/2} & -z I 
    \end{bmatrix}.
    \end{align*}
    Based on this, and the assumptions put on $\Sigma_0$, it will suffice to analyze either matrix.
\end{proof}

\begin{remark}\label{r:rand_vec}
    It is important to note that the previous results also apply to random unit vectors $\mathbf{u}, \mathbf{v} \in \mathbb S^{N-1}$ drawn from distributions within the same probability space. Specifically, consider $z\in \mathcal{D}\cup\mathcal{D}_0$ and let $\Omega_{\epsilon,\mathbf{u},\mathbf{v}}$ denote the event where $|\mathbf{u}^*{\Sigma_0}^{-1}(\widetilde{G}(z)-\widetilde{\Pi}(z)){\Sigma_0}^{-1}\mathbf{v}| > M^{\epsilon}\Psi(z)$. Then, we have the following expression for the probability
    \begin{align*}
        \Prb(\Omega_{\epsilon,\mathbf{u},\mathbf{v}}) = \E \braces{\mathds{1}_{\Omega_{\epsilon,\mathbf{u},\mathbf{v}}}\mathds{1}_{\mathbf{u},\mathbf{v}\in S}} = \E \braces{\mathds{1}_{\mathbf{u},\mathbf{v}\in S}\E\braces{\mathds{1}_{\Omega_{\epsilon,\mathbf{u},\mathbf{v}}}|\mathbf{u},\mathbf{v}}}.
    \end{align*}
    By the uniformity of stochastic domination, $\E\braces{\mathds{1}_{\Omega_{\epsilon,\mathbf{u},\mathbf{v}}}|\mathbf{u},\mathbf{v}}\leq N^{-D}$, $N > N_0(\epsilon,D)$,  we see that Lemma~\ref{lem:NonSpikedLocLaw} also holds for such random vectors.
\end{remark}

Before we extend the local laws to the spiked model, we discuss some technical consequences of Assumption~\ref{as:TechA}(b,c).  In relation to Assumption~\ref{as:TechA}(b), consider the function
\begin{align*}
    f_i(z) = 1 + m_{\dMP}(z) \sigma_i, \quad z \in \mathcal R = \{ \lambda + \mathrm{i} \eta ~:~ \gamma_- - \delta \leq \lambda \leq \gamma_+ + \delta, 0 < \eta < \delta\},
\end{align*}
for $\delta > 0$. From the assumptions made, it follows that $f_i(z)$ extends to be uniformly $1/2$-H\"older continuous on $\overline{\mathcal R}$.  Thus, if the boundary value of $f_i(z)$ has a uniform lower bound on $[\gamma_-,\gamma_+]$, by taking $\delta$ sufficiently small, we will have a uniform lower bound on the closure of $\mathcal R$. Indeed, Assumption \ref{as:TechA}(b) implies this lower bound and therefore for a uniform constant $\tau_2$ and
\begin{align}\label{eq:fest}
\tau_2^{-1} \geq |f_i(z)| \geq \tau_2, \quad z \in \overline{\mathcal R}.
\end{align}
Concerning Assumption~\ref{as:TechA}(c), consider the function
\begin{align*}
    g_i(z) = d_i^{-1} + 1 - 1/f_i(z) = \frac{1 + m_{\dMP}(z) \tilde \sigma_i}{d_i f_i(z)}.
\end{align*}
The arguments made for $f_i$ apply to $w_i$ giving
\begin{align}\label{eq:west}
\tau_2^{-1} \geq |g_i(z)| \geq \tau_2, \quad z \in \overline{\mathcal R}.
\end{align}

\begin{remark}
    The maximum modulus principle can then be used to extend the set $\mathcal R$ in \eqref{eq:fest} and \eqref{eq:west} to be unbounded in the imaginary direction, i.e., for $0 < \eta < \infty$. 
\end{remark}

 Continuing, let 
\begin{align}\label{eq:SpikedCompanion}
    \bunderline{W} = X^*\Sigma X,
\end{align}
represent the companion matrix associated to $W$, and define $G_1(z) = (W-z)^{-1}$ and $G_2(z) = (\bunderline{W}-z)^{-1}$.
The spiked counterparts of the matrices in \eqref{eq:LinMatrix} and \eqref{eq:NonSpikedGenRes} are given by
\begin{align}
    H = H(z,X) := \sqrt{z} \begin{bmatrix}
        0 & \Sigma^{1/2}X \\
        X^*\Sigma^{1/2} & 0
    \end{bmatrix},
\end{align}
and $G(z) = (H-z)^{-1}$. Denote $\widehat{\Sigma}_0\in \R^{N+M}$ as
\begin{align*}
    \widehat{\Sigma}_0 := \begin{bmatrix}
        \Sigma_0 & 0\\
        0 & I
    \end{bmatrix},
\end{align*}
and similarly define $\widehat{\Sigma}$. Let ${V}_r$ be a matrix formed using the first $r$ spiked eigenvectors of $\Sigma$ and set ${D}_r = \mathrm{diag}(d_1,\dots,d_r)$. 
We first represent the resolvent of the spiked model in terms of its non-spiked counterpart.
\begin{lemma}\label{lem:SpikedvsNoSpikeRes}
    The resolvent $G_1(z)$ of the spiked covariance model can be expressed in terms of $\widetilde{G}_1(z)$ as follows: 
    \begin{align}\label{eq:SpikedvsSpikeRes}
        \Sigma_0^{-1/2}\Sigma^{1/2} G_1(z) \Sigma^{1/2}\Sigma_0^{-1/2} = \widetilde{G}_1(z)-z\widetilde{G}_1(z){V}_r ({D}_r^{-1}+I+z{V}_r^*\widetilde{G}_1(z){V}_r)^{-1}{V}_r^*\widetilde{G}_1(z).
    \end{align}
\end{lemma}
\begin{proof}
    See Lemma C.1 in \cite{Ding2021}.
\end{proof}

\begin{lemma}\label{lem:SpikedvsNoSpikeCompRes}
    For any vector $\mathbf{u} \in \R^M$, denote $\widetilde{\mathbf{u}} \in \R^{N+M}$ as the natural embedding of $\mathbf{u}$ such that
    \begin{align}\label{eq:uEmb}
        \widetilde{\mathbf{u}} = 
        \begin{bmatrix}
        0 \\ 
        \mathbf{u}
        \end{bmatrix}.
    \end{align}
    Moreover, denote $\widehat{V}_r \in \R^{(N+M) \times r}$ as the natural embedding of $V_r$ such that
    \begin{align}\label{eq:VEmb}
        \widehat{V}_r = 
        \begin{bmatrix}
        V_r \\ 
        0
        \end{bmatrix}.
    \end{align}
    Then we have that
    \begin{align*}
        {\mathbf{u}}^*  G_2(z) {\mathbf{v}} = 
        \mathbf{u}^* \widetilde{G}_2(z) \mathbf{v} 
        - z \widetilde{\mathbf{u}}^* \widetilde{G}(z) \widehat{V}_r 
        \left( 
        D_r^{-1} + I + z \widehat{V}_r^* \widetilde{G}(z) \widehat{V}_r
        \right)^{-1} 
        \widehat{V}_r^* \widetilde{G}(z) \widetilde{\mathbf{v}}.
    \end{align*}
\end{lemma}
\begin{proof}
    See Lemma C.2 in \cite{Ding2021}.
\end{proof}

Denote the spectral parameter set
\begin{align}\label{eq:SpikedSpecSet}
    \widetilde{\mathcal{D}} = \widetilde{\mathcal{D}}(\tau,N) = \paren{\mathcal{D}(\tau,N)\cup \mathcal{D}_0(\tau,N)}\cap \brac{\min_{1\leq i\leq r} |z-f(-\tilde{\sigma}_i^{-1})|\geq \tau}, 
\end{align}
where $\tau$ is some small fixed constant. The following lemma presents a generalized version of the local laws applicable to the spiked model.
\begin{lemma}\label{lem:SpikedLocLaw}
    Consider the eigenvectors $\{\mathbf{v}_i\}$ of $\Sigma$ and any unit deterministic vectors $\mathbf{u},\mathbf{v}\in \R^{N}$. Define $ {u}_i = \mathbf{v}_i^*\mathbf{u}$, ${v}_i = \mathbf{v}_i^*\mathbf{v}$ and
    \begin{align}\label{eq:wLdef}
        \mathcal{L}_i = \mathds{1}(i\leq r) z^{-1}(1+m_{\dMP}(z)\sigma_i)^{-2}(d_i^{-1}+1-(1+m_{\dMP}(z)\sigma_i)^{-1})^{-1}.
    \end{align}
    Assume Assumption \ref{as:TechA} holds, then for all $z\in \widetilde{\mathcal{D}}$, we have
    \begin{align}\label{eq:SpikedLawG1}
        \mathbf{u}^*G_1(z)\mathbf{v} = \sum_{i=1}^N \frac{{u}_i^*{v}_i}{1+d_i} \paren{\mathbf{v}_i^* \widetilde{G}_1(z)\mathbf{v}_i - \mathcal{L}_i} + \OO_{\prec}(\Psi(z)).
    \end{align}
    Similarly, for any deterministic vectors $\mathbf{u},\mathbf{v}\in \R^{M}$,
    \begin{align}\label{eq:SpikedLawG2}
        \mathbf{u}^*G_2(z)\mathbf{v} = \mathbf{u}^*\widetilde{G}_2(z)\mathbf{v} + \OO_{\prec}(\Psi(z)).
    \end{align}
\end{lemma}
\begin{proof}
   See Appendix \ref{sec_appendixxuxilimarylemmasproof}. 
\end{proof}

\begin{remark}\label{rmk:ESDRel}
Lemmas \ref{lem:NonSpikedLocLaw} and \ref{lem:SpikedLocLaw} establish that the VESDs of $\bunderline{W}$ and $\bunderline{W}_0$ have the same VASD regardless of the existence of spikes. Consequently, the ASDs of $W$ and $W_0$ are identical, as $\mu_{W}(\mathrm{d}x) = c_N^{-1}\mu_{\bunderline{W}}(\mathrm{d}x) + (1-c_N^{-1})\delta_0$ where $\mu_W$ and $\mu_{\bunderline{W}}$ denote the ESDs of $W$ and $\bunderline{W}$, respectively, with an analogous relation holding between $W_0$ and $\bunderline{W}_0$.
\end{remark}

\subsection{Asymptotic VESDs (VASDs)}\label{subsec:LimVESDs}

For any given deterministic vector $\mathbf{b}\in \R^{N}$, let $\mu_{\mathbf{b}}$ and $\mu_{0,\mathbf{b}}$ represent the VASDs of $W$ and $W_0$, respectively, corresponding to $\mathbf{b}$. The Stieltjes transforms of the VASDs can be characterized using Lemmas \ref{lem:NonSpikedLocLaw} and \ref{lem:SpikedLocLaw}. More specifically, we have 
\begin{align}\label{eq:LimCharact}
    m_{0,\mathbf{b}}(z) = -\frac{1}{z}\mathbf{b}^*(I+m_{\dMP}(z)\Sigma_0)^{-1}\mathbf{b}, \quad m_{\mathbf{b}}(z) = \sum_{i=1}^N \frac{\omega^2_i}{1+d_i} \paren{-\frac{1}{z}(1+m_{\dMP}(z)\sigma_i)^{-1}-\mathcal{L}_i},
\end{align}
where $\omega_i = \vec v_i^*\vec b$ and $\mathcal{L}_i$ are defined in \eqref{eq:wLdef}. Furthermore, for $z \in \widetilde {\mathcal D}$ defined in \eqref{eq:SpikedSpecSet}, we have
\begin{align*}
    |m_{W_0,\vec b}(z) - m_{0,\vec b}(z)| = \OO_{\prec} (\Psi(z)), \quad |m_{W,\vec b}(z) - m_{\vec b}(z)| = \OO_{\prec} (\Psi(z)),
\end{align*}
where $\Psi(z)$ is given in \eqref{eq:ControlParam}. Here $m_{W_0,\vec b}$ and $m_{W,\vec b}$ denote the Stieltjes transforms of the VESDs of $W_0$ and $W$, respectively, associated with $\vec b$.


\begin{lemma}\label{lem:VESDLimDet}
    Under Assumption \ref{as:TechA}, the asymptotic densities of $ \mu_{0,\mathbf{b}}$ and $\mu_{\mathbf{b}}$ satisfy Assumption \ref{as:measurecond}. Moreover, the support of $\mu_{0,\mathbf{b}}$ and $\mu_{\mathbf{b}}$ are given by
    \begin{align}
        \mathrm{supp}(\mu_{0,\mathbf{b}}) = \braces{\gamma_{-},\gamma_{+}}, \quad \text{and} \quad \mathrm{supp}(\mu_{\mathbf{b}}) = \braces{\gamma_{-},\gamma_{+}}\cup  P, \quad P \subset \brac{\gamma_i}_{i=1}^r,
    \end{align}
    where $\gamma_i$ is defined in \eqref{eq:EigValVecConv}.
\end{lemma}
\begin{proof}
   See Appendix \ref{sec_appendixxuxilimarylemmasproof}. 
\end{proof}

Next, we analyze the asymptotic behavior of VESDs at a uniform vector on the unit hypersphere $\mathbb S^{N-1}$. 
    Define $f(\mathbf{b}) = \mathbf{b}^* A\mathbf{b}$ and observe that $f$ is differentiable, with its gradient given by $\nabla f = (A + A^*)\mathbf{b} $. Therefore, $f$ is $2\norm{A}_2$-Lipschitz on $\mathbb{S}^{N-1}$ as 
    \begin{align*}
        \norm{f}_{\mathrm{Lip}}\leq \norm{\nabla f}_{\infty}\leq 2\norm{A}.
    \end{align*}
From the concentration of Lipschitz functions on the hypersphere \cite[Theorem 5.1.4]{Vershynin2018}, there exists an absolute constant $c$ such that if $\vec b$ is uniformly distributed on $\mathbb S^{N-1}$ then
    \begin{align}\label{eq:ConcUnifSphere}
        \Prb\paren{\abs{\mathbf{b}^* A\mathbf{b}-\frac{1}{N}\mathrm{tr}(A)}\geq t}\leq 2 \exp \paren{\frac{-cNt^2}{\norm{A}^2}}.
    \end{align}

\begin{theorem}\label{thm:VESDHaar}
    Let $\mathbf{b}$ be a uniform vector on the unit hypersphere $\mathbb S^{N-1}$.
    The Stieltjes transforms of $\mu_{\mathbf{b}}$ and $\mu_{0,\mathbf{b}}$ satisfy
    \begin{align*}
        |m_{\mathbf{b}}(z) - m_{0,\mathbf{b}}(z)| =\OO_{\prec} (N^{-1/2}),
    \end{align*}
    for $z \in \widetilde {\mathcal D}$.  Furthermore, for $\mu_{\mathrm{ASD}}(\mathrm d \lambda) = \varrho_{\mathrm{ASD}}(\lambda) \mathrm d \lambda$ with 
    \begin{align}\label{eq:DensHaar}
         \varrho_{\mu_{\mathrm{ASD}}}(\lambda) = \frac{\varrho_{\dMP}(\lambda)}{\lambda} \sum_{i=1}^N \sigma_i\braces{1+2\mathrm{Re}~m_{\dMP}(\lambda+\ri0^+)\sigma_i + |m_{\dMP}(\lambda+\ri0^+)|^2 \sigma_i^2}^{-1}, 
    \end{align}
    we have
    \begin{align*}
        |m_{0,\mathbf{b}}(z) - m_{\rm ASD}(z)| =\OO_{\prec} (N^{-1/2}),
    \end{align*}
    for $z \in \mathcal D\cup \mathcal{D}_0$. Lastly, $\mu_{\rm ASD}$ satisfies Assumption~\ref{as:measurecond} with the same support as $\varrho_{\dMP}$, provided that Assumption \ref{as:TechA} holds.
\end{theorem}
\begin{proof}
We establish the last statement first.  From \eqref{eq:fest}
    \begin{align}\label{eq:DetEqBound}
        |1+m_{\dMP}(z)\sigma_i|\geq \tau_2 \quad \text{for all } z\in \mathcal{D}\cup\mathcal{D}_0. 
    \end{align}
    Given that $|z|\geq \tau$ on $\mathcal{D}\cup \mathcal{D}_0$, we find that
    \begin{align}\label{eq:DetEqNormBound}
        \norm{\frac{1}{z}\paren{1+m_{\dMP}(z)\Sigma_0}^{-1}}_2 \leq  \tau^{-1}\tau_2^{-1}.
    \end{align}
    Thus, using the concentration results for uniform vectors on the unit hypersphere, we have
    \begin{align*}
        \left|m_{0,\mathbf{b}}(z) + \frac{1}{N} \sum_{j=1}^N z^{-1}(1+m_{\dMP}(z)\sigma_i)^{-1} \right| =\OO_{\prec} (N^{-1/2}).
    \end{align*}

    Now observe that
    \begin{align*}
        m_{\vec b}(z) - m_{0,\mathbf{b}}(z) = \sum_{i=1}^r \frac{\omega^2_i d_i}{1+d_i} \paren{-\frac{1}{z}(1+m_{\dMP}(z)\sigma_i)^{-1}-\mathcal{L}_i}.
    \end{align*}
    The claim follows from $\omega_i^2 =\OO_{\prec} (N^{-1/2})$ and \eqref{eq:fest}, \eqref{eq:west} which imply that the terms in parentheses are uniformly bounded in $\widetilde {\mathcal D}$.
\end{proof}

\section{Analysis of the pilot algorithms}\label{sec:VESDLan}
In this section, we examine the pilot algorithms introduced in Section \ref{sec:pilotestimation}. The analysis reduces to studying the asymptotic behavior of Jacobi matrices associated with VESDs and their Cholesky factors, interpreting the measures as perturbations of their deterministic asymptotic approximations.



\subsection{Perturbation analysis of Jacobi matrices and analysis of Algorithm \ref{Ea:realVESDA}} Let $\mu$ be a measure supported on a single interval, with a finite number of spikes, and satisfying Assumption \ref{as:measurecond}. Let $\nu$ represent a perturbed (and possibly random) version of $\mu$.  The following result from \cite{Ding2021} establishes the relation between perturbed and unperturbed Jacobi matrices, as well as their Cholesky factors, asymptotically in terms of the difference of the Stieltjes transform
\begin{align*}
    m(z,\mu-\nu) = \int_{\R} \frac{(\mu-\nu)(\mathrm{d}x)}{x-z}.
\end{align*}

\begin{theorem}\label{thm:PertJacobCholAsympt}
    Let $N$ be a positive integer and suppose $\mu = \mu(N)$ satisfies Assumption \ref{as:measurecond} for sufficiently large $N$. Suppose further that a measure $\nu = \nu(N)$ is such that $\nu - \sum_{i=1}^{p}w_j \delta_{c_j},$ has its support inside $\Gamma = \Gamma(\eta)$, where $\Gamma(\eta)$ is the rectangle that is a distance $\eta$ from $[a,b]$ for some $\eta>0$, and assume $\|m(z,\mu-\nu)\|_{L^{\infty}(\Gamma)}\leq E(N,\eta)$. Recall (\ref{eq:JacobOPDef}) and (\ref{eq:CholOpDef}). If $n\leq C \eta^{-1/2}$, $C>0$ and $\eta = \eta(N)$ is such that $E(N,\eta)\eta^{-1}\to 0$ as $N\to \infty$, then\footnote{Note that the factor $\eta^{-1/2}$ in \cite[Theorem 2.4]{Ding2021b} should be $\eta^{-1}$.}
    \begin{align*}
        a_n(\nu) = a_n(\mu) + \OO\paren{E(N,\eta)\eta^{-1}}, \quad b_n(\nu) = b_n(\mu) + \OO\paren{E(N,\eta)\eta^{-1}},
    \end{align*}
    and
    \begin{align*}
        \alpha_n(\nu) = \alpha_n(\mu) + \OO\paren{E(N,\eta)\eta^{-1}}, \quad \beta_n(\nu) = \beta_n(\mu) + \OO\paren{E(N,\eta)\eta^{-1}}.
    \end{align*}
\end{theorem}

Given that the VESD $\mu_{W,\mathbf{b}}$ is a discrete measure with $N$ positive support points, we recall that it has a Jacobi matrix $\mathcal{J}(\mu_{W,\mathbf{b}})$, defined similarly to \eqref{eq:JacobOPDef}, but with dimension $N$ rather than being semi-infinite. Moreover $\mathcal{J}(\mu_{W,\mathbf{b}})$ is invertible and its Cholesky factor, $\mathcal{L}(\mu_{W,\mathbf{b}})$, can likewise be defined as in \eqref{eq:CholOpDef}. While Theorem~\ref{thm:PertJacobCholAsympt} suggests a method for obtaining the asymptotics of $\mathcal{J}(\mu_{W,\mathbf{b}})$ and $\mathcal{L}(\mu_{W,\mathbf{b}})$, its perturbation result does not apply directly to compare $\mu_{W,\mathbf{b}}$ to $\mu_{\mathbf{b}}$ because point masses away from the support of the density of $\mu_{\mathbf{b}}$ do not coincide. To address this issue, we define
\begin{align}\label{eq:hatmu}
    \widehat \mu_{\vec b} = \mu_{\vec b} - \mu_{\rm{Disc}}, \quad \mu_{\rm Disc}:=\sum_{j=1}^r v_j \delta_{\gamma_j} - \sum_{j=1}^r \abs{ \mathbf{u}_i^*\mathbf{b}}^2 \delta_{\lambda_j(W)},
\end{align}
where $\{\vec u_i\}_{i=1}^N$ are the eigenvectors of $W$ and the weights $v_j$ are chosen such that $\mu_{\vec b} - \sum_{j=1}^r v_j \delta_{\gamma_j}$ has a density and no point masses. It is important to highlight that $\widehat \mu_{\mathbf{b}}$ shares the same deterministic continuous density as $\mu_{\mathbf{b}}$, while incorporating the random outliers of $\mu_{W,\mathbf{b}}$. Additionally, we define 
\begin{align}\label{eq:checkmu}
    \widecheck \mu_{\vec b} = \frac{1}{\int_\R \widehat \mu_{\vec b}(\mathrm{d}\lambda)}\widehat \mu_{\vec b},
\end{align}
and observe that $\widecheck \mu_{\vec{b}}$ and $\widehat\mu_{\vec{b}}$ share identical Jacobi and Cholesky matrices. The measure $\widecheck \mu_{\vec b}$ represents the spiked VASD that connects the VESD and the VASD.
Given the structure of $\widehat \mu_{\vec b}$, the asymptotics of $\mathcal{J}(\widecheck\mu_{\mathbf{b}})$ and $\mathcal{L}(\widecheck\mu_{\mathbf{b}})$ can be effectively characterized.
\begin{corollary}\label{cor:checkmuAsympt}
    Let $\mathbf{b}$ be a unit vector and assume the spiked covariance matrix $W$ defined in \eqref{eq:SCM_Model} satisfies Assumption \ref{as:TechA}. Consider the stochastic measure $\widecheck\mu_{\mathbf{b}}$ defined in \eqref{eq:checkmu} and suppose that $\abs{ \mathbf{u}_i^*\mathbf{b}}^2>N^{-\sigma}$ for some $\sigma>0$ with overwhelming probability, where $\{\vec u_i\}$ are the eigenvectors of $W$. Then there exists $\kappa>0$ that depends only on $\sigma$ and $D,\delta,\tau_1$ from Assumption \ref{as:TechA}, such that
    \begin{align*}
        a_n(\widecheck\mu_{\mathbf{b}}) = \frac{\gamma_{+}+\gamma_{-}}{2}+\OO_{\prec}(e^{-\kappa n}), \quad b_n(\widecheck\mu_{\mathbf{b}}) = \frac{\gamma_{+}-\gamma_{-}}{4}+\OO_{\prec}(e^{-\kappa n}),
    \end{align*}
    and
    \begin{align*}
        \alpha_n(\widecheck\mu_{\mathbf{b}}) = \frac{\sqrt{\gamma_{-}}+\sqrt{\gamma_{+}}}{2}+\OO_{\prec}(e^{-\kappa n}), \quad \beta_n(\widecheck\mu_{\mathbf{b}}) = \frac{\sqrt{\gamma_{+}}-\sqrt{\gamma_{-}}}{2}+\OO_{\prec}(e^{-\kappa n}).
    \end{align*}
\end{corollary}
\begin{proof}
    By Lemma \ref{lem:VESDLimDet} and the condition on $\abs{ \mathbf{u}_i^*\mathbf{b}}$, it follows that $\widecheck\mu_{\mathbf{b}}$ satisfies Assumption \ref{as:measurecond} with overwhelming probability. The asymptotics are an immediate consequence of Theorem \ref{thm:JacobCholAsympt}.
\end{proof}

The analysis of Algorithm \ref{Ea:realVESDA} is summarized in the following corollary. It connects the asymptotic behavior of the Jacobi and Cholesky entries of the VESD with those of the spiked VASD $\widecheck\mu_{\mathbf{b}}$ by treating the VESD as a perturbation of it. Using the conventions in (\ref{eq:JacobOPDef}) and (\ref{eq:CholOpDef}), for $\widehat{\alpha}_n, \widehat{\beta}_n$ in (\ref{eq:L00}), we have that
\begin{equation*}
\widehat{\alpha}_n \equiv \alpha_n(\mu_{W, \mathbf{b}}),  \  \widehat{\beta}_n \equiv \beta_n(\mu_{W, \mathbf{b}}).    
\end{equation*}
\begin{corollary}\label{cor:checkmuPert}
    Let $\mathbf{b}$ be a unit vector and suppose the spiked matrix $W$ defined in \eqref{eq:SCM_Model} satisfies Assumption \ref{as:TechA}. Further suppose that $\abs{ \mathbf{u}_i^*\mathbf{b}}^2>N^{-\sigma}$ for some $\sigma>0$ with overwhelming probability, where $\{\vec u_i\}$ are the eigenvectors of $W$. Then, for $n\ll N^{1/6}$, we have
    \begin{align*}
        a_n(\mu_{W,\mathbf{b}}) = a_n(\widecheck\mu_{\mathbf{b}}) + \OO_{\prec}(N^{-1/2}n^3), \quad b_n(\mu_{W,\mathbf{b}}) = b_n(\widecheck\mu_{\mathbf{b}}) + \OO_{\prec}(N^{-1/2}n^{3}),
    \end{align*}
    and
    \begin{align*}
        \alpha_n(\mu_{W,\mathbf{b}}) = \alpha_n(\widecheck\mu_{\mathbf{b}}) + \OO_{\prec}(N^{-1/2}n^{3}), \quad
        \beta_n(\mu_{W,\mathbf{b}}) = \beta_n(\widecheck\mu_{\mathbf{b}}) + \OO_{\prec}(N^{-1/2}n^{3}).
    \end{align*}
\end{corollary}
\begin{proof}
    From the local laws (c.f. Theorems~\ref{lem:NonSpikedLocLaw} and \ref{lem:SpikedLocLaw}) and \eqref{eq:mImProp}, we have $\|m(z,\widehat \mu_{\vec b}-\mu_{W,\vec b})\|_{L^\infty (\Gamma)} \prec N^{-1/2}\eta^{-1/2},$ where 
    \begin{align*}
        \Gamma=\Gamma(\eta) &= (\braces{\gamma_{-}-\eta,,\gamma_{+}+\eta}+\ri\eta)\cup(\braces{\gamma_{-}-\eta,\gamma_{+}+\eta}-\ri\eta)\\
        &\quad \cup (\gamma_{+}+\eta+\ri\braces{-\eta,\eta})\cup (\gamma_{-}-\eta +\ri\braces{-\eta,\eta}).
    \end{align*}
    The statement of the corollary follows from Theorem~\ref{thm:PertJacobCholAsympt} by choosing $\eta\gg N^{-1/3}$ and noting $\widecheck\mu_{\vec b}$ also satisfies Assumption \ref{as:measurecond} with overwhelming probability.
\end{proof}

\begin{remark}
    It is important to note that $W$ may have eigenvalues exceeding $\gamma_{+}$. Indeed, it is well established in the literature \cite{Fan2022} that the largest eigenvalues within the bulk fluctuate around $\gamma_{+}$ on a scale of $N^{-2/3}$ and follow the Tracy-Widom distribution. However, these eigenvalues do not need to be included in $\widecheck{\mu}_{\vec{b}}$, as they remain within $\Gamma$ with overwhelming probability. The only spiked eigenvalues accounted for in $\widecheck{\mu}_{\vec{b}}$ are the top $r$, which concentrate around the asymptotic spikes $\{\gamma_j\}_{j=1}^k$. This distinction will play a critical role in Section~\ref{sec:Estimation} in identifying the spikes.
\end{remark}

\subsection{Analysis of Algorithm \ref{Ea:realVESD}}

The following theorem establishes the accuracy and consistency of the estimators in Algorithm \ref{Ea:realVESD} as $N \to \infty$. Moreover, it justifies that the number of Lanczos iterations required is given by $n = \lceil C \log N \rceil$ for $C$ sufficiently large.  

\begin{theorem}\label{thm:ConvRVESDEstim}
    Suppose that the spiked covariance matrix $W$, as defined in \eqref{eq:SCM_Model}, satisfies Assumption \ref{as:TechA}, and let $\vec b \in \mathbb S^{N-1}$.
    Consider the estimators $\widehat{\gamma}_{\pm}$, $\widehat{m}_0(z)$, from Algorithm \ref{Ea:realVESD}. Then there exists $C> 0$ such that if Lanczos is run for $n = \lceil C \log N \rceil $ steps the estimators satisfy
    \begin{align}
        |\widehat{\gamma}_{\pm}-\gamma_{\pm}| \prec N^{-1/2}, \quad |\widehat{m}_{0}(z)-\widecheck m_{\mathbf{b}}(z)| \prec \frac{ N^{-1/2}}{\mathrm{Im}^2z},
    \end{align}
    where $\widecheck m_{\vec b}(z)$ is the Stieltjes transform of $\widecheck \mu_{\vec b}$ defined in \eqref{eq:checkmu}.
\end{theorem}
\begin{proof}
    
    We first show that the Cholesky factors $\{\widecheck\alpha_i,\widecheck\beta_i\}$ and $\{\widehat{\alpha}_i,\widehat{\beta}_{i}\}$ associated with $\widecheck{\mu}_{\mathbf{b}}$ and $\mu_{W,\mathbf{b}}$, respectively, are bounded from above, with overwhelming probability. From the structure of the Cholesky decomposition, we have $\widecheck\alpha_0^2 = \widecheck a_0$ and $\widecheck\alpha_i^2+\widecheck\beta_i^2 = \widecheck a_i$ where $\{\widecheck a_i\}$ are the diagonal entries of $\mathcal{J}(\widecheck{\mu}_{\mathbf{b}})$. Using the definition of $\widecheck\mu_{\mathbf{b}}$, we know that $\|\mathcal{J}(\widecheck{\mu}_{\mathbf{b}})\|_2 \prec \max_{1 \leq i \leq r} f(-\tilde{\sigma}_i^{-1}) := K^2$. Since $\widecheck a_i \leq \|\mathcal{J}(\widecheck{\mu}_{\mathbf{b}})\|_2$, it follows directly that $|\widecheck\alpha_i|, |\widecheck\beta_i| \prec K$. Moreover, Lemma \ref{lem:EigValVecConv} shows that $\|\mathcal{J}(\mu_{W,\mathbf{b}})\|_2 = \|\mathcal{J}(\widecheck{\mu}_{\mathbf{b}})\|_2 + \OO_{\prec}(N^{-1/2})$ and by a similar argument $|\widehat{\alpha}_i|,|\widehat{\beta}_i|$ are bounded by $K$ with overwhelming probability.

    The relation between the sequences $\{\widehat{\alpha}_i, \widehat{\beta}_i\}$ and $\{\widecheck\alpha_i, \widecheck\beta_i\}$, along with the exponential convergence of $\{\widecheck\alpha_i, \widecheck\beta_i\}$, can be used to establish the first probabilistic result. In particular, we have
    \begin{align*}
        |\widehat{\gamma}_{+}-\gamma_{+}| = \abs{(\widehat{\alpha}_{n-2}+\widehat{\beta}_{n-2})^2-(\alpha+\beta)^2} \leq C_1 \paren{|\widehat{\alpha}_{n-2}-\alpha|+|\widehat{\beta}_{n-2}-\beta|},
    \end{align*}
    and using Corollaries \ref{cor:checkmuAsympt} and \ref{cor:checkmuPert}, we find 
    \begin{align}\label{eq:ErrBoundnm2}
    |\widehat{\alpha}_{n-2} - \alpha| \leq |\widehat{\alpha}_{n-2} - \widecheck\alpha_{n-2}| + |\widecheck\alpha_{n-2} - \alpha| =\OO_{\prec} (n^3 N^{-1/2}+e^{-\kappa n}),
    \end{align}
    and a similar bound holds for $|\widehat{\beta}_{n-2}-\beta|$. Thus, when $n = \lceil C \log N \rceil $ with $C > \frac{1}{2\kappa}$, it follows that $|\widehat{\gamma}_{\pm} - \gamma_{\pm}| = \OO_{\prec} (N^{-1/2})$.

    To examine the convergence of $\widehat m_0(z)$, we consider 
    \begin{align}\label{eq:L0}
        \mathcal{L}_0 = \begin{bmatrix}
            \widehat{\alpha}_0 & & & & & & \\
            \widehat{\beta}_0 & \ddots & & & & & \\
            & \ddots & \widehat{\alpha}_{n-3} & & & & \\
            & & \widehat{\beta}_{n-3} & \widehat{\alpha}_{n-2} & & &\\
            & & & \widehat{\beta}_{n-2} & \widehat{\alpha}_{n-2} & & \\
            & & & & \widehat{\beta}_{n-2} & \ddots & \\
            & & & & & \ddots & \ddots
        \end{bmatrix}, ~ \mathcal{L} = \begin{bmatrix}
            \widecheck\alpha_0 & & & & & & \\
            \widecheck\beta_0 & \ddots & & & & & \\
            & \ddots & \widecheck\alpha_{n-3} & & & & \\
            & & \widecheck\beta_{n-3} & \alpha & & &\\
            & & & \beta & \alpha & & \\
            & & & & \beta & \ddots & \\
            & & & & & \ddots & \ddots
        \end{bmatrix},
    \end{align}
    and define $m(z) = \vec e_1^*(\mathcal{L}\mathcal{L}^*-z)^{-1}\vec e_1$ where we emphasize that $\mathcal L_0$ is constant on the diagonal from the $(n-1)$th entry onwards. Note that 
    \begin{align*}
        m_0(z) = \vec e_1^*(\mathcal{L}_0\mathcal{L}_0^*-z)^{-1}\vec e_1 \quad \text{and} \quad \widecheck m_{{\mathbf{b}}}(z) = \vec e_1^*\paren{\widecheck{\mathcal{L}}\widecheck{\mathcal{L}}^*-z}^{-1}\vec e_1,
    \end{align*}
    where $\widecheck{\mathcal{L}} = \mathcal{L}(\widecheck{\mu}_{\mathbf{b}})$. Our argument is structured in two main steps: first, we show that $m(z)$ approximates $\widecheck m_{\vec b}(z)$; then we demonstrate that the estimator $\widehat m_0(z)$ is relatively close to $m(z)$.
    
    Observe that
    \begin{align*}
        |m(z)-\widecheck m_{{\mathbf{b}}}(z)| &= \abs{\vec e_1^* \paren{(\mathcal{L}\mathcal{L}^*-z)^{-1}-(\widecheck{\mathcal{L}}\widecheck{\mathcal{L}}^*-z)^{-1}}\vec e_1}\\ &\leq \|(\mathcal{L}\mathcal{L}^*-z)^{-1}\|_2~\|\mathcal{L}\mathcal{L}^*-\widecheck{\mathcal{L}}\widecheck{\mathcal{L}}^*\|_2~ \|(\widecheck{\mathcal{L}}\widecheck{\mathcal{L}}^*-z)^{-1}\|_2,
    \end{align*}
    where the inequality follows from the second resolvent identity. The resolvents can be easily bounded as
    \begin{align*}
        \| (\widecheck{\mathcal{L}}\widecheck{\mathcal{L}}^*-z)^{-1} \|_2 \leq \frac{1}{\mathrm{Im}\,z} \quad \text{and} \quad \|(\mathcal{L}\mathcal{L}^{*}-z)^{-1}\|_2 \leq \frac{1}{\mathrm{Im}\,z},
    \end{align*}
    for $z\in \C^{+}$. On the other hand, we have
    \begin{align*}
        \| \mathcal{L}\mathcal{L}^* - \widecheck{\mathcal{L}}\widecheck{\mathcal{L}}^* \|_2 \leq \| \mathfrak{a} \|_{\infty} + 2 \| \mathfrak{b} \|_{\infty},
    \end{align*}
    where $\mathfrak{a}$ and $\mathfrak{b}$ represents the vectors of diagonal and off-diagonal entries of $\mathcal{L}\mathcal{L}^* - \widecheck{\mathcal{L}}\widecheck{\mathcal{L}}^*$, respectively. Using the structure of $\mathcal{L}$ and $\widecheck{\mathcal{L}}$, along with Corollary \ref{cor:checkmuAsympt}, we find that
    \begin{align*}
        \|\mathfrak{a}\|_{\infty} \leq \sup_{n-2\leq i <\infty} |\widecheck\alpha_{i}^2-\alpha^2|+ \sup_{n-2\leq i <\infty} |\widecheck\beta_i^2 - \beta^2 | = \OO_{\prec}(e^{-\kappa n}), 
    \end{align*}
    and
    \begin{align*}
        \|\mathfrak{b}\|_{\infty} \leq \sup_{n-2\leq i<\infty} |\widecheck\alpha_{i}\widecheck\beta_{i}-\alpha\beta| = \OO_{\prec}(e^{-\kappa n}).
    \end{align*}
    Using the logarithmic lower bound on $n$, we find $\|\mathcal{L}\mathcal{L}^* - \widecheck{\mathcal{L}}\widecheck{\mathcal{L}}^* \|_2 = \OO_{\prec}(N^{-1/2})$. We conclude that
    \begin{align}\label{eq:ErrB1}
        |m(z)-\widecheck m_{{\mathbf{b}}}(z)|= \OO_{\prec}\paren{ \frac{ N^{-1/2}}{\mathrm{Im}^2z}}.
    \end{align}
    
    Employing the same methods, we now demonstrate that $\widehat{m}_0(z)$ is approximated well by $m(z)$ with overwhelming probability. Again using the second resolvent identity and the resolvent bounds, we have
    \begin{align*}
        |m(z)-\widehat{m}_{0}(z)| \leq \frac{1}{\mathrm{Im}^2z}\|\widecheck{\mathcal{L}}\widecheck{\mathcal{L}}^*-\mathcal{L}_0\mathcal{L}_0^*\|_2.
    \end{align*}
    The right-hand side can be further bounded using
    \begin{align*}
        \| \widecheck{\mathcal{L}}\widecheck{\mathcal{L}}^*-\mathcal{L}_0\mathcal{L}_0^* \|_2 \leq \|\mathfrak{a}_0\|_{\infty} + 2\|\mathfrak{b}_0\|_{\infty},
    \end{align*}
    where $\mathfrak{a}_0$ and $\mathfrak{b}_0$ are the vectors of diagonal and off-diagonal entries of $\widecheck{\mathcal{L}}\widecheck{\mathcal{L}}^*-\mathcal{L}_0\mathcal{L}_0^*$. The vectors $\mathfrak{a}_0$ and $\mathfrak{b}_0$ can be easily bounded as
    \begin{align*}
        \|\mathfrak{a}_0\|_{\infty} \leq |\widecheck\alpha_0^2-\widehat{\alpha}_0^2| + |\alpha^2-\widehat{\alpha}_{n-2}^2| + |\beta^2-\widehat{\beta}_{n-2}^2| + \max_{1\leq i\leq n-3} |\widecheck\alpha_i^2-\widehat{\alpha}_i^2| + |\widecheck\beta_{i-1}^2-\widehat{\beta}_{i-1}^2| = \OO_{\prec} (N^{-1/2}),
    \end{align*}
    and 
    \begin{align*}
        \|\mathfrak{b}_0\|_{\infty}\leq |\alpha\beta-\widehat{\alpha}_{n-2}\widehat{\beta}_{n-2}| + \max_{0\leq i\leq n-3} |\widecheck\alpha_i\widecheck\beta_{i}-\widehat{\alpha}_i\widehat{\beta}_{i}| =\OO_{\prec} (N^{-1/2}),
    \end{align*}
    where we used \eqref{eq:ErrBoundnm2}, Corollaries \ref{cor:checkmuAsympt}, \ref{cor:checkmuPert} and the fact that $n= \frac{\log N}{2\kappa}$. We deduce that 
    \begin{align}\label{eq:ErrB2}
        |m(z)-\widehat{m}_0(z)| =\OO_{\prec} \paren{\frac{N^{-1/2}}{\mathrm{Im}^2 z}}.
    \end{align}
    Finally, the consistency of $\widehat m_0(z)$ follows by combining \eqref{eq:ErrB1} and \eqref{eq:ErrB2}.

    \end{proof}

\section{Statistical consistency of our proposed estimator}\label{sec:Estimation}
In this section, we demonstrate the performance and consistency of our estimators for the ASD of $W$ and the spikes in the covariance matrix. We begin by proving that our estimator for the Stieltjes transform of the ASD is robust, meaning that it converges to the true Stieltjes transform as $N\to \infty$.
\begin{theorem}\label{thm:SolveP1}
    Assume that the sample covariance matrix $W$, described in \eqref{eq:SCM_Model}, satisfies Assumption~\ref{as:TechA} and consider the estimators $\widehat{\gamma}_{\pm}$ and $\widehat{ m}_0$, as defined in Algorithm \ref{Ea:ESD} for $k = 1$. Then there exists $C> 0$ such that if Lanczos is run for $n = \lceil C \log N \rceil $ steps the estimators satisfy
    \begin{align}
        |\widehat{\gamma}_{\pm}-\gamma_{\pm}|\prec N^{-1/2}, \quad |\widehat{m}_{0}(z)-m_{\rm{ASD}}(z)| \prec \frac{ N^{-1/2}}{\mathrm{Im}^2z}, \quad z \in \widetilde {\mathcal D}.
    \end{align}
\end{theorem}
\begin{proof}
    The estimate on $\widehat \gamma_\pm$ follows using the same argument as in Theorem~\ref{thm:ConvRVESDEstim}. Using \eqref{eq:hatmu} and \eqref{eq:checkmu}, we have
    \begin{align*}
        \widecheck m_{{\vec b}}(z) = \frac{1}{1+\Delta} \widehat m_{{\vec b}}(z) = \frac{1}{1+\Delta} (m_{{\vec b}}(z) - m_{{\mathrm{Disc}}}(z)), \quad \text{where} \quad  \Delta = \sum_{j=1}^r (\abs{ \mathbf{u}_i^*\mathbf{b}}^2 - v_j).
    \end{align*}
    The concentration bound in \eqref{eq:ConcUnifSphere} gives $\abs{ \mathbf{u}_i^*\mathbf{b}}^2\prec N^{-1/2}$, which in turn implies $\Delta \prec N^{-1/2}$. Combining this with the fact that $|m_{{\mathrm{Disc}}}(z)|\prec \frac{1}{\mathrm{Im}z}N^{-1/2}$, we find 
     \begin{align}\label{eq:ErrB3}
         |\widecheck m_{{\vec b}}(z) - m_{\vec b}(z)| =\OO_{\prec} \paren{\frac{N^{-1/2}}{\im z}}.
     \end{align}
    Now, observe that
    \begin{align*}
        |\widehat m_0(z) - m_{\mathrm{ASD}}(z)| &\leq |\widehat m_0(z)-\widecheck m_{\vec b}(z)| + |\widecheck m_{\vec b}(z)-m_{\vec b}(z)|\\ & \quad + |m_{\vec b}(z) - m_{0,\vec b}(z)| + |m_{0,\vec b}(z)-m_{\mathrm{ASD}}(z)|,
    \end{align*}
    where $m_{0,\vec b}$ is defined in \eqref{eq:LimCharact}.
     Finally, the bound for $|\widehat m_0(z) - m_{\mathrm{ASD}}(z)|$ follows by combining the bounds in \eqref{eq:ErrB3}, Theorems~\ref{thm:VESDHaar} and Theorem~\ref{thm:ConvRVESDEstim}.
    
\end{proof}


Next, we analyze the consistency of our proposed methods for determining both the number $r$ and the positions of the spikes in the spectrum of $W$, without the need to compute the true eigenvalues of the covariance matrix $W$.

\begin{theorem}[Solution of \MainProblem]\label{thm:SolveP23}
Suppose that the sample covariance matrix $W$, as given in \eqref{eq:SCM_Model}, satisfies Assumption~\ref{as:TechA}.
With $k = 1$,  suppose Lanczos is run for $n = \lceil C \log N \rceil $ steps for $C >0$ sufficiently large. Let $\widehat m_0(z)$ be the output of Estimation Algorithm~\ref{Ea:ESD} and fix $0 < \delta < 1/2$.  Let $\widehat r$ be the number of poles of $\widehat m_0(z)$ for $z > \widehat \gamma_+ + N^{-\delta}$  and let $\widehat \gamma_j$, $j = 1,2,\ldots,\widehat r$ be their locations. Then for every $D > 0$ there exists $N_0(D)$ such that for $N\geq N_0(D)$,
\begin{align*}
   \mathbb P (\widehat r \neq r) \leq  N^{-D},
\end{align*}
and therefore, by choosing $\delta$ arbitrarily close to $1/2$,
\begin{align*}
    \max_{1 \leq j \leq r} | \widehat \gamma_j -\gamma_j| \prec N^{-1/2}.
\end{align*}

\end{theorem}
\begin{proof}
    Consider the stochastic measure $\widecheck\mu_{\vec b}$ introduced in \eqref{eq:checkmu} and define the set 
    \begin{align*}
        \Omega = \{z \in \mathbb C ~:~ \| (\mathcal J(\widecheck\mu_{\vec b}) - z)^{-1} \| \leq N^{\delta} \}, \quad \delta < 1/2,
    \end{align*}
    which is nothing more that the set of all $z \in \mathbb C$ that are a distance at least $N^{-\delta}$ from the spectrum of $\mathcal J(\widehat{\mu}_{\vec b})$.
Recall \eqref{eq:L0} where now $\{\widehat \alpha_i,\widehat \beta_i\}_{i=0}^{n-2}$ are the Cholesky entries in Estimation Algorithm \ref{Ea:ESD}.  We estimate
    \begin{align*}
        \|\mathcal J(\widecheck\mu_{\vec b}) - \mathcal L_0 \mathcal L_0^*\| \leq \|\mathcal J(\widecheck\mu_{\vec b}) - \mathcal L \mathcal L^*\| + \|  \mathcal L \mathcal L^* - \mathcal L_0 \mathcal L_0^*\|.
    \end{align*}
    We have, from previous considerations, that $\widecheck\mu_{\vec b}$ satisfies Assumption~\ref{as:measurecond}, with fixed constants, with overwhelming probability due to Lemma~\ref{lem:EigValVecConv}, and the fact that $| \vec u_j^* \vec b |^2$ is chi-distributed.  So, on a set of overwhelmingly large probability
    \begin{align*}
        \|\mathcal J(\widecheck\mu_{\vec b}) - \mathcal L \mathcal L^*\|  \leq C e^{-\kappa n}.
    \end{align*}
    Using the argument of Theorem~\ref{thm:PertJacobCholAsympt}, we find
    \begin{align*}
        \|\mathcal J(\widecheck\mu_{\vec b}) - \mathcal L_0 \mathcal L_0^*\| \prec N^{-1/2} n^3 + e^{-\kappa n}.
    \end{align*}
    Choosing $n = \lceil c \log N \rceil $ for $c$ sufficiently large, we find
    \begin{align*}
        \|\mathcal J(\widecheck\mu_{\vec b}) - \mathcal L_0 \mathcal L_0^*\| \prec N^{-1/2}.
    \end{align*}
    Then, for $z \in \Omega$, we have
    \begin{align*}
        \|\mathcal J(\widecheck\mu_{\vec b}) - \mathcal L_0 \mathcal L_0^*\|\| (\mathcal J(\widecheck\mu_{\vec b}) - z)^{-1} \| \prec N^{-1/2 + \delta},
    \end{align*}
    which implies that $\| (\mathcal L_0 \mathcal L_0^* - z)^{-1} \| \prec N^\delta$, and, in particular, with overwhelming probability $\mathcal L_0 \mathcal L_0^*$ has no elements of its spectrum within $\Omega$. We have the second resolvent identity
    \begin{align*}
        (\mathcal L_0 \mathcal L_0^* - z)^{-1} - (\mathcal J(\widecheck\mu_{\vec b}) - z)^{-1} = (\mathcal L_0 \mathcal L_0^* - z)^{-1} ( \mathcal J(\widecheck\mu_{\vec b}) - \mathcal L_0 \mathcal L_0^*) (\mathcal J(\widecheck\mu_{\vec b}) - z)^{-1},
    \end{align*}
    which implies
    \begin{align*}
        \| (\mathcal L_0 \mathcal L_0^* - z)^{-1} - (\mathcal J(\widecheck\mu_{\vec b}) - z)^{-1}\| \prec N^{-1/2 + 2 \delta}, \quad z \in \Omega.
    \end{align*}
    Then by taking contour integrals of the resolvents around small circles $C_j$ of radius $\tau(N) > 0$, lying in $\Omega$, about $f(- \tilde \sigma_j^{-1})$ for $j = 1,2,\ldots,r$ we see that the spectral projectors for $\mathcal L_0 \mathcal L_0^*$ and $\mathcal J(\widecheck\mu_{\vec b})$ associated to eigenvalues within the circles satisfy
    \begin{align*}
        \left\| \frac{1}{2\pi i}\oint_{C_j} (z - \mathcal J(\widecheck\mu_{\vec b}))^{-1} \mathrm{d}z -   \frac{1}{2\pi i}\oint_{C_j} (z - \mathcal L_0 \mathcal L_0^*)^{-1} \mathrm{d}z\right\| \prec 2 \pi \tau(N) N^{-1/2 + 2\delta}.
    \end{align*}
    Since $|\lambda_j(W) - f(- \tilde \sigma_j^{-1})| \prec N^{-1/2}$, it follows that we may take $\tau(N) = CN^{-\delta}$ for some $C > 0$.
    Therefore, the spectral projectors have the same rank with overwhelming probability as $N \to \infty$.

   Next, suppose one of the off-diagonal entries of $\mathcal L_0 \mathcal L_0^*$ vanishes.  This implies that Lanczos failed to run to completion, implying that either $W$ has a repeated eigenvalue, or one of the projections of $\vec b_1$ in Estimation Algorithm~\ref{Ea:ESD} onto one of the eigenvectors of $W$ vanishes --- i.e. the VESD $\mu_{W,\vec b}$ is degenerate.  Using \cite{Christoffersen2025} along with \cite{Tikhomirov2016}, for example, this degeneracy occurs with probability $\OO(N^{-D})$ for any $D > 0$ --- so that with overwhelming probability $\mathcal L_0 \mathcal L_0^*$ has non-zero off-diagonal entries.  From the theory of Jacobi operators \cite{DeiftOrthogonalPolynomials}, with overwhelming probability, the discrete eigenvalues of $\mathcal L_0 \mathcal L_0^*$ are simple and $\widehat m_0(z)$ will have a pole at every discrete eigenvalue of $\mathcal L_0 \mathcal L_0^*$.  So, $\widehat m_0(z)$ will have a pole  within a distance $\tau$ of an asymptotic outlier $f(- \tilde \sigma_j^{-1})$.  On the event where the ranks of the projectors and are equal and $\mathcal L_0 \mathcal L_0^*$ has non-zero off-diagonal entries, the only way the multiplicity of eigenvalues within a circle of radius $\tau$, centered at $f(- \tilde \sigma_j^{-1})$, would be miscounted by counting the poles of $\widehat m_0(z)$ is if $\mathcal L_0 \mathcal L_0^*$ had an eigenvalue of multiplicity greater than one.  This has been ruled out.
\end{proof}


\section{Numerical Experiments and Comparisons}\label{sec:Comp}

 Throughout this section, we examine our approach using $C=1$ and $\delta=0.25$ in Algorithm~\ref{finaldetectionalgorithm}, as they yield reliable results. Based on several experiments and the fact that our support estimates exhibit fluctuations of order $N^{-1/2}$ , we find that values of $\delta$ closer to $1/2$ tend to overestimate the number of spikes, while values near $0$ tend to underestimate. A more detailed analysis of the choice of parameters is left for future work. 

We also compare our approach to other methods, including BEMA0, BEMA \cite{Ke2023}, DDPA \cite{PA}, and the eigen-gap method of Passemier and Yao (Pass\&Yao) \cite{passemier2014estimation}. The BEMA0 method constructs the right endpoint using bulk eigenvalues and applies a correction derived from the distribution of the largest eigenvalue. BEMA approximates the non-spiked covariance matrix $\Sigma_0$ as a diagonal matrix with iid entries from a Gamma distribution. It then uses Monte Carlo simulations to estimate the distribution of the largest eigenvalue, which is used to build a threshold for the spikes. DDPA, on the other hand, relies on parallel analysis, estimating the non-spiked covariance matrix using the diagonal of the sample covariance and establishing a threshold through a deterministic recursive procedure. Finally, the Pass\&Yao method approximates $\Sigma_0$ by $\sigma^2 I$, where $\sigma^2$ is estimated using maximum likelihood. The number of spikes is then identified by comparing the spectral gaps to a threshold derived from Monte Carlo simulations.

\begin{simulation}\label{sim:Johnstone}
    This example examines the performance of Algorithms~\ref{Ea:ESD} and \ref{finaldetectionalgorithm} in the context of Johnstone's spiked covariance model. We consider the sample covariance matrix
    \begin{align}\label{eq:Sim1Model}
        W = \frac{1}{M}\Sigma^{1/2}XX^*\Sigma^{1/2},
    \end{align}
    where 
    \begin{align}\label{eq:Sim1Sigma}
        \Sigma = \mathrm{diag}(5,5,4.5,\sigma^2,\dots, \sigma^2), \quad \sigma^2 = 1.5,
    \end{align}
    and $X\in \R^{N\times M}$ consists of iid normal entries. We note that the ASD of $W$ is explicitly characterized by the Marchenko–Pastur law, given by
    \begin{align}\label{eq:MPdens}
        \mu_{\mathrm{MP}}(\mathrm{d}x) =\varrho_{\mathrm{MP}}(x)\mathds{1}_{[\gamma_{-},\gamma_{+}]}\mathrm{d}x :=\frac{\sqrt{\gamma_{+}-x}\sqrt{x-\gamma_{-}}}{2\pi c \sigma^2 x}\mathds{1}_{[\gamma_{-},\gamma_{+}]}\mathrm{d}x,
    \end{align}
    where $c =  N/M$, $\gamma_{+} = \sigma^2(1+\sqrt{c})^2$ and $\gamma_{-} = \sigma^2(1-\sqrt{c})^2$. 

    \begin{figure}[tbp]
    \centering
    \begin{subfigure}[b]{0.45\textwidth}
        \centering
        \includegraphics[width=\textwidth]{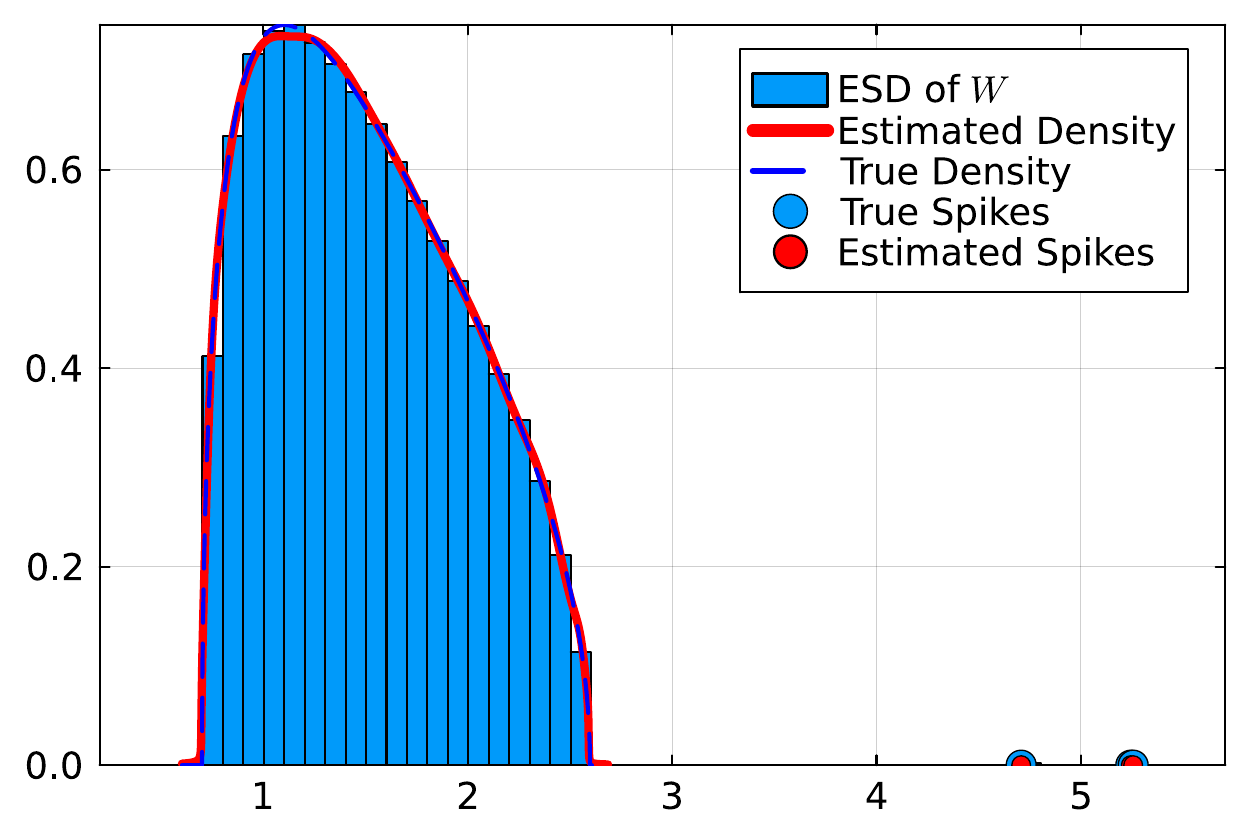}
    \end{subfigure}
    \hfill
    \begin{subfigure}[b]{0.45\textwidth}
        \centering
        \sisetup{table-number-alignment = center} 
        \raisebox{2.3cm}{
        \begin{tabular}{
            >{\centering\arraybackslash}p{3cm} 
            >{\centering\arraybackslash}p{3cm}}
            \toprule
            Estimated outliers & Error \\
            \midrule
            $5.25537$ & $1.50990e\!-\!14$\\
            $5.24390$ & $1.77635e\!-\!15$\\
            $4.70768$ & $2.04281e\!-\!14$\\
            \bottomrule
        \end{tabular}}
    \end{subfigure}
    \vspace{0.5cm} 
    \begin{subfigure}[b]{0.45\textwidth}
        \centering
        \includegraphics[width=\textwidth]{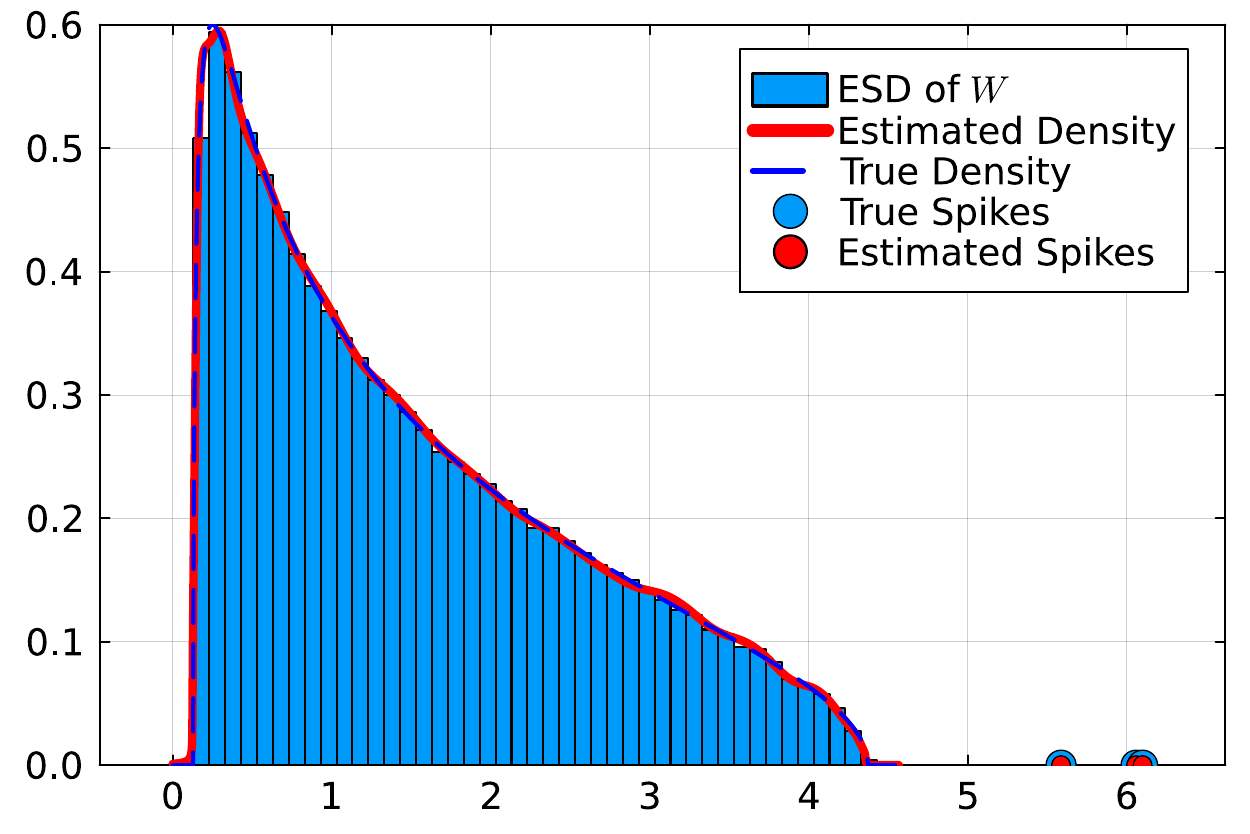}
    \end{subfigure}
    \hfill
    \begin{subfigure}[b]{0.45\textwidth}
        \centering
        \sisetup{table-number-alignment = center} 
        \raisebox{2.3cm}{
        \begin{tabular}{
            >{\centering\arraybackslash}p{3cm} 
            >{\centering\arraybackslash}p{3cm}}
            \toprule
            Estimated outliers & Error \\
            \midrule
            $6.09914$ & $8.88178e\!-\!16$\\
            $6.05833$ & $1.50990e\!-\!14$\\
            $5.58658$ & $1.24344e\!-\!14$\\
            \bottomrule
        \end{tabular}}
    \end{subfigure}
    \vspace{0.5cm} 
    \begin{subfigure}[b]{0.45\textwidth}
        \centering
        \includegraphics[width=\textwidth]{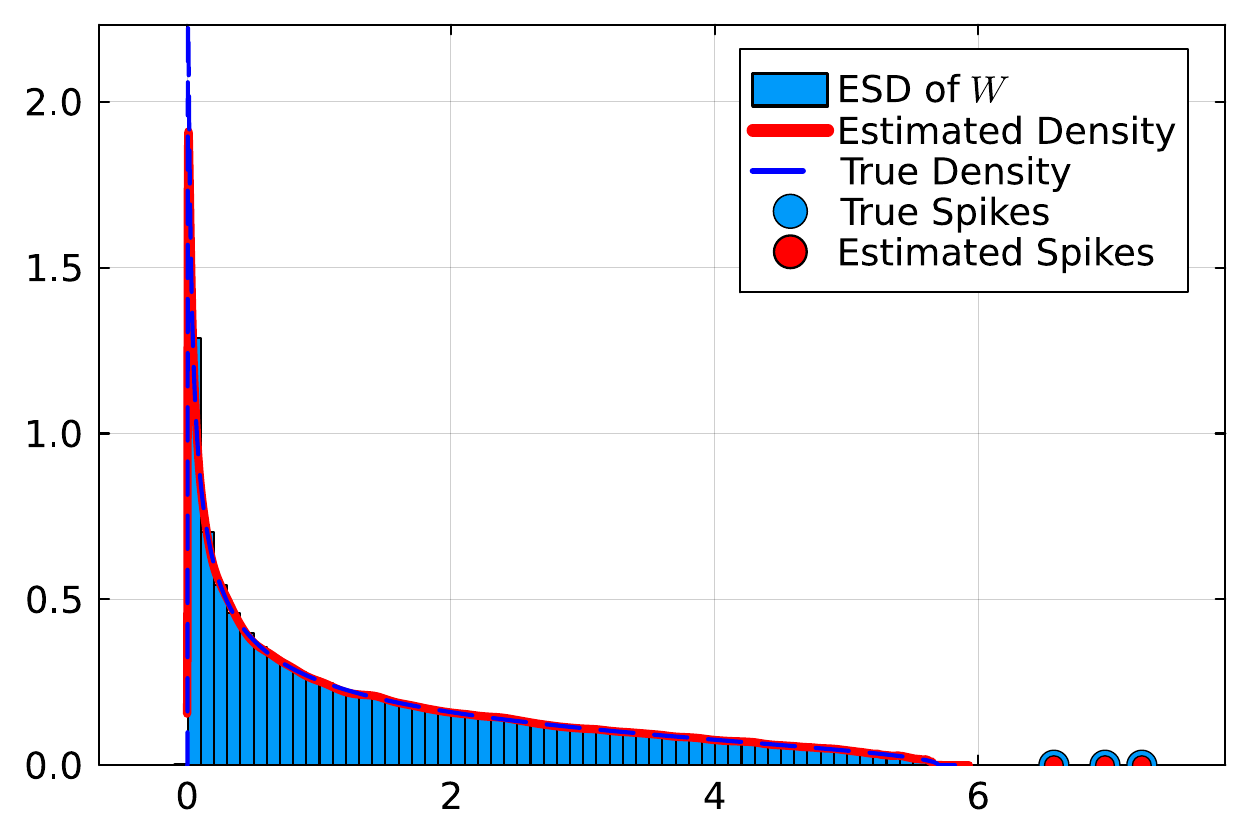}
    \end{subfigure}
    \hfill
    \begin{subfigure}[b]{0.45\textwidth}
        \centering
        \sisetup{table-number-alignment = center} 
        \raisebox{2.3cm}{
        \begin{tabular}{
            >{\centering\arraybackslash}p{3cm} 
            >{\centering\arraybackslash}p{3cm}}
            \toprule
            Estimated outliers & Error \\
            \midrule
            $7.23994$ & $8.88178e\!-\!15$\\
            $6.96060$ & $3.55271e\!-\!15$\\
            $6.57313$ & $6.21724e\!-\!15$\\
            \bottomrule
        \end{tabular}}
    \end{subfigure}
    \caption{The rows correspond to $c = 0.1$, $c = 0.5$, and $c = 0.9$, respectively. Left: Estimated outliers and ASD obtained using Algorithms~\ref{Ea:ESD} and~\ref{finaldetectionalgorithm} with $k = 100$ vectors. These estimates are compared to the ESD of $W$ from Simulation \ref{sim:Johnstone} and the MP density given in \eqref{eq:MPdens}. Right: Estimated outlier locations and the corresponding error between the estimates and the true outliers of $W$.}
    \label{fig:Sim1Density}
    \end{figure}

    The behavior of $W$ is analyzed for three values of $c$: $0.1$, $0.5$, and $0.9$, representing different regimes within the range $c \in (0,1)$. For each case, the estimated ASD and outliers of $W$ are visualized with $N = 5000$, and the number of vectors is chosen as $k = 100$ in Algorithms~\ref{Ea:ESD} and \ref{finaldetectionalgorithm}. Figure~\ref{fig:Sim1Density} shows that the approximate ASD closely matches the exact density, and the detected outliers align well with the true ones, with errors on the order of $10^{-14}$. 

    Next, we investigate the accuracy of Algorithm~\ref{Ea:ESD} in estimating both the density endpoints and the density itself, with the number of vectors $k$ fixed at 100. This time, we consider the three scenarios for $c$ while varying $N$ from $100$ to $8000$. For each value of $N$, the errors are averaged over 50 trials. In each trial, the error in the support is calculated as $\max\{|\gamma_{+}-\widehat{\gamma}_{+}|, |\gamma_{-}-\widehat{\gamma}_{-}|\}$, and the density error is determined as
    \begin{align*}
        \max_{x\in [\widehat{\gamma}_{-} + 0.2, \widehat{\gamma}_{+} - 0.2]} \abs{\frac{\widehat{\varrho}_0(x)}{\sqrt{\widehat \gamma_{+}-x}\sqrt{x-\widehat \gamma_{-}}} - \frac{\varrho_{\mathrm{MP}}(x)}{\sqrt{\gamma_{+}-x}\sqrt{x-\gamma_{-}}}},
    \end{align*}
    where $\widehat \gamma_{\pm}$ and $\widehat \varrho_0$ denote the estimated support and density, respectively. The results are shown in Figure~\ref{fig:Sim1Conv}. The convergence of the support is consistent with our theoretical findings, which predict a convergence rate of $N^{-1/2}$ with overwhelming probability for large $N$. We also observe that the error in estimating the support is greatest for $c = 0.9$, which can be attributed to the density being relatively small near the right endpoint. Moreover, the convergence of the density follows the optimal rate of $N^{-1/2}$, with the largest error occurring for $c = 0.1$.  The rigorous analysis of this will be addressed in future work.

    \begin{figure}[tbp]
        \centering
    \begin{subfigure}{0.48\textwidth}
        \centering
        \includegraphics[width=\textwidth]{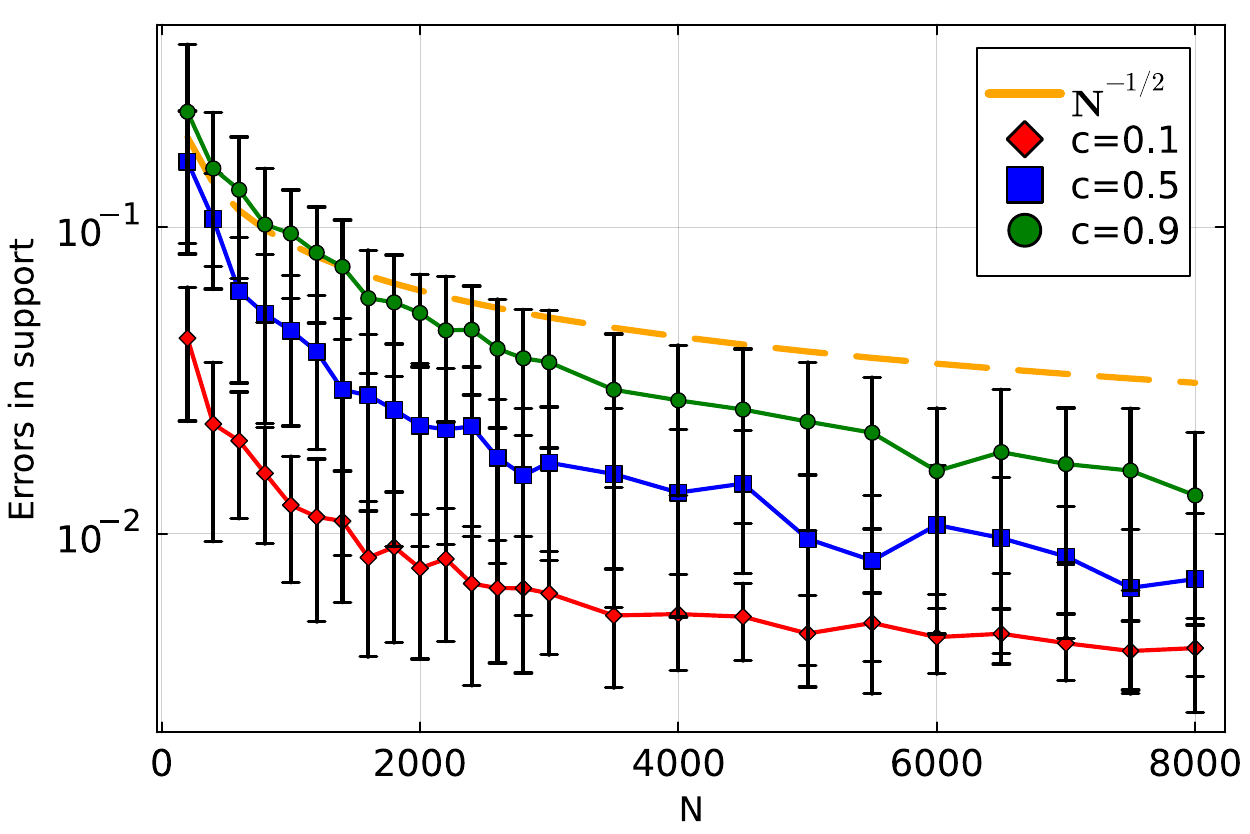}
    \end{subfigure}
    \begin{subfigure}{0.48\textwidth}
        \centering
        \includegraphics[width=\textwidth]{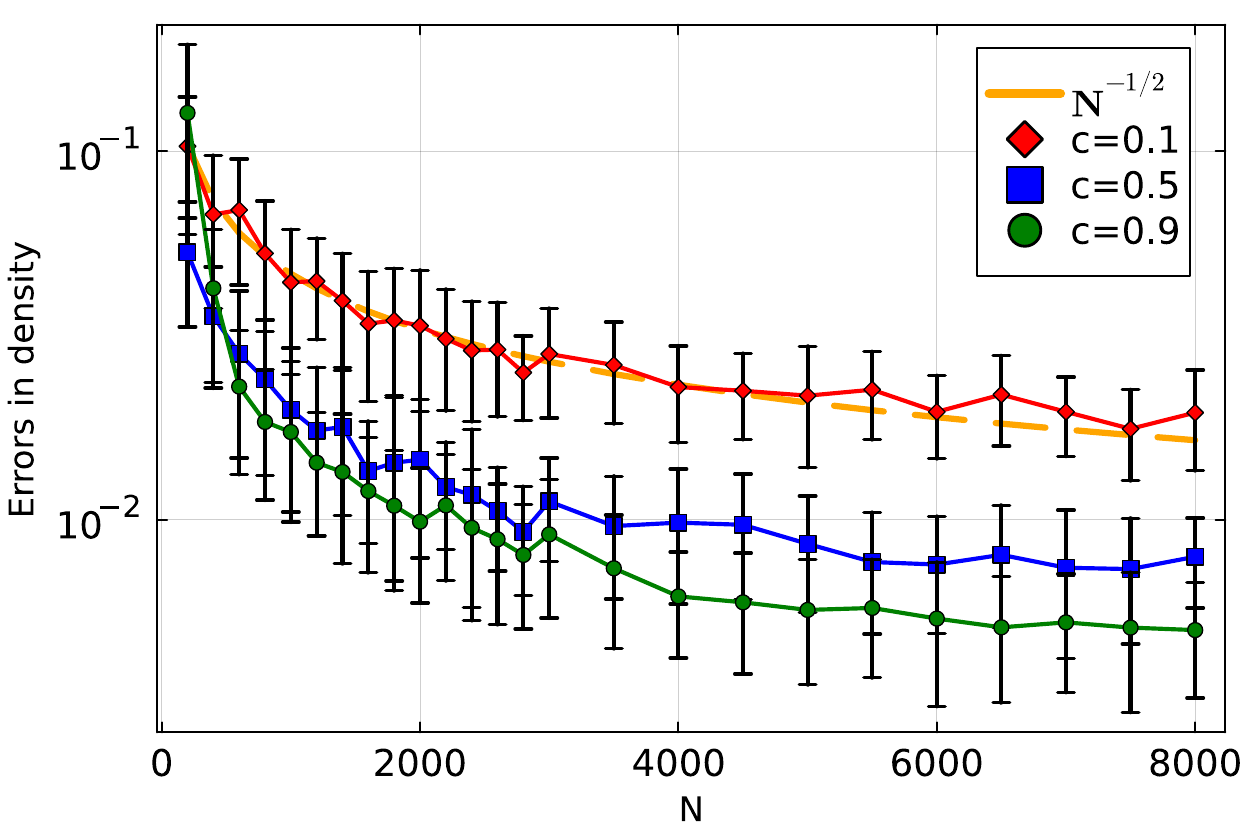}
    \end{subfigure}
    \caption{ Comparison of the estimated support and density from Algorithm~\ref{Ea:ESD} with $k=100$ against the true counterparts from the MP law. The plots illustrate the convergence of errors in Simulation~\ref{sim:Johnstone} for $c = 0.1$, $0.5$, and $0.9$ as $N$ increases. Errors are averaged over 50 trials, with each point representing the mean error and vertical error bars indicating the standard deviation.}
    \label{fig:Sim1Conv}
    \end{figure}

    Finally, we examine the accuracy of our method for spike detection. To highlight its effectiveness, we consider one vector ($k=1$) in Algorithm~\ref{finaldetectionalgorithm}. As before, we consider the three regimes $c = 0.1, 0.5, 0.9$ and vary $N$. For each value of $N$, we generate 50 samples of the covariance matrix $W$ and evaluate both the probability of correctly identifying the number of spikes and the average number of detected spikes. We also demonstrate the efficiency of our approach by comparing its average runtime\footnote{The timing tests in this section were done on a MacBook Pro running macOS Sequoia 15.7.1 with 10 cores and 16 GB of RAM with an Apple M4 chip.} with the time required to compute eigenvalues across the 50 samples. The results are presented in Figure~\ref{fig:Sim1Acc&Eff} and Table~\ref{tab:Sim1Avrg}. Our estimator for $\widehat{r}$ exhibits greater accuracy for $c = 0.1$ and $c = 0.5$ than for $c = 0.9$, which can be attributed to the larger gap between the spikes and the density endpoints. Nonetheless, the accuracy improves as $N$ increases and eventually reaches 1 in all cases. Furthermore, even when the probability of correctly estimating the number of spikes is low, the average value of $\widehat{r}$ remains relatively close to the true spike count. These results also highlight the efficiency of the detection algorithm, with its runtime remaining small as $N$ grows, in contrast to the substantial increase in the computational cost of eigenvalue computation.

    \begin{figure}[tbp]
        \centering
    \begin{subfigure}{0.48\textwidth}
        \centering
        \includegraphics[width=\textwidth]{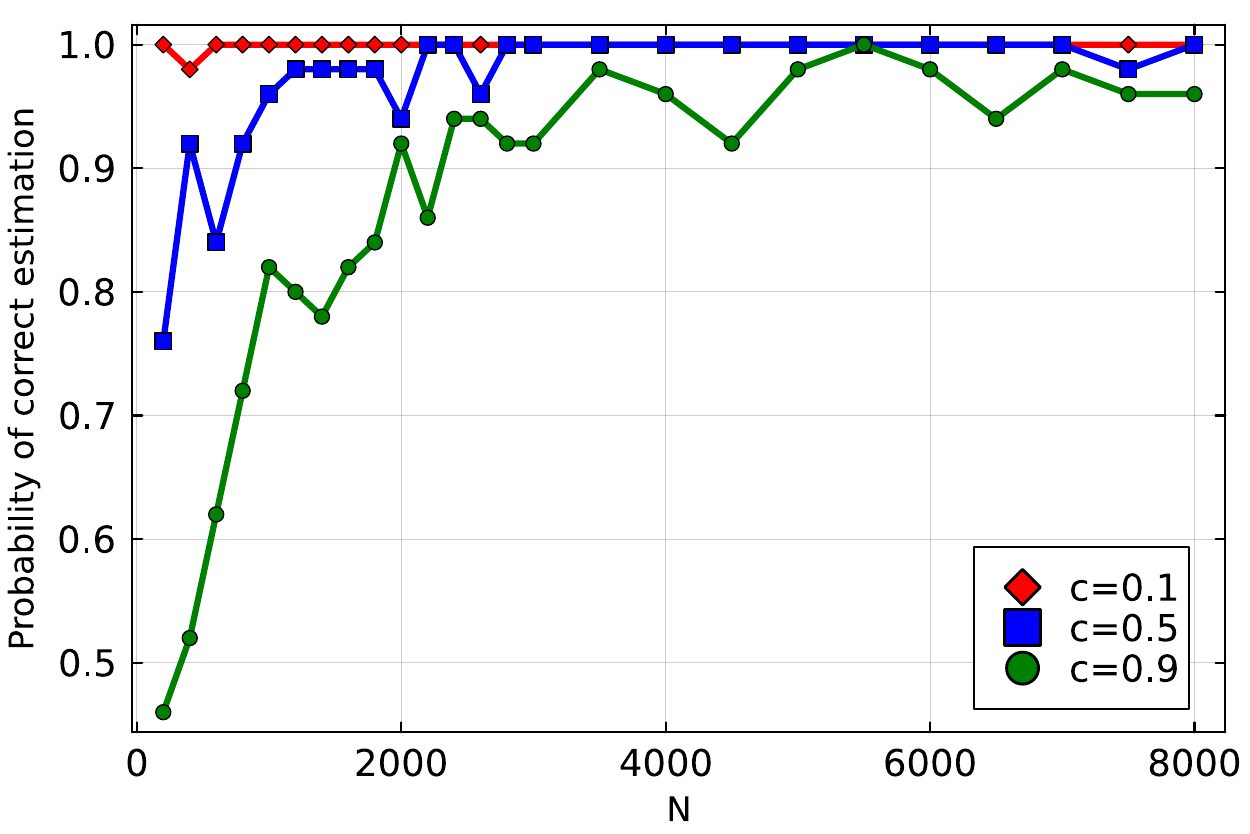}
    \end{subfigure}
    \begin{subfigure}{0.48\textwidth}
        \centering
        \includegraphics[width=\textwidth]{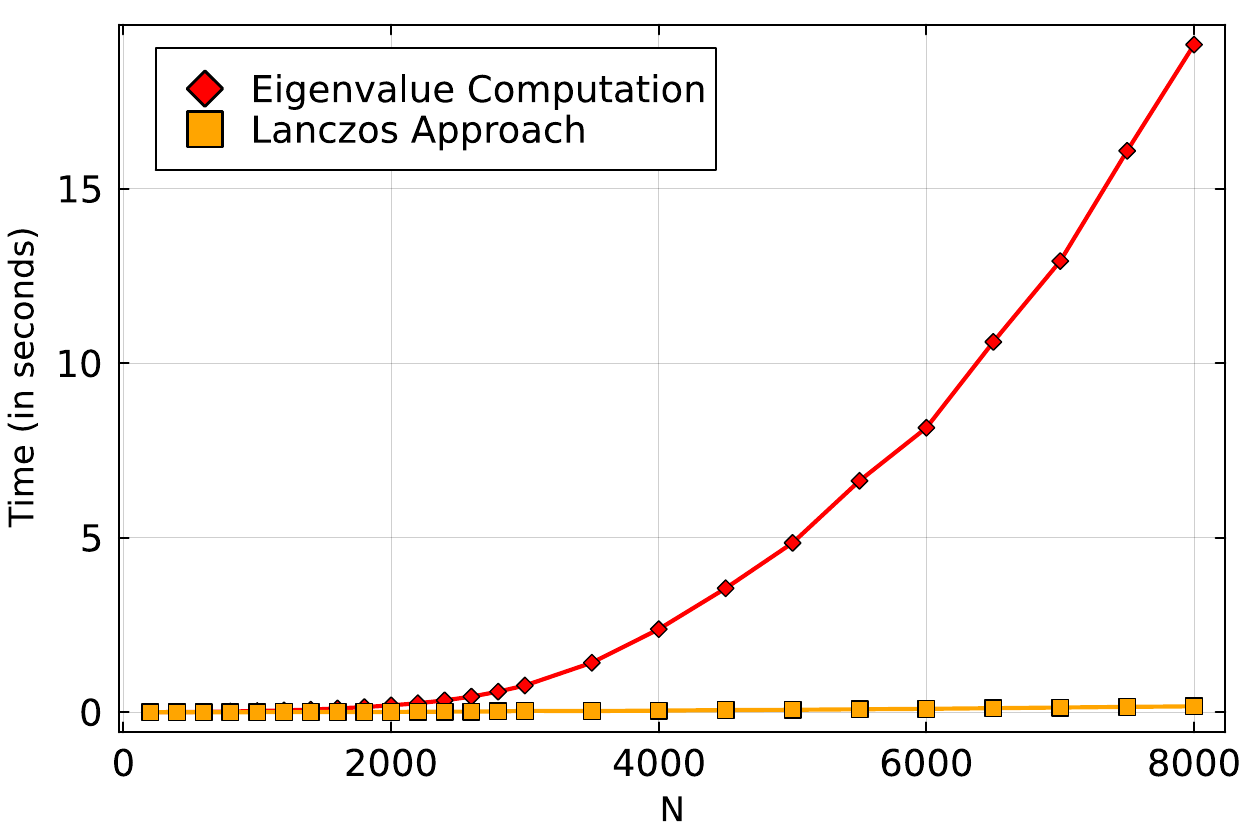}
    \end{subfigure}
    \caption{Left: Spike detection accuracy in Simulation~\ref{sim:Johnstone} for $c = 0.1$, $0.5$, and $0.9$ as $N$ varies, based on 50 sample realizations with the algorithm run using $k=1$. Right: Comparison of the average computation time for eigenvalue calculation and spike detection across the samples for $c=0.5$.}
    \label{fig:Sim1Acc&Eff}
    \end{figure}

    \begin{table}[tbp]
    \centering
    \sisetup{table-number-alignment = center} 
    \begin{tabular}{
        >{\centering\arraybackslash}p{2cm} 
        >{\centering\arraybackslash}p{3cm} 
        >{\centering\arraybackslash}p{3cm} 
        >{\centering\arraybackslash}p{3cm}}
        \toprule
        $N$ & $c=0.1$ & $c = 0.5$ & $c = 0.9$ \\
        \midrule
        $200$ & $3.00$ $(1.00)$ & $3.26$ $(0.76)$ & $3.20$ $(0.46)$ \\
        $2000$ & $3.00$ $(1.00)$ & $3.06$ $(0.94)$ & $3.10$ $(0.92)$ \\
        $4000$ & $3.00$ $(1.00)$ & $3.00$ $(1.00)$ & $3.04$ $(0.96)$ \\
        $6000$ & $3.00$ $(1.00)$ & $3.00$ $(1.00)$ & $3.02$ $(0.98)$ \\
        $8000$ & $3.00$ $(1.00)$ & $3.00$ $(1.00)$ & $3.04$ $(0.96)$ \\
        \bottomrule
    \end{tabular}
    \caption{Estimated number of spikes for $W$ in Simulation~\ref{sim:Johnstone} for different values of $c$ and $N$, using a single vector ($k=1$) in the detection algorithm. Each table entry represents the average over 50 samples, with the value in brackets denoting the probability of correctly detecting the true number of spikes.}
    \label{tab:Sim1Avrg}
    \end{table}
\end{simulation}

\begin{simulation}\label{sim:gap}
     This example examines the performance of Algorithm~\ref{finaldetectionalgorithm} as the gap between the right endpoint and the outliers varies. We compare our method, using different number of vectors, with alternative approaches such as BEMA0 and DDPA. Due to the large matrix dimensions considered in this example, Monte Carlo simulations based on eigenvalue computations become computationally prohibitive, restricting the methods to which we can compare.

    We analyze the standard sample covariance matrix $W$ as defined in \eqref{eq:Sim1Model}, with
    \begin{align}\label{eq:Sim2Sigma}
    \Sigma = \mathrm{diag}(6,5,\delta,\sigma^2,\dots,\sigma^2), \quad \sigma^2=1.
    \end{align}
    We set $N = 8000$ and $M = 16000$, with $X \in \mathbb{R}^{N \times M}$ consisting of iid standard normal entries \footnote{Similar results were obtained when the entries of $X$ followed a Rademacher distribution or a Beta distribution with shape parameters $\alpha = \beta = 1/2$.}. The parameter $\delta$ ranges from $1.5$ to $3$, crossing the BBP transition, where the number of spikes increases from $2$ to $3$ at $\delta = 1+\sqrt{c}$ with $c= N/M $. We evaluate the performance of Algorithm~\ref{finaldetectionalgorithm} with $k=1,50,100,200$, alongside BEMA0 and DDPA, in estimating the number of spikes across $50$ realizations of the sample covariance matrix. In addition, we examine the average number of detected spikes and compare the computational efficiency of each method based on their average runtime. The results are summarized in Figures~\ref{fig:Sim2AccNorm}, \ref{fig:Sim2TimeNorm} and Table~\ref{tab:Sim2Avrg}.

    \begin{figure}[tbp]
      \centering
      \includegraphics[width=\textwidth]{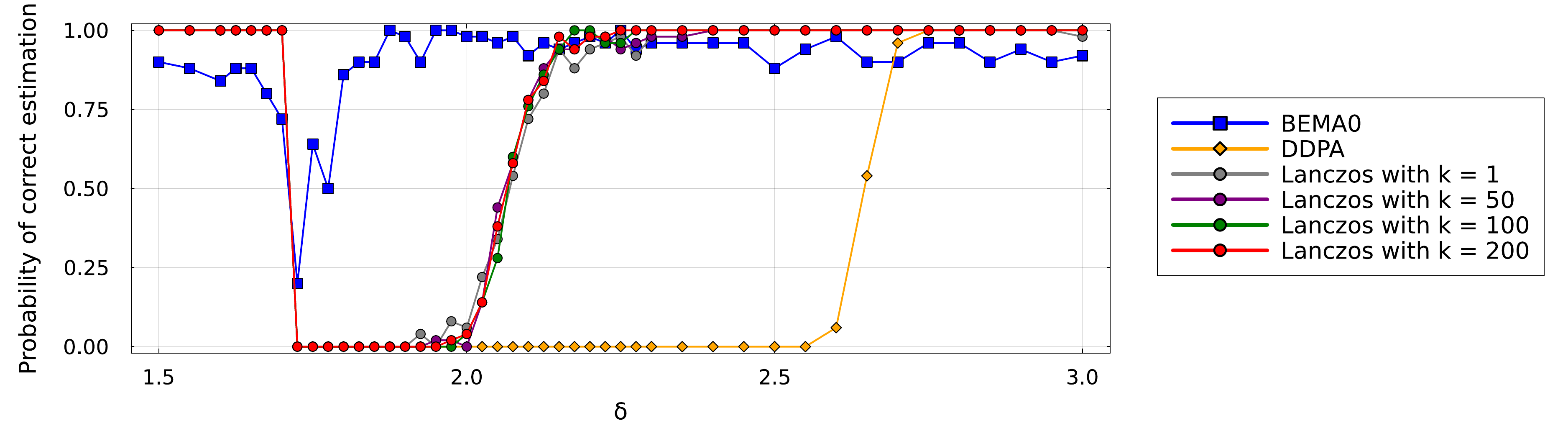}
      \caption{Spike detection accuracy in Simulation~\ref{sim:gap}, averaged over 50 trials for each value of $\delta$ as it varies from 1 to 3. During this process, the sample covariance matrix undergoes the BBP transition, increasing the number of spikes from 2 to 3. We compare the performance of Algorithm~\ref{finaldetectionalgorithm} with different numbers of vectors against the BEMA0 and DDPA methods.}
      \label{fig:Sim2AccNorm}
    \end{figure}

    \begin{figure}[tbp]
      \centering
      \includegraphics[width=\textwidth]{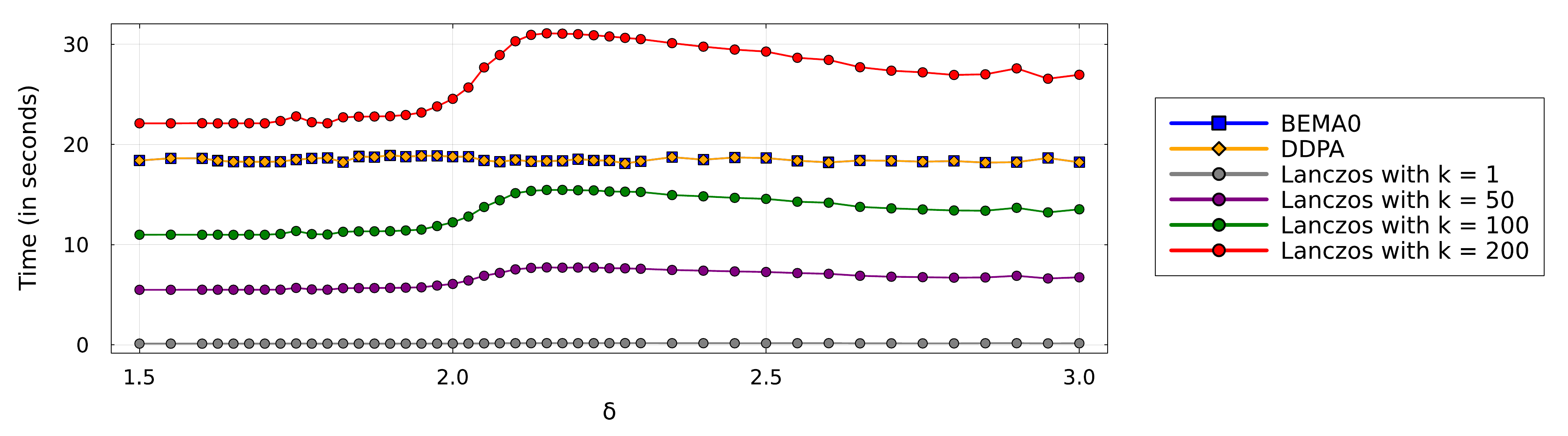}
      \caption{Comparison of the average computation time in Simulation~\ref{sim:gap}, measured over 50 trials, for Algorithm~\ref{finaldetectionalgorithm} with varying numbers of vectors, alongside BEMA0 and DDPA. The runtimes of BEMA0 and DDPA are nearly identical, as both are primarily determined by the cost of eigenvalue computation.}
      \label{fig:Sim2TimeNorm}
    \end{figure}

    Figure~\ref{fig:Sim2AccNorm} and Table~\ref{tab:Sim2Avrg} show that increasing $k$ has little effect on the accuracy of our estimator for the number of spikes, suggesting that it already exhibits low variance and minimal statistical noise. Furthermore, Algorithm~\ref{finaldetectionalgorithm} and DDPA struggle to accurately detect the correct number of spikes at the BBP transition, whereas BEMA0 performs better in this regime. However, our detection algorithm improves in accuracy beyond the transition, outperforming DDPA.

    Figure~\ref{fig:Sim2TimeNorm} further demonstrates that Algorithm~\ref{finaldetectionalgorithm} is significantly more efficient than BEMA0 and DDPA when $k=1$. Although its runtime increases with $k$, it remains comparable to the other two methods. Moreover, BEMA0 and DDPA exhibit nearly identical runtimes, as both are computationally inexpensive, with their cost primarily dictated by the eigenvalue computation.

    \begin{table}[tbp]
    \centering
    \sisetup{table-number-alignment = center} 
    \begin{tabular}{
        >{\centering\arraybackslash}p{2cm} 
        >{\centering\arraybackslash}p{2cm} 
        >{\centering\arraybackslash}p{2cm} 
        >{\centering\arraybackslash}p{2cm}
        >{\centering\arraybackslash}p{2cm}
        >{\centering\arraybackslash}p{2cm}
        >{\centering\arraybackslash}p{2cm}}
        \toprule
        \vspace{0.18cm} $\delta$ & Lanczos \textcolor{white}{.} with \textcolor{white}{.} $k=1$ & Lanczos \textcolor{white}{.} with \textcolor{white}{.} $k=50$ & Lanczos with \textcolor{white}{.} $k=100$ & Lanczos with $k=200$ & \vspace{0.18cm} BEMA0 & \vspace{0.18cm} DDPA\\
        \midrule
        $1.50$ & $2.00$ $(1.00)$ & $2.00$ $(1.00)$ & $2.00$ $(1.00)$ & $2.00$ $(1.00)$ & $2.10$ $(0.90)$ & $2.00$ $(1.00)$ \\
        $1.75$ & $2.00$ $(0.00)$ & $2.00$ $(0.00)$ & $2.00$ $(0.00)$ & $2.00$ $(0.00)$ & $2.64$ $(0.64)$ & $2.00$ $(0.00)$ \\
        $1.90$ & $2.00$ $(0.00)$ & $2.00$ $(0.00)$ & $2.00$ $(0.00)$ & $2.00$ $(0.00)$ & $3.02$ $(0.98)$ & $2.00$ $(0.00)$ \\
        $2.00$ & $2.06$ $(0.06)$ & $2.00$ $(0.00)$ & $2.04$ $(0.04)$ & $2.04$ $(0.04)$ & $3.02$ $(0.98)$ & $2.00$ $(0.00)$ \\
        $2.10$ & $2.72$ $(0.72)$ & $2.78$ $(0.78)$ & $2.76$ $(0.76)$ & $2.78$ $(0.78)$ & $3.08$ $(0.92)$ & $2.00$ $(0.00)$ \\
        $2.25$ & $2.98$ $(0.98)$ & $2.94$ $(0.94)$ & $2.96$ $(0.96)$ & $3.00$ $(1.00)$ & $3.00$ $(1.00)$ & $2.00$ $(0.00)$ \\
        $2.50$ & $3.00$ $(1.00)$ & $3.00$ $(1.00)$ & $3.00$ $(1.00)$ & $3.00$ $(1.00)$ & $3.12$ $(0.88)$ & $2.00$ $(0.00)$ \\
        $2.75$ & $3.00$ $(1.00)$ & $3.00$ $(1.00)$ & $3.00$ $(1.00)$ & $3.00$ $(1.00)$& $3.04$ $(0.96)$ & $3.00$ $(1.00)$\\
        \bottomrule
    \end{tabular}
    \caption{Estimated number of spikes for $W$ in Simulation~\ref{sim:gap} as $\delta$ crosses the BBP transition, increasing the number of spikes from $2$ to $3$. We compare Algorithm~\ref{finaldetectionalgorithm} with different numbers of vectors against BEMA0 and DDPA. Each table entry represents the average over 50 samples, with the value in brackets indicating the probability of correctly identifying the true number of spikes.}
    \label{tab:Sim2Avrg}
    \end{table}
\end{simulation}

\begin{simulation}\label{sim:DetSigma}
    This example examines the performance of our spike detection algorithm for a deterministic covariance matrix $\Sigma$, whose eigenvalues follow a nontrivial deterministic density. We compare Algorithm~\ref{finaldetectionalgorithm} with $k=1$ against other methods, including BEMA0, BEMA, DDPA, and Pass\&Yao.

    \begin{figure}[tbp]
        \centering
        \begin{subfigure}[b]{0.48\textwidth}
            \centering
            \includegraphics[width=\textwidth]{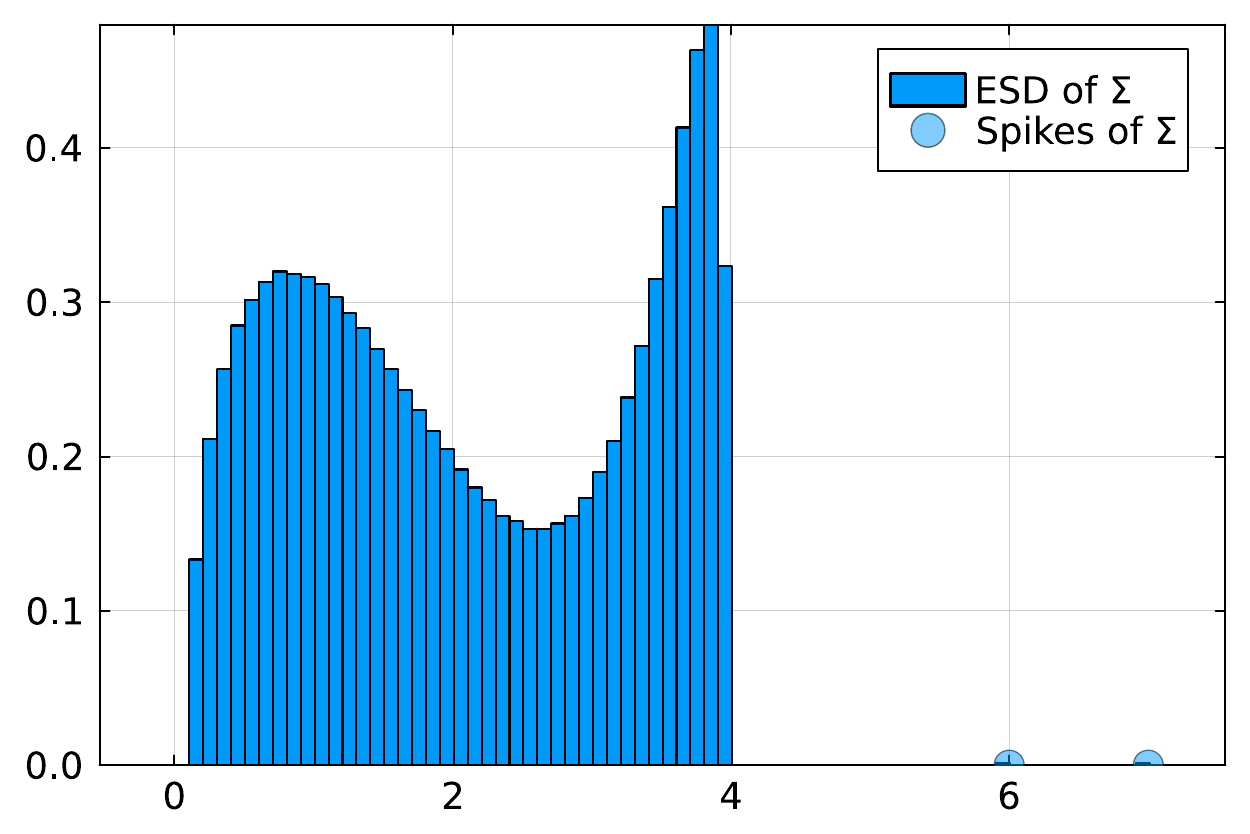}
        \end{subfigure}
        \hfill
        \begin{subfigure}[b]{0.48\textwidth}
            \centering
            \includegraphics[width=\textwidth]{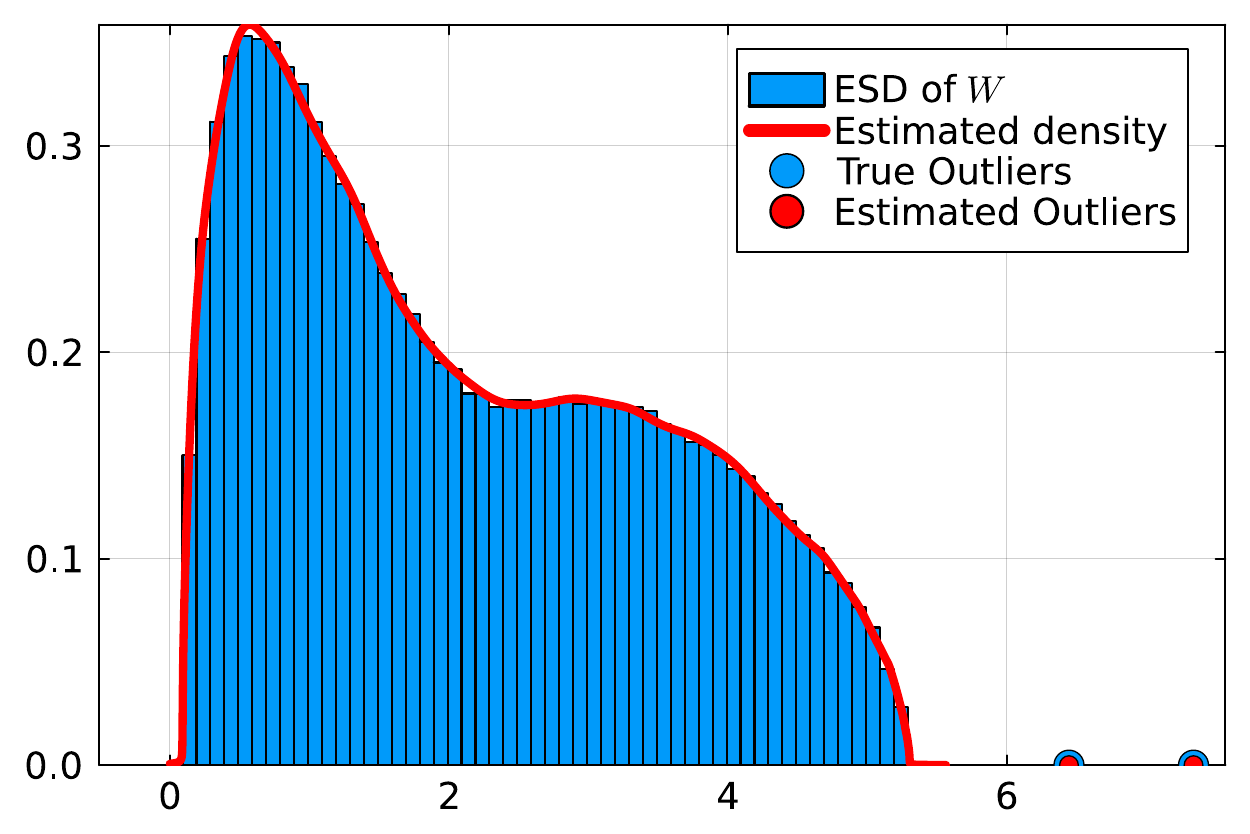}
        \end{subfigure}
        \caption{Left: The ESD of $\Sigma$, where $\Sigma$ is an $N \times N$ diagonal matrix with $N = 3000$ and entries given by the quantiles of the density defined in \eqref{eq:sim3SigmaDens}. The first two diagonal entries are modified to 7 and 6 forming the spikes of $\Sigma$. Right: The ESD of $W$ defined in Simulation~\ref{sim:DetSigma}. The ESD is compared against the estimated outliers, their locations, and the approximate ASD obtained using Algorithms~\ref{Ea:ESD} and~\ref{finaldetectionalgorithm} with $k = 100$.}
        \label{fig:Sim3Dens}
    \end{figure}
    
    We consider the sample covariance matrix $W$ under the spiked covariance model in \eqref{eq:SCM_Model}, where $N=3000$, $M=30000$ and $\sqrt{M} X \in \mathbb{R}^{N \times M}$ has iid standard Gaussian entries. The covariance matrix $\Sigma\in \R^{N\times N}$ is diagonal, and its eigenvalues follow a deterministic distribution consisting of a bulk component and additional spikes. The bulk density is expressed as
    \begin{align}\label{eq:sim3SigmaDens}
        \mu(x) = \frac{1}{K}\frac{(2(3.5-x)^3+x)}{(4.5-x)^2}\sqrt{4-x}\sqrt{x-0.1},
    \end{align}
    where $K$ is the normalization constant, and the spikes of $\Sigma$ are located at $7, \delta$. See Figure~\ref{fig:Sim3Dens} for a visualization of the ESD of $\Sigma$ and the ASD of $W$ when $\delta = 6$. We approximate the ASD and detect the outliers of $W$ using Algorithms~\ref{Ea:ESD} and \ref{finaldetectionalgorithm} with $k=100$ vectors. Our ASD estimate closely matches the empirical distribution, accurately identifying both the location and number of spikes.

    Next, we vary the parameter $\delta = 5, 7, 9, 21$ and evaluate the performance of our detection algorithm with a single vector ($k = 1$), comparing it to other spike detection methods. For each value of $\delta$, we perform the evaluation over 50 samples. Specifically, we examine the accuracy, the average number of spikes detected, and the runtime for each method. The results are summarized in Table~\ref{tab:Sim3Res}. 
    
    Our algorithm outperforms the other methods, as it consistently captures the correct number of spikes with a probability close to 1. DDPA performs second best, with its performance improving as $\delta$ increases, reaching a probability of 1 when $\delta = 9,21$. This demonstrates the versatility of our approach, as it can handle a wide range of covariance matrices without assuming any specific distribution or structure for $\Sigma$. Additionally, we observe that Algorithm~\ref{finaldetectionalgorithm} achieves the shortest runtime, followed by BEMA0 and DDPA, which have similar runtimes as their costs are mainly dominated by the eigenvalue computations. BEMA and Pass\&Yao are more computationally demanding, as they rely on Monte Carlo simulations with large matrices.

    \begin{table}[tbp]
    \centering
    \sisetup{table-number-alignment = center} 
    \begin{tabular}{
        >{\centering\arraybackslash}p{2cm} 
        >{\centering\arraybackslash}p{2cm} 
        >{\centering\arraybackslash}p{2cm} 
        >{\centering\arraybackslash}p{2cm}
        >{\centering\arraybackslash}p{2cm}
        >{\centering\arraybackslash}p{2cm}}
        \toprule
        \vspace{0.02cm}$\delta=5$ & \vspace{0.02cm} BEMA0 & \vspace{0.02cm} BEMA & \vspace{0.02cm} Pass\&Yao & \vspace{0.02cm} DDPA & Lanczos with $k=1$ \\
        \midrule
        Probability & $0.00$ & $0.00$ & $0.50$ & $0.00$ & $0.88$ \\
        \midrule
        Average & $553.66$ & $0.00$ & $2.86$ & $1.00$ & $1.88$ \\
        \midrule
        Time (sec) & $0.794847$ & $288.6114$ & $227.1758$ & $0.794932$ & $0.035953$ \\
        \bottomrule
    \end{tabular}
    \begin{tabular}{
        >{\centering\arraybackslash}p{2cm} 
        >{\centering\arraybackslash}p{2cm} 
        >{\centering\arraybackslash}p{2cm} 
        >{\centering\arraybackslash}p{2cm}
        >{\centering\arraybackslash}p{2cm}
        >{\centering\arraybackslash}p{2cm}}
        \toprule
        \vspace{0.02cm}$\delta=7$ & \vspace{0.02cm} BEMA0 & \vspace{0.02cm} BEMA & \vspace{0.02cm} Pass\&Yao & \vspace{0.02cm} DDPA & Lanczos with $k=1$ \\
        \midrule
        Probability & $0.00$ & $0.00$ & $0.36$ & $0.76$ & $1.00$ \\
        \midrule
        Average & $553.68$ & $0.00$ & $3.18$ & $1.52$ & $2.00$ \\
        \midrule
        Time (sec) & $0.797435$ & $288.3089$ & $245.7479$ & $0.797535$ & $0.033605$ \\
        \bottomrule
    \end{tabular}
    \begin{tabular}{
        >{\centering\arraybackslash}p{2cm} 
        >{\centering\arraybackslash}p{2cm} 
        >{\centering\arraybackslash}p{2cm} 
        >{\centering\arraybackslash}p{2cm}
        >{\centering\arraybackslash}p{2cm}
        >{\centering\arraybackslash}p{2cm}}
        \toprule
        \vspace{0.02cm}$\delta=9$ & \vspace{0.02cm} BEMA0 & \vspace{0.02cm} BEMA & \vspace{0.02cm} Pass\&Yao & \vspace{0.02cm} DDPA & Lanczos with $k=1$ \\
        \midrule
        Probability & $0.00$ & $0.00$ & $0.40$ & $1.00$ & $0.98$ \\
        \midrule
        Average & $553.74$ & $0.00$ & $3.04$ & $2.00$ & $2.02$ \\
        \midrule
        Time (sec) & $0.784069$ & $288.4510$ & $238.0594$ & $0.784192$ & $0.032017$ \\
        \bottomrule
    \end{tabular}
    \begin{tabular}{
        >{\centering\arraybackslash}p{2cm} 
        >{\centering\arraybackslash}p{2cm} 
        >{\centering\arraybackslash}p{2cm} 
        >{\centering\arraybackslash}p{2cm}
        >{\centering\arraybackslash}p{2cm}
        >{\centering\arraybackslash}p{2cm}}
        \toprule
        \vspace{0.02cm}$\delta=21$ & \vspace{0.02cm} BEMA0 & \vspace{0.02cm} BEMA & \vspace{0.02cm} Pass\&Yao & \vspace{0.02cm} DDPA & Lanczos with $k=1$ \\
        \midrule
        Probability & $0.00$ & $0.00$ & $0.34$ & $1.00$ & $1.00$ \\
        \midrule
        Average & $553.74$ & $1.00$ & $3.16$ & $2.00$ & $2.00$ \\
        \midrule
        Time (sec) & $0.779923$ & $290.1505$ & $246.6577$ & $0.779997$ & $0.028166$ \\
        \bottomrule
    \end{tabular}
    \caption{Probability of correct estimation, average number of spikes detected, and runtime for different methods in Simulation~\ref{sim:DetSigma}, evaluated over 50 samples for each value of $\delta$. }
    \label{tab:Sim3Res}
    \end{table}
\end{simulation}

\section*{Acknowledgements}

This material is based upon work supported by NSF DMS-2306438 (TT), DMS-2306439 (XD). Any opinions, findings, and conclusions or recommendations expressed in this material are those of the author and do not necessarily reflect the views of the National Science Foundation.

\appendix

\section{Linear Algebra Identities and Algorithms}

\subsection{The Lanczos and Cholesky algorithms}\label{sec:lanczos}\label{sec:cholesky}
For the sake of completeness, we give the full Lanczos algorithm (Algorithm \ref{a:lanczos}) and present the algorithm to compute the Cholesky decomposition of a Jacobi matrix (Algorithm~\ref{a:Chol}).

\newcounter{basic-algorithm}
\renewcommand{\thealgorithm}{B.\arabic{basic-algorithm}}
\setcounter{basic-algorithm}{0}
\begin{algorithm}[tbp]
    \refstepcounter{basic-algorithm}
	\caption{Lanczos Algorithm}
    \label{a:lanczos}
    \begin{algorithmic}[1]
        \Statex \textbf{Input:} Hermitian matrix $W\in \mathbb{R}^{N\times N}$, initial vector $q_1$ such that $\|q_1\|^2 = q_1^* q_1 = 1$.
        \State Set $b_{-1} = 0$ and $q_0 = 0$.
        \For{$j=1,2,\dots,n, \quad n\leq N$} 
            \State Compute $a_{j-1} = (W q_j - b_{j-2} q_{j-1})^* q_j$.
            \State Set $v_j = W q_j - a_{j-1} q_j - b_{j-2} q_{j-1}$.
            \State Compute $b_{j-1} = \|v_j\|$.
            \If{$b_{j-1} = 0$}
                \State \Return{$a_0,\dots,a_{j-1},b_0,\dots,b_{j-1}$.}
            \Else
                \State Set $q_{j+1} = v_j / b_{j-1}$.
            \EndIf
        \EndFor
        \State \Return{$a_0,\dots,a_{n-1},b_0,\dots,b_{n-1}$.}
    \end{algorithmic}
\end{algorithm}


\begin{algorithm}[tbp]
    \refstepcounter{basic-algorithm}
    \caption{Cholesky Algorithm}\label{a:Chol}
    \begin{algorithmic}[1]
        \Statex \textbf{Input:} Jacobi matrix $T\in \mathbb{R}^{N\times N}$ generated from Algorithm~\ref{a:lanczos}.
        \State Set $L = T$.
        \For{$i=1,2,\ldots,N$}
            \State Set $L_{i,i+1} = 0$.
            \State Set $L_{i+1,i+1} = L_{i+1,i+1} - L_{i+1,i} L_{i,i+1} / L_{ii}$.
            \State Set $L_{i:i+1,i} = L_{i:i+1,i} / \sqrt{L_{ii}}$.
        \EndFor
        \State \Return $L$.
    \end{algorithmic}
\end{algorithm}

\subsection{Matrix Identities}\label{sec:matrix}

We emphasize two well-known identities that are used in a key way.
\begin{lemma}[Schur Complement]\label{lem:SchurCompl}
    Suppose $p,q$ are nonnegative integers such that $p+q>0$, and suppose $A,B,C,D$ are respectively $p\times p$, $p\times q$, $q\times p$, $q\times q$ matrices of complex numbers. Let 
    \begin{align*}
        M = \begin{bmatrix}
            A & B \\
            C & D
        \end{bmatrix}
    \end{align*}
    so that $M$ is a $(p+q)\times (p+q)$ matrix. If $D$ and $M$ are invertible then the upper left $q\times q$ block of $M^{-1}$ is given by $(M/D)^{-1}$, where the Schur complement of the block $D$ is given by
    \begin{align*}
        M/D := A-BD^{-1}C.
    \end{align*}
\end{lemma}
\begin{lemma}[Woodbury Matrix Identity]\label{lem:Woodbury}
    Let $A,U,C$ and $V$ be conformable matrices then
    \begin{align*}
        (A+UCV)^{-1} = A^{-1}-A^{-1}U(C^{-1}+VA^{-1}U)^{-1}VA^{-1},
    \end{align*}
    whenever the requisite inverses exist.
\end{lemma}

\section{Proofs of some techinal lemmas}\label{sec_appendixxuxilimarylemmasproof}

\begin{proof}[\bf Proof of Lemma \ref{lem:SpikedLocLaw}]
    First, we expand the vectors $\mathbf{u},\mathbf{v}$ as $\mathbf{u} = \sum_{i=1}^r {u}_i\mathbf{v}_i + \widehat{\mathbf{u}}$ and $\mathbf{v} = \sum_{i=1}^r {v}_i\mathbf{v}_i + \widehat{\mathbf{v}}$
    where $\widehat{\mathbf{u}}$, $\widehat{\mathbf{v}}$ are vectors in the orthogonal complement of $\{\mathbf{v}_i\}_{i=1}^r$. Multiplying \eqref{eq:SpikedvsSpikeRes} on the left and right by $\mathbf{v}_i$ yields
    \begin{align*}
        \mathbf{v}_i^*G_1(z)\mathbf{v}_i = \frac{\sigma_i}{\tilde{\sigma}_i} \paren{\mathbf{v}_i^*\widetilde{G}_1(z)\mathbf{v}_i-z\mathbf{v}_i^*\widetilde{G}_1(z)V_r\paren{D_r^{-1}+I+zV_r^*\widetilde{G}_1(z)V_r}^{-1}V_r^*\widetilde{G}_1(z)\mathbf{v}_i}.
    \end{align*}
    A key calculation is
    \begin{align*}
        \vec v_i^* \widetilde{G}_1(z) \vec v_j 
        = - \frac 1 z (1 + m_{\mu_\dMP}(z) \sigma_i)^{-1} \vec v_i^* \vec v_j + \OO_\prec(\Psi(z)).
    \end{align*}
    We then consider
    \begin{align*}
        \mathbf{v}_i^*\widetilde{G}_1(z)V_r = \begin{bmatrix} \vec v_i^* \widetilde{G}_1(z) \vec v_1 & \cdots & \vec v_i^* \widetilde{G}_1(z) \vec v_r \end{bmatrix} = - \frac{1}{z} (I + m_{\mu_\dMP}(z) \sigma_i)^{-1} \vec e_i^* + \OO_\prec( \Psi(z)),
    \end{align*}
    using Lemma~\ref{lem:NonSpikedLocLaw}.  We compute, again using Lemma~\ref{lem:NonSpikedLocLaw},
    \begin{align*}
        z V_r^* \widetilde{G}_1(z) V_r = - (I + m_{\mu_\dMP}(z) \mathrm{diag}(\sigma_1,\ldots,\sigma_r))^{-1} + \OO_\prec(\Psi(z)).
    \end{align*}
    From \eqref{eq:west}, we have
    \begin{align*}
        \left\|\left(D_r^{-1}+I+zV_r^*\widetilde{G}_1(z)V_r\right)^{-1}\right\| = \OO_{\prec} (1).
    \end{align*}

 We find that for $z\in \widetilde{\mathcal{D}}$ and $1\leq i\leq r$,
    \begin{align*}
        \mathbf{v}_i^*G_1(z)\mathbf{v}_i = \frac{1}{1+d_i}\paren{\mathbf{v}_i^*\widetilde{G}_1(z)\mathbf{v}_i-\mathcal{L}_i} + \OO_\prec(\Psi(z)).
    \end{align*}
    On the other hand, since $\widehat{\mathbf{u}}^*V_r = 0$ and $\widehat{\mathbf{v}}^*V_r = 0$, it follows from Lemma \ref{lem:NonSpikedLocLaw} that for $z \in \widetilde{\mathcal{D}}$
    \begin{align*}
        \widehat{\mathbf{u}}^*G_1(z) \widehat{\mathbf{v}} = \widehat{\mathbf{u}}^*\widetilde{G}_1(z)\widehat{\mathbf{v}} + \OO_{\prec}(\Psi(z)),
    \end{align*}
    where we used $\tilde{\sigma}_i=\sigma_i$ for $i>r$. Moreover, we have that for $1\leq i,j \leq r$, $i\neq j$
    \begin{align*}
        \widehat{\mathbf{u}}^*G_1(z)\mathbf{v}_i = \OO_{\prec} (\Psi(z)), \quad \mathbf{v}_i^*G_1(z)\widehat{\mathbf{v}} = \OO_{\prec} (\Psi(z)).
    \end{align*}
    This establishes $\eqref{eq:SpikedLawG1}$ by observing that $\widehat{\mathbf{u}} = \sum_{i=r+1}^N {u}_i \mathbf{v}_i$ and $\widehat{\mathbf{v}} = \sum_{i=r+1}^{N} {v}_i \mathbf{v}_i$.

    Now, let $\Delta(z) = \widetilde{G}(z)-\widetilde{\Pi}(z)$ and, by Lemma \ref{lem:SpikedvsNoSpikeCompRes}, we have that
    \begin{align*}
        \mathbf{u}^* G_2(z) \mathbf{v} &= \mathbf{u}^* \widetilde{G}_2(z) \mathbf{v} 
        + z \widetilde{\mathbf{u}}^* \widetilde{\Pi}(z) \widehat{V}_r 
        \left( D_r^{-1} + I + z \widehat{V}_r^* \widetilde{G}(z) \widehat{V}_r \right)^{-1} 
        \widehat{V}_r^* \widetilde{G}(z) \widetilde{\mathbf{v}} \\
        &\quad - z \widetilde{\mathbf{u}}^* \Delta(z) \widehat{V}_r 
        \left( D_r^{-1} + I + z \widehat{V}_r^* \widetilde{G}(z) \widehat{V}_r \right)^{-1} 
        \widehat{V}_r^* \widetilde{G}(z) \widetilde{\mathbf{v}}.
    \end{align*}
    Using the structure of \eqref{eq:DetEquiv}, \eqref{eq:uEmb} and \eqref{eq:VEmb}, we have that
    \begin{align*}
        z \widetilde{\mathbf{u}}^* \widetilde{\Pi}(z) \widehat{V}_r 
        \left( D_r^{-1} + I + z \widehat{V}_r^* \widetilde{G}(z) \widehat{V}_r \right)^{-1} 
        \widehat{V}_r^* \widetilde{G}(z) \widetilde{\mathbf{v}} = 0.
    \end{align*}
    From Lemma \ref{lem:NonSpikedLocLaw}, we find that
    \begin{align*}
        \norm{\widetilde{\mathbf{u}}^*\Delta(z)\widehat{V}_r}= \OO_{\prec} (\Psi(z)), \quad \text{and} \quad \norm{\widehat{V}_r^* \widetilde{G}(z) \widetilde{\mathbf{v}}} =\OO_{\prec} (1),
    \end{align*}
    and
    \begin{align*}
        \left( D_r^{-1} + I + z \widehat{V}_r^* \widetilde{G}(z) \widehat{V}_r \right)^{-1} = \left(D_r^{-1}+I+zV_r^*\widetilde{G}_1(z)V_r\right)^{-1}.
    \end{align*}
    This completes the proof for \eqref{eq:SpikedLawG2}.
\end{proof}

\begin{proof}[\bf Proof of Lemma \ref{lem:VESDLimDet}]
    Using the Stieltjes inversion formula \eqref{eq:StieltjesInv}, it follows from \eqref{eq:LimCharact} that the density of $\mu_{0,\mathbf{b}}$ is given by
\begin{align}
    \varrho_{0,\mathbf{b}}(\lambda) = \frac{\varrho_{\dMP}(\lambda)}{\lambda} \mathbf{b}^* \Sigma_0 \braces{I+2\mathrm{Re}~m_{{\dMP}}(\lambda+\ri0^+)\Sigma_0 + |m_{{\dMP}}(\lambda+\ri0^+)|^2\Sigma_0^2}^{-1}\mathbf{b},
\end{align}
where $\mathrm{Im}~m_{\dMP}(x+\ri0^+) = \lim_{\epsilon\downarrow 0^{+}} \mathrm{Im}~m_{\dMP}(x+\ri\epsilon)$. Under Assumption \ref{as:TechA}, it 
follows that $\mu_{0,\mathbf{b}}$ satisfies Assumption \ref{as:measurecond} with the same support as $\mu_\dMP$. On the other hand, by observing that $d_{i}^{-1} + 1 - \left(1 + m_{\dMP}\big(f(-\tilde{\sigma}_i^{-1})\big)\sigma_i\right)^{-1} = 0$ and applying \eqref{eq:LimCharact}, and Lemma \ref{lem:EigValVecConv}, it follows that \(\mu_{\mathbf{b}}\) satisfies Assumption \ref{as:measurecond} also with the same support as $\mu_{\dMP}$ and, for each $1 \leq i \leq r$, we have
\begin{align*}
    c_i = \gamma_i = f(-\tilde{\sigma}_i^{-1}).
\end{align*}
For more details, see \cite{Ding2021c}.
\end{proof}

\section{Pole computation}\label{a:poles}

We present an approach for computing the poles of the Stieltjes transform associated with perturbations of Toeplitz Jacobi operators through connection coefficients, as described in \cite{Webb2021}. An alternative method to approximate these poles, by truncating the semi-infinite matrix, is also presented here. The first method, that of Olver \& Webb, gives guaranteed accuracy in exact arithmetic. The second offers weaker guarantees but offers a simpler implementation. Other techniques, such as those in \cite{bini_computing_2023}, may also be useful in this context.

\subsection{Finite rank perturbation of Jacobi operators and connection coefficients}

Consider a semi-infinite matrix $\mathcal{J}$ of the form
\begin{align}\label{eq:RefToeJac}
    \mathcal{J} = \begin{bmatrix}
        \alpha & \beta \\
        \beta & \alpha & \beta \\ 
        & \beta & \alpha & \ddots \\
        & & \ddots & \ddots
    \end{bmatrix},\quad \beta >0,
\end{align}
which is referred to as a Toeplitz Jacobi operator. The spectral properties of this operator are well understood, and its spectral measure is explicitly given by
\begin{align}\label{eq:Refmeas}
    \mu(\mathrm{d}\lambda) = \frac{1}{2\pi\beta^2} \sqrt{\gamma_{+}-x}\sqrt{x-\gamma_{-}}, \quad \text{where} \quad \gamma_{\pm} = \alpha \pm 2\beta.
\end{align}
Let $\widetilde{\mathcal{J}}$ be a finite-rank perturbation of $\mathcal{J}$ defined as
\begin{align}\label{eq:Jpertdef1}
    \widetilde{\mathcal{J}} = \begin{bmatrix}
        \tilde\alpha_0 & \tilde \beta_0 \\
        \tilde \beta_0 & \tilde \alpha_0 & \tilde\beta_1 \\ 
        & \tilde \beta_1 & \tilde\alpha_1 & \ddots \\
        & & \ddots & \ddots
    \end{bmatrix}, \quad \tilde\beta_j > 0,
\end{align}
with
\begin{align}\label{eq:Jpertdef2}
    \tilde\alpha_j=\alpha,~\tilde\beta_j=\beta~\text{for all }j\geq n.
\end{align}
Such matrices are referred to as Toeplitz-plus-finite-rank Jacobi operators.
Define $C = C_{\widetilde{\mathcal{J}} \to \mathcal{J}} = (c_{i,j})_{i,j=0}^{\infty}$ as the upper triangular matrix representing the change of basis between the orthonormal polynomials $(P_k)_{k=0}^{\infty}$ associated with $\widetilde{\mathcal{J}}$ and the orthonormal polynomials $(Q_k)_{k=0}^{\infty}$ corresponding to $\mathcal{J}$. This matrix, known as the connection coefficient matrix, satisfies
\begin{align}\label{eq:ConnCoefCond}
    P_k(\lambda) = c_{0,k}Q_0(\lambda) + c_{1,k}Q_1(\lambda) + \dots + c_{k,k}Q_k(\lambda), \quad c_{i,j} = \abrac{P_j,Q_i}_{\mu},
\end{align}
where $\abrac{\cdot,\cdot}_{\mu}$ is the inner product for $L^2(\mu)$.
The significance of this matrix lies in its ability to connect the spectral properties of $\widetilde{\mathcal{J}}$ with those of $\mathcal{J}$, allowing the spectrum of $\widetilde{\mathcal{J}}$ and its spectral measure to be fully determined. It is worth noting that the entries of $C$ follow a well-defined recurrence relation. For further details, see \cite{sack_algorithm_1971}, \cite[Lemma 3.2]{Webb2021} and \cite{wheeler_modified_1974}.
\begin{lemma}[\!\!\cite{Webb2021}, Lemma 3.2]
    The entries $c_{i,j}$ of the connection coefficient matrix $C=C_{\widetilde{\mathcal{J}}\to\mathcal{J}}$ satisfy a 5-term recurrence relation
    \begin{align}\label{eq:5Rec}
        -\beta c_{i-1,j}+\tilde{\beta}_{j-1}c_{i,j-1}+(\tilde{\alpha}_j-\alpha)c_{i,j} + \tilde{\beta}_{j}c_{i,j+1}-\beta c_{i+1,j}=0, \quad \text{for all }0\leq i\leq j,
    \end{align}
    where
    \begin{align*}
        c_{i,j} = \begin{cases}
            1 & \text{if }i=j=0,\\
            0 & \text{if }j=0\text{ and }i\neq 0,\\
            0 & \text{if }j=-1\text{ or }i=-1.
        \end{cases}
    \end{align*}
\end{lemma}

Due to the structured form of $\widetilde{\mathcal{J}}$ in \eqref{eq:Jpertdef1}, \eqref{eq:Jpertdef2}, the connection coefficients matrix also inherits a specific structure. In particular, it decomposes as \cite{Webb2021},
\begin{align}\label{eq:CDecomp}
    C = C_{\mathrm{Toe}} + C_{\mathrm{fin}}
\end{align}
where $C_{\mathrm{Toe}}$ is a Toeplitz matrix with bandwidth $2n-1$, and $ C_{\mathrm{fin}}$ has nonzero entries only within the $(n-1) \times (n-2)$ principal submatrix. The next theorem provides a precise description and demonstrates that the entries of $C$ can be computed in $\OO(n^2)$ operations.
\begin{theorem}[\!\!\cite{Webb2021}, Theorem 4.8]
    Consider a Toeplitz-plus-finite-rank Jacobi operator $\widetilde{\mathcal{J}}$  \eqref{eq:Jpertdef1}, \eqref{eq:Jpertdef2} and a Toeplitz Jacobi operator $\mathcal{J}$ \eqref{eq:RefToeJac}. The connection coefficients matrix $C=C_{\widetilde{\mathcal{J}}\to\mathcal{J}}=(c_{i,j})_{i,j=0}^{\infty}$ satisfies
    \begin{align}\label{eq:Cstruc}
        c_{i,j} = c_{i-1,j-1} \text{ for all } i,j>0 \text{ such that }i+j\geq 2n,\quad \text{and} \quad 
        c_{0,j} = 0 \text{ for all }j\geq 2n.
    \end{align}
\end{theorem}

The connection coefficient matrix $C_{\widetilde{\mathcal{J}} \to \mathcal{J}}$ provides a way to find the poles of the Stieltjes transform associated with $\widetilde{\mathcal{J}}$ and to express its spectral measure in terms of that of $\mathcal{J}$. For a more detailed discussion, refer to \cite[Theorems 4.12, 4.14 and Remark 4.15]{Webb2021}.
\begin{theorem}
    Consider a Toeplitz-plus-finite-rank Jacobi operator  $\widetilde{\mathcal{J}}$ as defined in \eqref{eq:Jpertdef1}, \eqref{eq:Jpertdef2} and the Stieltjes transform  $s(z) = \vec e_1^*(\widetilde{\mathcal{J}} - z)^{-1} \vec e_1$. Let $\mathcal{J}$ be a Toeplitz Jacobi operator\eqref{eq:RefToeJac} and recall $\gamma_{\pm}$ in \eqref{eq:Refmeas}. Let $C$ denote the connection coefficient matrix \eqref{eq:ConnCoefCond} and consider the polynomial $p_C(\lambda)$ given by
    \begin{align*}
        p_C(\lambda) = \sum_{k=0}^{2n-1}c_{0,k}P_k(\lambda) = \sum_{k=0}^{2n-1}\abrac{C^T\vec e_k,C^T\vec e_0}Q_k(\lambda).
    \end{align*}
    The poles $\lambda_1,\dots,\lambda_r$ of $s(z)$, where $r\leq n$, are the roots of $p_C$ in $\R\setminus\{\gamma_{\pm}\}$ such that 
    \begin{align*}
        w_i = \lim_{\epsilon\downarrow 0} \frac{\epsilon}{i}s(\lambda_i+i\epsilon)\neq 0.
    \end{align*}
    Moreover, the spectral measure of $\widetilde{\mathcal{J}}$ is given by
    \begin{align*}
        \widetilde{\mu}(\mathrm{d}\lambda) = \frac{1}{p_C(\lambda)}\mu(\mathrm{d}\lambda) + \sum_{i=1}^r w_i\delta_{\lambda_i}(\lambda).
    \end{align*}
\end{theorem}

Although there is a clear connection between the poles of $s(z)$ and the roots of $p_C(\lambda)$, the condition $w_i \neq 0$ introduces challenges in both computation and analysis. However, the Joukowski map 
\begin{align}
    J(z) = \frac{1}{2}(z+z^{-1}),
\end{align}
can be used to overcome these issues. The Joukowski map is a conformal map from $\mathbb{D}=\{z\in \C : |z|<1\}$ to $\C\setminus [-1,1]$ that sends the unit circle to two copies of the interval $[-1,1]$.
\begin{theorem}[\!\!\cite{Webb2021}, Theorems 4.21 and 4.22]
    Let $\widetilde{\mathcal{J}}$ be a finite rank perturbation of the Toeplitz Jacobi operator $\mathcal{J}$ \eqref{eq:Jpertdef1}, \eqref{eq:Jpertdef2} and let $C = C_{\mathrm{Toe}}+C_{\mathrm{fin}}$ be a decomposition of $C_{\widetilde{\mathcal{J}}\to\mathcal{J}}$ \eqref{eq:CDecomp}. Consider $s(z)= \vec e_1^*(\widetilde{\mathcal{J}}-z)^{-1} \vec e_1$ and define the Toeplitz symbol of $C_{Toe}=(t_{i,j})_{i,j=0}^{\infty}$ as 
    \begin{align}
        c(z) = \sum_{i=0}^{2n-1}t_{0,i}z^i.
    \end{align}
    The poles $\lambda_1,\dots,\lambda_r$ of $s(z)$ are given by
    \begin{align*}
        \lambda_i = M_{[\gamma_{-},\gamma_{+}]}\circ J(z_i)
    \end{align*}
    where $\{z_i\}_{i=1}^r$ are the roots of $c$ that lie in $\mathbb{D}$, which are all real and simple, and $M_{[\gamma_{-},\gamma_{+}]}$ is the affine map from $[-1,1]$ to $[\gamma_{-},\gamma_{+}]$ with $\gamma_{\pm}$ defined in \eqref{eq:Refmeas}. Moreover, the spectral measure of $\widetilde{\mathcal{J}}$ is
    \begin{align*}
        \widetilde{\mu}(\mathrm{d}\lambda) = \frac{1}{p_C(\lambda)}\mu(\mathrm{d}\lambda) + \frac{\gamma_{+}-\gamma_{-}}{2}\sum_{i=1}^r \dfrac{(z_i-z_i^{-1})^2}{z_ic^{\prime}(z_i)c(z_i^{-1})}\delta_{\lambda_i}(\lambda).
    \end{align*}
\end{theorem}

It is important to note that the roots of the Toeplitz symbol $c(z)=\sum_{i=0}^{2n-1}t_{0,i}z^i$ can be efficiently computed using a companion matrix approach. Specifically, the roots coincide with the eigenvalues of the matrix
\begin{align}\label{eq:CompMat}
    \begin{bmatrix}
        & & & -\frac{t_{0,1}}{t_{0,2n-1}}\\
        1 & & & -\frac{t_{0,2}}{t_{0,2n-1}}\\
        & \ddots & & \vdots\\
        & & 1 & -\frac{t_{0,2n-2}}{t_{0,2n-1}}
    \end{bmatrix}.
\end{align}

The key ideas of this section are summarized in Algorithm~\ref{PE:ConnCoef}. As established in \cite[Theorem 6.8]{Webb2021}, this pole detection algorithm determines the exact number and location of the poles with arbitrary accuracy in a finite number of operations.

\renewcommand{\thealgorithm}{PE.\arabic{algorithm}}
\setcounter{algorithm}{0}
\begin{algorithm}
    \caption{Pole estimation via connection coefficients} \label{PE:ConnCoef}
    \begin{algorithmic}[1]
        \Statex \textbf{Input:} A Toeplitz-plus-finite-rank Jacobi operator $\widetilde{\mathcal{J}}$.
        \Statex \textbf{Output:} Poles $\lambda_1,\dots,\lambda_r$ of the Stieltjes transform associated with $\widetilde{\mathcal{J}}$.
        \State Find $\alpha,\beta$ such that 
        \begin{align*}
            \widetilde{\mathcal{J}}_{i,i}=\alpha,~ \widetilde{\mathcal{J}}_{i,i+1}=\beta, ~\text{for all } i\geq n,
        \end{align*}
        and construct $\mathcal{J}$ as in \eqref{eq:Refmeas}.
        \State Construct $C=C_{\widetilde{\mathcal{{J}}}\to\mathcal{J}}$ using \eqref{eq:5Rec}, \eqref{eq:Cstruc} and decompose it into $C = C_{\mathrm{Toe}} + C_{\mathrm{fin}}$.
        \State Build the symbol $c(z)=\sum_{i=0}^{2n-1}t_{0,i}z^i$ of $C_{\mathrm{Toe}}=(t_{i,j})$ and find its roots $\{z_i\}_{i=1}^r$ inside $\mathbb{D}$ using \eqref{eq:CompMat}. 
        \State Compute $\lambda_i = M_{[\gamma_{-},\gamma_{+}]}\circ J(z_i)$ where $i=1,\dots,r$ and $\gamma_{\pm}$ defined in \eqref{eq:Refmeas}. 
        \State \Return $\{\lambda_i\}_{i=1}^r$.
    \end{algorithmic}
\end{algorithm}

\subsection{Finite section of Jacobi operators}

Consider a Jacobi operator $\mathcal{J}$ of the form
\begin{align*}
    \mathcal{J} = \begin{bmatrix}
        \alpha_0 & \beta_0 \\
        \beta_0 & \alpha_0 & \beta_1 \\ 
        & \beta_1 & \alpha_1 & \ddots \\
        & & \ddots & \ddots
    \end{bmatrix}, \quad \beta_j\geq 0,
\end{align*}
where 
\begin{align}\label{eq:JTrunCond1}
    \alpha_j=\alpha,~\beta_j=\beta~\text{for all }j\geq n.
\end{align}
Suppose that the spectral measure associated with $\mathcal{J}$ is of the form 
\begin{align}\label{eq:JTrunCond2}
    \mu(\mathrm{d}\lambda) = h(\lambda)\sqrt{\lambda-\gamma_{-}}\sqrt{\gamma_{+}-\lambda}\, \mathds{1}_{[\gamma_-,\gamma_+]}(\lambda) + \sum_{i=1}^r w_i\delta_{\lambda_i},
\end{align}
where $\gamma_{\pm} = \alpha\pm 2\beta$ and $\lambda_1\geq \dots \geq \lambda_r > \gamma_{+}$, and further assume that it sastisfies Assumption \ref{as:measurecond}. The \textit{finite-section method} \cite{gohberg_convolution_2005,hagen_c*-algebras_2001} is a common approach to estimating the spectrum of a linear self-adjoint operator. It approximates the spectrum by truncating the associated semi-infinite matrix to an $N \times N$ principal submatrix and using its eigenvalues for sufficiently large $N$. Since the poles of the Stieltjes transform associated with $\mathcal{J}$ correspond to the discrete eigenvalues $\{\lambda_i\}_{i=1}^r$, the finite-section method, as described in Algorithm~\ref{PE:Trunc} can be applied to estimate the number and locations of these poles.

Although this method is straightforward, it often has a significant drawback, known as \textit{spectral pollution}. This occurs when the eigenvalues of the truncated operators (or a subsequence) converge to values that are not part of the true spectrum of the original operator. However, under the current assumptions, it can be shown that the finite-section method does not suffer from this issue and accurately captures the correct number of discrete eigenvalues when $N$ is sufficiently large.  This can be established, for example,  using the asymptotics of the orthogonal polynomials associated with $\mathcal{J}$, as outlined in \cite{Ding2021b}. The convergence of individual eigenvalues occurs at an exponential rate as $N$ increases.  The rate does potentially degenerate as gaps in the spectrum of $\mathcal J$ close.  A more detailed analysis of this observation is left for future work.

\begin{algorithm}[tbp]
    \caption{Pole estimation via finite section} \label{PE:Trunc}
    \begin{algorithmic}[1]
        \Statex \textbf{Input:} A Jacobi operator $\mathcal{J}$ satisfying \eqref{eq:JTrunCond1} and \eqref{eq:JTrunCond2} and a threshold $\gamma$.
        \Statex \textbf{Output:} Approximations of the poles $\lambda_1,\dots,\lambda_r$ of the Stieltjes transform associated with $\mathcal{J}$.
        \State Set $J_\ell = \mathcal{J}_{1:\ell,1:\ell}$ where $\ell$ is sufficiently large.
        \State \Return Return the top eigenvalues $\{\lambda_i\}_{i=1}^{\hat r}$ of $J_\ell$ that are larger than $\gamma$.
    \end{algorithmic}
\end{algorithm}

\bibliographystyle{abbrv}
\bibliography{library}

\end{document}